\documentclass[11pt,twoside, leqno]{article}

\usepackage{amssymb}
\usepackage{amsmath}
\usepackage{mathrsfs}
\usepackage{amsthm}
\usepackage{amsfonts}
\usepackage{txfonts}
\usepackage{enumerate}
\usepackage{hyperref}
\usepackage{esint}
\usepackage{latexsym,amssymb}
\usepackage{color}
\usepackage{marginnote}
\usepackage{multirow}
\usepackage{makecell} 
\usepackage{float} 

\usepackage{indentfirst}
\usepackage{graphicx}
\usepackage{pgf,tikz}
\usepackage{mathrsfs}
\usetikzlibrary{arrows}

\definecolor{ttttff}{rgb}{0.2,0.2,1.}
\definecolor{ttffcc}{rgb}{0.2,1.,0.8}
\definecolor{qqqqff}{rgb}{0.,0.,1.}

\definecolor{zzttqq}{rgb}{0.6,0.2,0.}
\definecolor{qqqqff}{rgb}{0.,0.,1.}

\allowdisplaybreaks

\hypersetup{
        colorlinks   = true,
        citecolor    = blue,
        linkcolor    = blue,
        urlcolor     = blue
}

\def\rr{{\mathbb R}}
\def\rn{{\mathbb{R}^n}}

\def\cc{{\mathbb C}}

\def\nn{{\mathbb N}}

\def\cj{{\mathcal J}}

\def\cl{{\mathcal L}}
\def\cm {{\mathcal M}}

\def\fz{\infty }

\def\lf{\left}
\def\r{\right}

\def\ls{\lesssim}
\def\gs{\gtrsim}

\def\noz{\nonumber}

\def\gfz{\genfrac{}{}{0pt}{}}

\def\loc{{\mathop\mathrm{\,loc\,}}}
\def\supp{\mathop\mathrm{\,supp\,}}

\newtheorem{theorem}{Theorem}[section]
\newtheorem{lemma}[theorem]{Lemma}
\newtheorem{corollary}[theorem]{Corollary}
\newtheorem{proposition}[theorem]{Proposition}
\newtheorem{example}[theorem]{Example}

\theoremstyle{definition}
\newtheorem{remark}[theorem]{Remark}
\newtheorem{definition}[theorem]{Definition}
\renewcommand{\appendix}{\par
   \setcounter{section}{0}%
   \setcounter{subsection}{0}%
   \setcounter{subsubsection}{0}%
   \gdef\thesection{\@Alph\c@section}%
   \gdef\thesubsection{\@Alph\c@section.\@arabic\c@subsection}%
   \gdef\theHsection{\@Alph\c@section.}%
   \gdef\theHsubsection{\@Alph\c@section.\@arabic\c@subsection}%
   \csname appendixmore\endcsname
 }

\numberwithin{equation}{section}

\definecolor{qqqqff}{rgb}{0,0,1}
\definecolor{ffqqqq}{rgb}{1,0,0}

\newcounter{rea}
\def\vd{\hyperref[e:VD]{\mathrm{(VD)}}}         

\def\SG{\hyperref[eq-LY]{\mathrm{(LY)}}}

\def\esup{\mathop{\mathrm{esssup}}}
\def\einf{\mathop{\mathrm{essinf}}}

\begin{document}

\arraycolsep=1pt

\title{\vspace{-1.5cm}
Mixed-Parabolicity and Mixed-Liouville Property \\
for
Products of Riemannian Manifolds
\footnotetext{\hspace{-0.35cm}
2020 {\it Mathematics Subject Classification}.  35K08;  31A15; 42B35.
\endgraf {\it Key words and phrases}.
Parabolicity, Liouville property, Green function, nonlinear capacity, mixed-norm Lebesgue space.
\endgraf
This project is supported by the National Natural Science Foundation of China
(\# 12371102, \# 12371206, \# 12201098, \# 12526436), the Fundamental Research Funds for the Central Universities (\# 050-63263078),
and
 the Natural Science Foundation of Liaoning Province (\# 2025-BS-0228).
}}
\author{
Liguang Liu, Yuhua Sun, and Suqing Wu}
\date{}
\maketitle


\begin{abstract}
Let $p_1,p_2\in(1,\infty)$ and $M=M_1\times M_2$ be the product of two geodesically complete Riemannian manifolds. In this paper, the authors first develop an
anisotropic potential-theoretic framework adapted to the Green operator $G^M$ and the mixed-norm Lebesgue space $L^{p_2}(L^{p_1})(M)$, and then demonstrate that the classical equivalence among \emph{parabolicity}, \emph{Green function integrability}, and \emph{Liouville property} persists in this genuinely anisotropic setting.

 More precisely, the authors establish the following  equivalence: $M$ is $L^{p_2}(L^{p_1})$-parabolic if and only if the Green function $G^M(x;\,\cdot\,)$ fails to belong to $L^{p_2'}(L^{p_1'})(M \setminus B(x,\,r))$, which is in turn equivalent to the $L^{p_2'}(L^{p_1'})$-Liouville property, where $p_i'$ denotes the conjugate exponent of $p_i$.  Under a weak radial Harnack-type inequality---in particular, under Li--Yau heat kernel estimates, and hence for products of manifolds with nonnegative Ricci curvature---these conditions are further equivalent to the divergence of the nonlinear mixed-potential $\mathcal{G}_{p_1,p_2}(f)$ for every nonzero nonnegative $f\in {\mathcal C}_c^\infty(M)$.

A key feature of this anisotropic theory is its sensitivity to the geometry of each factor \(M_i\), rather than merely to that of the total manifold \(M\). In contrast to the isotropic case, where parabolicity and the classical Liouville property holds on \(\mathbb{R}^n\) precisely when \(n \le 2\), the anisotropic setting exhibits a refined threshold: the \(L^{p_2}(L^{p_1})\)-parabolicity and the \(L^{p_2'}(L^{p_1'})\)-Liouville property holds on \(\mathbb{R}^{n_1} \times \mathbb{R}^{n_2}\) if and only if
$
D_{\mathrm{eff}} := \frac{n_1}{p_1} + \frac{n_2}{p_2} \le 2.
$
This effective dimension $D_{\mathrm{eff}}$ captures the anisotropic interplay between the exponents \(p_1, p_2\) and the geometries of \(M_1, M_2\).
\end{abstract}

\tableofcontents

\section{Introduction}

\subsection{Parabolicity and Liouville property on manifolds}\label{ss1.1}
Let \((M, g)\) be a noncompact, geodesically complete Riemannian manifold. Denote by \(d(\cdot,\, \cdot)\) the geodesic distance induced by \(g\) and by \(V\) the Riemannian measure (volume) induced by \({g}\). Then \((M, d, V)\) is a noncompact, geodesically complete metric measure space.
The \emph{Laplace--Beltrami operator} $\varDelta_M $ on $M$ is defined as follows:
$$\varDelta_M =\frac1{\sqrt{\det {g}}}\sum_{i,j} \partial_{x_i} \left( \sqrt{\det {g}}\, { g}^{ij}\partial_{x_j}\right),$$
where $g=({ g}_{ij})$ denotes the Riemannian metric tensor and $({ g}^{ij})=({ g}_{ij})^{-1}$. In this paper, we will use the notations $(M, g)$ or $(M,  d, V)$, or $(M, g, d,V)$, depending on the context.

For a function \(u\) in the local Sobolev space \( W^{1,2}_\loc(M)\), we say that \(u\) is \emph{superharmonic} on $M$ if for all nonnegative \(\phi \in \mathcal{C}_c^\infty(M)\) (the space of infinitely differentiable functions with compact support),
\[
\int_M \nabla u \cdot \nabla \phi \, dV\ge 0.
\]
Further, \(u\) is called \emph{subharmonic} if \(-u\) is superharmonic. If \(u\) is both superharmonic and subharmonic, then we say that \(u\) is \emph{harmonic} on \(M\).

 The notion of  \emph{parabolic manifold} originates from the uniformization theorem of Koebe--Poincar\'e. This theorem states that every simply connected Riemann surface $S$ is conformally equivalent to one of the following three canonical surfaces: the compact surface $\mathbb S^2$, the noncompact Euclidean space $\mathbb R^2$, or the noncompact hyperbolic plane $\mathbb H^2$.
Since superharmonicity is preserved under conformal transformations, a crucial distinction between the noncompact cases $\mathbb{R}^2$ and $\mathbb{H}^2$ lies in the behavior of positive superharmonic functions. On the Euclidean plane $\mathbb{R}^2$, every positive superharmonic function is necessarily constant. In contrast, the hyperbolic plane $\mathbb{H}^2$ admits a wealth of non-constant positive superharmonic functions.
This distinction leads to the following definition.

\begin{definition}\label{def-para}
  A Riemannian manifold $M$ is said to be \emph{parabolic} if every positive superharmonic function on $M$ is constant, and \emph{non-parabolic} otherwise.

\end{definition}

The characterization of parabolicity has been studied extensively by many authors. As the primary motivation for this paper, we list here the following equivalent conditions, which are taken from \cite[Theorem~5.1]{Grigoryan1999BAMS}.

\begin{theorem}[\cite{Grigoryan1999BAMS}]\label{thm-G-para}
The following properties are equivalent to each other:
\begin{enumerate}[\rm (i)]
  \item $M$ is parabolic.
  \item Every positive bounded superharmonic function on $M$ is constant.
  \item The positive Green function  does not exist on $M$.
  \item For any compact set $K\subset M$, the capacity
  $$\mathrm{cap}(K):=\inf\left\{\|\nabla f\|_{L^2(M)}^2:\  f\in {\mathcal C}_c^\infty(M),\, f=1\ \text{on}\ K\right\}=0.$$
\end{enumerate}
\end{theorem}

An interesting probabilistic characterization of parabolicity is that $M$ is parabolic if and only if Brownian motion on $M$ is recurrent (see also \cite{Grigoryan1999BAMS}). We shall not, however, pursue this direction further here.

An elegant way to determine whether a manifold is parabolic is to examine its volume growth at infinity.
For convenience,  denote the open geodesic \emph{ball} centered at \(x\in M\) with radius \(r\in(0,\infty)\) by
\[
B(x,r) := \{ z \in M :\ d(z,x) < r \},
\]
and set
\[
V(x,r):=V\bigl(B(x,r)\bigr).
\]
Cheng and Yau \cite{ChengYau1975CPAM} proved that a noncompact, geodesically complete Riemannian manifold \(M\) is parabolic if, for some point \(x \in M\) and some positive constant \(C\),
\begin{align*}
V(x,r) \le C r^2 \quad \text{as }\ r \to \infty.
\end{align*}
A sharp sufficient condition for parabolicity was later obtained by Grigor'yan \cite{Grigoryan1985MSb}, Karp \cite{Karp1982AMS}, Varopoulos  \cite{Varopoulos1983WMS} independently, who showed that
\begin{align}\label{eq-parabolic}
\int^\infty \frac{r\,dr}{V(x,r)} = \infty \quad \Rightarrow \quad M \ \, \text{is parabolic}.
\end{align}
The integral condition in \eqref{eq-parabolic} is necessary provided that \((M,  g)\) has nonnegative Ricci curvature. Indeed, if \((M,  g)\) has nonnegative Ricci curvature, Li and Yau \cite{LiYau1986} proved that the heat kernel, denoted by \(\{p_t^M\}_{t\in(0,\infty)}\), exists on \(M\) and satisfies the following two-sided \emph{Gaussian estimate} $\SG$: for all \(t\in (0,\infty)\) and \(x,y \in M\),
\begin{align}\label{eq-LY}
p_{t}^M(x,y) \asymp \frac{C}{\sqrt{V(x,\sqrt{t})\,V(y,\sqrt{t})}} \exp\left( -c \frac{d(x,y)^2}{t}\right), \tag*{\hyperref[eq-LY]{$(\mathrm{LY})$}}
\end{align}
where the notation \(\asymp\) means that both \(\leq\) and \(\geq\) are valid, but with possibly different values of the positive constants \(C\) and \(c\) on each side. It was respectively proved by Grigor'yan \cite{Grigoryan1991MSb} and  Saloff-Coste \cite{Saloff-Coste1992IMRN} that the Li--Yau estimate $\SG$ is equivalent to  the volume doubling condition $\vd$ and the known Poincar\'e inequality. For a further generalization
of this equivalence to the setting of general metric measure spaces, we refer the reader to
Grigor'yan--Hu--Lau \cite{GrigoryanHuLau2015JMSP}.
Under  $\SG$, one has
\begin{align}\label{eq-parabolic-iff}
\int^\infty \frac{r\,dr}{V(x,r)} = \infty \quad \Leftrightarrow \quad \text{Green function does not exist}.
\end{align}
In view of Theorem~\ref{thm-G-para}, under $\SG$, the integral condition in \eqref{eq-parabolic-iff} is both necessary and sufficient for $M$ to be parabolic, as well as for the classical Liouville property.

\subsection{$L^p$-parabolicity and $L^{p'}$-Liouville property on manifolds}\label{ss1.2}

Let \((M,  g,d,V)\) be a noncompact geodesically complete Riemannian manifold.
Suppose that  $\{p_t^M\}_{t\in(0,\infty)}$ is the \emph{heat kernel} on \(M\) generated by
the Laplace--Beltrami operator $\varDelta_M$.
Then
$p_t^M(x,y)$
 is ${\mathcal C}^\infty$-smooth jointly in $(t,x,y)\in (0,\infty)\times M\times M$ (see \cite[Theorem~7.20]{Grigoryan2009book}).
Associated to  $\{p_t^M\}_{t\in(0,\infty)}$, we naturally associate the
\emph{heat semigroup} \(\{P_t^M\}_{t\in(0,\infty)}\), defined for all \(f \in L^2(M)\) and \(t \in( 0,\infty)\), by
\begin{equation*}
P_{t}^Mf(x)= e^{t \varDelta_M} f(x) =\int_{M}p_{t}^M(x,y) f(y)\, dV(y)\quad \text{for all}\ \, x\in M.
\end{equation*}
The \emph{Green function} \(G^M(x,\, y)\) on \(M\) can be defined by
\begin{equation}\label{eq:Green-kernel}
G^M(x,\, y) = \int_0^\infty p_t^{M}(x, y) \, dt,
\end{equation}
where the integral may diverge to \(+\infty\). By a slight abuse of notation, we
still use \(G^M\) to denote the \emph{Green operator} on \(M\), which acts on nonnegative measurable functions \(f\) on \(M\) by
\begin{equation}\label{eq:Green-operator}
G^M f(x) = \int_M G^M(x,\, y) \, f(y) \, dV(y)=  \int_0^\infty P_t^{M}f(x) \, dt \quad \text{for all}\ \,x \in M.
\end{equation}
For a general measurable function \(f\) on \(M\), via writing \(f = f^+ - f^-\) with \(f^+ := \max\{f, 0\}\) and \(f^- := -\min\{f, 0\}\), we define
\[
G^M f := G^M f^+ - G^M f^-,
\]
whenever the right-hand side is well-defined (i.e., not of the form \(\infty - \infty\)).
For a systematic study of Green functions on manifolds, we refer the reader to Grigor'yan's monograph \cite{Grigoryan2009book}.

Faraji and Grigor'yan \cite{FarajiGrigoryan2019RMI} introduced
the notion of biparabolicity for Riemannian manifolds. A function \(u \in \mathcal{C}^4(M)\) is said to be \emph{bi-superharmonic} if it satisfies
\[
-\varDelta_M u \ge 0 \quad \text{and} \quad \varDelta_M^2 u \ge 0.
\]
A Riemannian manifold \(M\) is called \emph{bi-parabolic} if every positive bi-superharmonic function on \(M\) is harmonic -- a property that can be viewed as a biharmonic analogue of the Liouville property. Faraji and Grigor'yan \cite[Theorem~3.1]{FarajiGrigoryan2019RMI} proved that \(M\) is biparabolic if and only if
\[
G^M(G^M f) \equiv \infty \quad \text{for every nonzero } \ 0\le f \in {\mathcal C}_c^\infty(M).
\]
Their proof makes essential use of  properties of  Green functions.

Building upon further developments in Green potential theory on manifolds, Grigor'yan, Pessoa, and Setti \cite{GrigoryanPessoaSetti2025} extended Theorem~\ref{thm-G-para} and established the equivalence between \(L^p\)-parabolicity and the \(L^{p'}\)-Liouville property (see Definition \ref{def-Lp-para}
and \ref{def-Lq-Liou}). Here \(p'\) denotes the H\"{o}lder conjugate exponent of \(p\), i.e., \(\frac{1}{p} + \frac{1}{p'} = 1\), with the convention that \(p' = \infty\) when \(p = 1\).

\begin{definition}\label{def-Lp-para}
Let \(p \in (1,\infty)\). For an arbitrary set \(E \subset M\), its \emph{Riesz capacity} \(\operatorname{cap}_p(E)\) is defined by
\[
\operatorname{cap}_p(E) := \inf\left\{ \| f \|_{L^p(M)}^{p}:\ \ 0 \le f \in L^{p}(M),\; G^{M} f \ge \mathbf{1}_{E} \right\}.
\]
A Riemannian manifold \(M\) is said to be \emph{\(L^p\)-parabolic} if for every precompact open set \(E \subset M\),
\[
\operatorname{cap}_p(\overline{E}) = 0.
\]
\end{definition}

\begin{definition}\label{def-Lq-Liou}
Let \(q \in (1,\infty)\). A Riemannian manifold \(M\) is said to possess the \emph{\(L^q\)-Liouville property} if every nonnegative superharmonic function \(u\) on \(M\) with \(u \in L^q(M)\) is constant.
\end{definition}

\begin{theorem}[\cite{GrigoryanPessoaSetti2025}]\label{thm-GPS-para}
Let $p\in(1,\infty)$. The following properties are equivalent to each other:
\begin{enumerate}[\rm (i)]
  \item $M$ is $L^p$-parabolic.
  \item $G^M\big((G^Mf)^{p'-1}\big)(x)=\infty$ for some / all nonzero nonnegative $f\in {\mathcal C}_c^\infty(M)$ and some / all $x\in M$.
  \item $G^M(x,\cdot)\notin L^{p'}(M\setminus B(x,\,r))$  for some / all $x\in M$
  and some / all $r\in (0,\infty)$.
  \item $M$ possesses the \emph{\(L^{p'}\)-Liouville property}.
\end{enumerate}
\end{theorem}

For any \(p \in [1,\infty)\) and any compact set \(K \subset M\), define
\begin{align*}
\operatorname{cap}_{p}^{\varDelta}(K) := \inf\left\{ \| \varDelta_M f \|_{L^p(M)}^{p} :\ \, f \in \mathcal{C}_c^{\infty}(M),\; f \ge 1 \text{ on } K \right\}.
\end{align*}
For \(p \in (1,\infty)\), this provides an equivalent definition of \(\operatorname{cap}_p(K)\) since (see \cite[Theorem~2.15]{GrigoryanPessoaSetti2025})
\[
\operatorname{cap}_{p}^{\varDelta}(K) = \operatorname{cap}_{p}(K).
\]
For the endpoint case \(p = 1\) and the case \(M = \mathbb{R}^n\), one has (see \cite[Appendix~4E]{BrezisMarcusPonce2007AMS})
\[
\operatorname{cap}_{1}^{\varDelta}(K) = 2\operatorname{cap}(K),
\]
where we recall that \(\operatorname{cap}\) denotes the standard capacity appearing in Theorem~\ref{thm-G-para}. In light of these observations, it is clear that Theorem~\ref{thm-G-para} can be viewed as the endpoint case \(p = 1\) of Theorem~\ref{thm-GPS-para}.

According to \cite[Proposition~5.3]{GrigoryanPessoaSetti2025}, if $M$ satisfies $\SG$, then each condition in Theorem~\ref{thm-GPS-para} is equivalent to the following \emph{volume growth criterion}
\begin{align}\label{eq-Lppara-iff-vol}
\int^\infty \left(\int_r^\infty \frac{t}{V(t)}\,dt\right)^{p'} V'(t)\,dr=\infty,
\end{align}
where $V(t)$ denotes the volume of the ball  $B(o,r)$ for some fixed point $o\in M$.

\subsection{Mixed-parabolicity and mixed-Liouville property on product manifolds}\label{ss1.3}

For each $i \in \{1,2\}$, let $(M_i, g_i, d_i, V_i)$ be a Riemannian manifold, where $g_i$, $d_i$, and $V_i$ denote the Riemannian metric, the corresponding Riemannian distance, and the Riemannian measure on $M_i$, respectively.
Throughout this paper,
the Riemannian product manifold $M := M_1 \times M_2$ is always equipped  with the natural product Riemannian metric $g = g_1 \oplus g_2$. We shall accordingly refer to the   \emph{product manifold} $(M, d, V)$, since the product Riemannian metric $g$  canonically induces the corresponding Riemannian distance $d$ and Riemannian measure $V$ on $M= M_1 \times M_2$. Further details are given in Section~\ref{ss2.1}.

The present work is the anisotropic counterpart of  Theorems~\ref{thm-G-para} and~\ref{thm-GPS-para}.
Our main aim is to explore the role of each individual factor
$M_i$ ​
  plays in the equivalence between \emph{parabolicity} and the \emph{Liouville property}. This investigation is particularly natural, since for two geodesically complete Riemannian manifolds \(M_1\) and \(M_2\), it is known that $$
M_1\ \text{and}\ M_2 \ \text{are parabolic}\ \
\nRightarrow \ \ M_1 \times M_2\ \text{is parabolic}.
  $$
A typical example is \(M_1 = M_2 = \mathbb{R}^2\). It is therefore worthwhile to study these properties on product manifolds \(M = M_1 \times M_2\).

This line of inquiry is closely  related to {\bf Problems 15 and 16} posed by Grigor'yan \cite{Grigoryan1999BAMS}, which mainly concern the stochastic completeness of $M_1 \times M_2$ and the existence of bounded harmonic functions when both $M_1$ and $M_2$ are stochastically complete or parabolic. We also note that earlier investigations of Liouville-type properties for the Schr\"odinger equation $\varDelta u - c(x)u = 0$ on product manifolds of the form $M_1 \times K$, with $K$ compact, can be found in the work of Losev \cite{Losev2012MN}.

The fundamental tools used in this paper are heat kernels and Green functions.
For $i\in\{1,2\}$, denote by $\{p_t^{M_i}\}_{t\in(0,\infty)}$ the \emph{heat kernel} of the
manifold $(M_i, { g}_i)$.
According to \cite[p.\,213, Exercise~7.14]{Grigoryan2009book},  the heat kernel $\{p_t^{M}\}_{t\in(0,\infty)}$ on
the Riemannian product manifold $M= M_1 \times M_2$ exists and satisfies that, for all $t\in(0,\infty)$, $x=(x_1,x_2)\in M$ and $y=(y_1,y_2)\in M$,
\begin{align}\label{eq-ptptMi}
p_t^{M}(x,y)=p_t^{M_1}(x_1,y_1)\,p_t^{M_2}(x_2,y_2).
\end{align}
Indeed, this factorization property of the heat kernel \(\{p_t^{M}\}_{t\in(0,\infty)}\) is a consequence of the fact that
the corresponding Laplace--Beltrami operator \(\varDelta_M \) on the Riemannian product manifold
\(M = M_1 \times M_2\) splits as a sum:
\begin{align}\label{eq-DeltaM}
\varDelta_M  = \varDelta_{M_1} + \varDelta_{M_2},
\end{align}
where \(\varDelta_{M_i}\) denotes the Laplace--Beltrami operator on \(M_i\).

With the product heat kernel \(\{p_t^{M}\}_{t\in(0,\infty)}\) given in \eqref{eq-ptptMi}, we can then define the corresponding Green function \(G^M(x,\,y)\) and Green operator \(G^M\) via \eqref{eq:Green-kernel} and \eqref{eq:Green-operator}.
Since \(M = M_1 \times M_2\), by a point \(x = (x_1, x_2) \in M\) we mean that \(x_1 \in M_1\) and \(x_2 \in M_2\). For any \(x = (x_1, x_2) \in M\) and \(y = (y_1, y_2) \in M\), to avoid ambiguity in the notation \( G^M(x, y) \), we adopt the clearer convention
\[
 G^M(x;\, y) = G^M(x_1, x_2;\, y_1, y_2).
\]
 For any  \(f \in L_{\mathrm{loc}}^1(M)\) and \(x = (x_1, x_2) \in M\), there is
\begin{align*}
G^M f(x)
&= \int_M G^M(x;\, y) f(y) \, dV(y)\\
&= \int_{M_2} \int_{M_1} G^M(x_1, x_2;\, y_1, y_2) f(y_1, y_2) \, dV_{1}(y_1) \, dV_{2}(y_2).
\end{align*}
Analogously, if $\mu\in \cm^+(M)$ (i.e., the set of all nonnegative Radon measures
 on $M$), then
 \[
G^M \mu(x) = \int_M G^M(x;\, y) \, d\mu(y).
\]
Hereinafter, we understand \(p_t^M(x,y)=p_t^M(x;\, y)\), \(p_t^Mf(x)\) and \(p_t^M\mu(x)\) in a similar manner. Further properties of Green functions that will be used in this paper are given in Section~\ref{ss2.2}.

As for the main idea of this paper, just as classical \(L^p\)-capacity theory serves as the foundation for the isotropic results (Theorems~\ref{thm-G-para} and~\ref{thm-GPS-para}), we shall employ an \emph{anisotropic} capacity theory on the product manifold \(M = M_1 \times M_2\) built on \emph{mixed-norm Lebesgue spaces}. We begin by introducing these spaces on \(M = M_1 \times M_2\), which were originally studied by Benedek and Panzone \cite{BenedekPanzone1961Duke}.

\begin{definition}\label{def-LupLvq}
Let \(M = M_1 \times M_2\).
For any  $p_1,\ p_2\in(0,\fz)$,
define the {\it  mixed-norm
Lebesgue space} $L^{p_2}(L^{p_1})(M)$  to be the set of all measurable functions $f$ on $M$ such that
\begin{align*}
	\|f\|_{L^{p_2}(L^{p_1})(M)}
	:= \lf(\int_{M_2}\lf(\int_{M_1}|f(x_1,x_2)|^{p_1} \,dV_{1}(x_1)\r)
^{\frac {p_2}{p_1} }\,dV_{2}(x_2)\r)^{\frac 1 {p_2}}
	< \infty.
\end{align*}
Similarly, define $L^{\infty}(L^{p_1})(M)$ and $L^{p_2}(L^{\infty})(M)$ in terms of the (quasi-)norms
\begin{align*}
	\|f\|_{L^{\infty}(L^{p_1})(M)}
	:= \esup_{x_2\in M_2}\lf(\int_{M_1}|f(x_1,x_2)|^{p_1} \,dV_{1}(x_1)\r)
^{\frac 1 {p_1} }
\end{align*}
and
\begin{align*}
	\|f\|_{L^{p_2}(L^{\infty})(M)}
	:= \lf(\int_{M_2}\lf(\esup_{x_1\in M_1}|f(x_1,x_2)|^{p_2}\r) \,dV_{2}(x_2)\r)
^{\frac 1 {p_2} }.
\end{align*}
respectively. In the case $p_1=p_2=p$, we will simply write $L^{p_2}(L^{p_1})(M)$
as $L^p(M)$.
\end{definition}

Further duality and density results for $L^{p_2}(L^{p_1})(M)$ are discussed in Section \ref{ss3.1} below.
In fact, mixed-norm spaces have attracted significant attention in harmonic analysis and PDEs. For instance, the well-known Strichartz estimate in PDEs is a space-time mixed-norm estimate that describes the regularity and decay of solutions to dispersive PDEs, such as the wave equation and Schr\"odinger equation (see, for example, Keel--Tao \cite{KeelTao1998AJM}), since neither a pure space norm nor a pure time norm can capture both the propagation and diffusion effects simultaneously. Recent developments on mixed-norm function spaces and the boundedness of various operators can be found in \cite{Kurtz2007RMJM, CleanthousGeorgiadisNielsen2017JGA, HuangLiuYangYuan2019PAMS, ChenSun2022JGA} and references therein.

In the following two definitions, we introduce the three main concepts of this paper: \emph{mixed-capacity}, \emph{mixed-parabolicity}, and the \emph{mixed-Liouville property}. Each of these notions is associated with the mixed-norm Lebesgue space \(L^{p_2}(L^{p_1})(M)\).

\begin{definition}\label{def-Rpquv}
Let \(M = M_1 \times M_2\) and \(p_1, p_2 \in (1,\infty)\).
For any open set \(\Omega \subset M\), and any set \(E \subset \Omega\), define the \emph{mixed-capacity} by
\begin{align}\label{eq-cap}
\operatorname{cap}_{p_1,\,p_2}(E;\,\Omega) := \inf\Bigl\{ \| f \|_{L^{p_2}(L^{p_1})(M)}^{p_2} :\ 0 \le f \in L^{p_2}(L^{p_1})(M),\; G^{\Omega} f \ge \mathbf{1}_E \Bigr\}.
\end{align}
In the special case \(\Omega = M\), we write \(\operatorname{cap}_{p_1,\,p_2}(E;\,M)\) simply as \(\operatorname{cap}_{p_1,\,p_2}(E)\).
\end{definition}

\begin{definition}\label{def-product-para-Liouville}
Let \(p_1,p_2 \in (1,\infty)\).
A product Riemannian manifold \(M= M_1 \times M_2\) is said to be \emph{\(L^{p_2}(L^{p_1})\)-parabolic} if for every precompact open set \(E \subset M\),
\[
\operatorname{cap}_{p_1,\,p_2}(\overline{E}) = 0.
\]
 A product Riemannian manifold \(M= M_1 \times M_2\) is said to admit the \emph{\(L^{p_2}(L^{p_1})\)-Liouville property} if every nonconstant nonnegative superharmonic function \(u\) on \(M\) with \(u \in L^{p_2}(L^{p_1})(M)\) is constant.
\end{definition}

\begin{remark}
In Definitions~\ref{def-Rpquv} and~\ref{def-product-para-Liouville}, one is allowed to define the corresponding capacities, parabolicity, and Liouville property for the endpoint cases $p_1, p_2 = 1$ or $\infty$. However, in this paper, we focus only on the case when both $p_1, p_2 \in (1, \infty)$.
\end{remark}

\begin{remark}
If $M = \mathbb{R}^n \times \mathbb{R}^n$ with $n \ge 2$, then the Green operator $G^M$ reduces to the classical \emph{Riesz potential operator} $\mathcal{I}_2^{(2n)}$, given by
\[
\mathcal{I}_2^{(2n)} f(x) = \frac{(n-2)!}{4\pi^{n}} \int_{\mathbb{R}^{2n}} |x-y|^{2-2n} f(y) \, dy.
\]
In this setting, the following \emph{mixed-Riesz capacity} was introduced in \cite{JinLiuWuXiao2026JGA}: for any set $E \subset \mathbb{R}^{2n}$,
\[
\mathcal{R}_{p_1,\, p_2}(E) := \inf \left\{ \| f \|_{L^{p_2}(L^{p_1})(\rn\times\rn)}^{p_1}:\ \ f\ge0, \ \mathcal{I}_2^{(2n)} f \ge \mathbf{1}_E \right\}.
\]
Moreover, the  capacitary inequalities for such mixed-Riesz capacity were established in \cite{JinLiuWuXiao2026JGA}. It is clear that
\[
\operatorname{cap}_{p_1,\, p_2}(E;\ \rn\times\rn) = \mathcal{R}_{p_1, p_2}(E)^{\frac{p_2}{p_1}}.
\]
 As a consequence of Theorem \ref{thm-main} and Proposition~\ref{prop-para-mix-Rn} below, we observe that
$$\mathcal{R}_{p_1,\, p_2}\equiv 0\quad \ \Leftrightarrow\quad \ \frac{n}{p_1}+\frac{n}{p_2}\le 2.$$
\end{remark}

Basic measure-theoretic properties of the mixed-capacity \(\operatorname{cap}_{p_1, p_2}\) are established in Section~\ref{sec-cap}, and several equivalent characterizations are given in Section~\ref{sec-cap-equiv}. Moreover, for any compact set \(K \subset \Omega\), we show in Section~\ref{sec-mix-cap-poten} that there exists a Radon measure \(\mu^K\) such that
\[
\mu^K(K) = \operatorname{cap}_{p_1, p_2}(K; \Omega).
\]
Such a measure is usually called a \emph{capacitary measure}. In fact, the quantity \(\mu^K(K)\) equals the integral over \(M\) of the \emph{nonlinear mixed-potential} \(\mathcal{G}_{p_1, p_2}^\Omega(\mu^K)\) with respect to \(\mu^K\) (see Theorem~\ref{thm-equiv-mu} below).

\begin{definition}\label{def-mix-poten}
Let \(M = M_1 \times M_2\) and \(p_1, p_2 \in (1,\infty)\). Suppose that $\Omega\subset M$ is an open set.
For any  nonnegative Radon measure $\mu$ on $M$,
and any $x=(x_1,x_2)\in M$ with $x_1\in M_1$ and $x_2\in M_2$, define
the \emph{nonlinear mixed-potential}
  \begin{align}\label{eq-GGf}
  {\mathcal G}_{p_1,\, p_2}^\Omega(\mu)(x):= \int_{M_2}  \int_{M_1}\|G^\Omega (\mu)(\cdot,\, y_2)\|_{L^{p_1'}(M_1)}^{p_2'-p_1'} G^\Omega(x;\, y)
  \left[G^\Omega (\mu)(y)\right]^{p_1'-1}\,dV_{1}(y_1)\,dV_{2}(y_2).
\end{align}
If $d\mu=f\,dV$ for some $f\in L_\loc^1(M)$,
then we  denote ${\mathcal G}_{p_1,\, p_2}^\Omega(\mu)$ by ${\mathcal G}_{p_1,\, p_2}^\Omega(f)$.
In the case $\Omega=M$, we omit the superscript and simply write ${\mathcal G}_{p_1,\, p_2}(\mu)$ or ${\mathcal G}_{p_1,\, p_2}(f)$.
\end{definition}

\begin{remark}
If \(\mu \equiv 0\), then we adopt the convention that
$ {\mathcal G}_{p_1,\, p_2}^\Omega(\mu)=0.$
If $G^\Omega (\mu)= \infty$ on a subset of $M$ with positive measure, then we set
$ {\mathcal G}_{p_1,\, p_2}^\Omega(\mu)=\infty.$
It may happen that \(G^\Omega(\mu) < \infty\) a.e. on \(M\), yet the set
\[
\left\{ y_2 \in M_2:\ \|G^\Omega(\mu)(\cdot, y_2)\|_{L^{p_1'}(M_1)} = \infty \right\}
\]
has positive measure in \(M_2\), in which case we define
$$ {\mathcal G}_{p_1,\, p_2}^\Omega(\mu)=\begin{cases}\infty\quad \, &\text{as}\ \, p_2<p_1;\\
G^\Omega([G^\Omega (\mu)]^{p_1'-1})(x)\quad\, &\text{as}\ \, p_2=p_1;\\
0\quad\, &\text{as}\  \, p_2>p_1.
\end{cases}$$
\end{remark}

\begin{remark}
Note that, if $p_1=p_2=p$, then
\begin{align*}
  {\mathcal G}_{p,\, p}(\mu)(x)
  &= \int_{M_2}  \int_{M_1} G^M(x;\, y)
  \left[G^M (\mu)(y)\right]^{p'-1}\,dV_{1}(y_1)\,dV_{2}(y_2)\\
  &  =G^M\left(\left[G^M (\mu)\right]^{p'-1}\right)(x),
\end{align*}
which coincides with the nonlinear potential used in Theorem \ref{thm-GPS-para}(ii).
Observe that the potential function
$$u={\mathcal G}_{p,\, p}(f)$$
solves the equation
\begin{align*}
\varDelta_M\left( |\varDelta_M u|^{p-2}  \varDelta_M u\right)=f.
\end{align*}
In the special case $p_1=p_2=2$, the function $G^M(G^Mf)$ is a solution to the bi-Laplace equation
$$\varDelta_M^2 u=f.$$
According to Proposition \ref{prop-PDE} below, for any $0\le f\in {\mathcal C}_c^\infty(M)$,
the following mixed equation
\begin{align*}
\varDelta_M\left( |\varDelta_M u|^{p_1-2}  \varDelta_M u \left\|\varDelta_M u(\cdot, x_2)\right\|_{L^{p_1}(M_1)}^{p_2-p_1}\right)=f
\end{align*}
has a solution that can be expressed using the nonlinear mixed-potential function
$$u={\mathcal G}_{p_1,\, p_2}^\Omega(f),$$
where $\Omega\Subset M$ is some precompact  open set.

\end{remark}

In contrast to Theorems~\ref{thm-G-para} and~\ref{thm-GPS-para}, the characterization of mixed-parabolicity and the mixed-Liouville property via the nonlinear mixed potential requires the following weak radial Harnack-type inequality.

\begin{definition}\label{def:weak-radial-Harnack}
We say that $M=M_1\times M_2$ satisfies the \emph{weak radial Harnack-type inequality} if for any $x= (x_1, x_2)$ with $x_1 \in M_1$ and $x_2 \in M_2$, and for any  $r \in (r_0,\, \infty)$ with  $r_0\in(0,\infty)$ a large fixed number, there exists a positive constant $C \in (1, \infty)$, which can depend on $x$ and $r$, such that for all $y = (y_1, y_2)$ satisfying $y_1 \notin B(x_1, 2r)$ and $y_2 \in B(x_2, r)$,
 \begin{equation}\label{eq-weak-radialharnack}
C^{-1} \, G^M(x;\, y_x) \le G^M(x;\,y) \le C \, G^M(x;\, y_x),
\end{equation}
where $y_x:=(y_1, x_2)\in M$.
\end{definition}

As will be seen in Section~\ref{ss5.1}, the Li--Yau estimate \(\SG\) on \(M\) implies the weak radial Harnack-type inequality \eqref{eq-weak-radialharnack}. Further examples of product manifolds \(M = M_1 \times M_2\) satisfying \eqref{eq-weak-radialharnack} can be easily obtained when one factor, say \(M_i\), is a connected sum of the form \(\mathcal{R}^n \mathbin{\#} \mathcal{R}^m\) and the other factor satisfies \(\SG\); see Example~\ref{eq-M-ends} for details.

We are now ready to state the main result of this paper.

\begin{theorem}\label{thm-main}
Let \(M = M_1 \times M_2\)
and \(p_1, p_2 \in (1,\infty)\). Then the following properties are equivalent:
\begin{enumerate}[\rm (i)]
  \item $M$ is $L^{p_2}(L^{p_1})$-parabolic.
 \item $G^M(x;\,\cdot)\notin L^{p_2'}(L^{p_1'})(M\setminus B(x,\,r))$  for some / all $x\in M$
  and some / all $r\in (0,\infty)$.
  \item $M$ admits the \emph{\(L^{p_2'}(L^{p_1'})\)-Liouville property}.
 \end{enumerate}
Further, if
$M$ satisfies the weak radial Harnack-type inequality \eqref{eq-weak-radialharnack},
then each of (i), (ii), and (iii) is  equivalent to the following:

\begin{enumerate}
  \item[\rm (iv)] ${\mathcal G}_{p_1,\, p_2}(f)(x)=\infty$ for some / all nonzero nonnegative $f\in {\mathcal C}_c^\infty(M)$ and some / all $x\in M$.

\end{enumerate}
\end{theorem}

The equivalence (i) $\Leftrightarrow$ (ii) $\Leftrightarrow$ (iii) is proved in Section \ref{s4} (see Theorem \ref{thm-para-Green}), while the equivalence of these conditions to (iv) is proved in Section \ref{s5} (see Theorem \ref{thm-G-poten}).

Applications of Theorem~\ref{thm-main} are given in Section~\ref{s6}. In particular, when \(\SG\) is imposed on \(M\), each condition in Theorem~\ref{thm-main} is equivalent to the following volume growth criterion (see Proposition~\ref{prop-para-vol}):
\begin{align}\label{eq-para-vol-mixed}
\int^\infty \left[ \int^\infty \left( \int_{r \vee s}^\infty \frac{t \, dt}{V_1(t) V_2(t)} \right)^{p_1'} V_1'(r) \, dr \right]^{\frac{p_2'}{p_1'}} V_2'(s) \, ds = \infty.
\end{align}
This mixed volume growth criterion extends the integral criterions in \eqref{eq-parabolic-iff} and \eqref{eq-Lppara-iff-vol}.

A remarkably different phenomenon emerges in the product setting. Consider, for example, the product space \(\mathbb{R}^{n_1} \times \mathbb{R}^{n_2}\). As an application of \eqref{eq-para-vol-mixed}, we obtain in Section~\ref{ss6.2} the following equivalences:
\begin{align}\label{eq1-Deff}
\mathbb R^{n_1} \times \mathbb R^{n_2}\text{ is } L^{p_2}(L^{p_1})\text{-parabolic}
&\quad\Leftrightarrow \quad
\mathbb R^{n_1} \times \mathbb R^{n_2} \text{ admits the } L^{q_2}(L^{q_1})\text{-Liouville property}\\
&\quad
\Leftrightarrow \quad
D_{\mathrm{eff}} := \frac{n_1}{p_1} + \frac{n_2}{p_2} \le 2,\notag
\end{align}
where $p_i = q_i'$ for $i = 1, 2$.
For comparison, we recall the classical isotropic criterion: for $p = q'$,
\begin{align}\label{eq0-Deff}
\mathbb R^N \text{ is } L^p\text{-parabolic}
\quad \Leftrightarrow \quad
\mathbb R^N \text{ admits the } L^q\text{-Liouville property}
\quad \Leftrightarrow \quad
\frac{N}{p} \le 2.
\end{align}
Note that the classical criterion \eqref{eq0-Deff} is governed entirely by the single parameter \(N/p\). In sharp contrast, the mixed criterion \eqref{eq1-Deff} involves \((n_1, p_1)\) and \((n_2, p_2)\) separately, through the anisotropic dimension
\[
D_{\mathrm{eff}} := \frac{n_1}{p_1} + \frac{n_2}{p_2}.
\]
Indeed, the classical criterion \eqref{eq0-Deff} is recovered as the special case \(n_1 = N\) and \(n_2 = 0\), for which \(D_{\mathrm{eff}} = N/p\). Thus, \eqref{eq1-Deff} serves as a genuine anisotropic analogue of \eqref{eq0-Deff}. This demonstrates that the mixed framework is naturally adapted to product structures  captures geometric features that are invisible from the standpoint of the usual isotropic theory.

\section{Riemannian product manifold and Green function}\label{s2}
In Section~\ref{ss2.1}, we state some basic facts about Riemannian product manifolds.
Then, in  Section~\ref{ss2.2}, we present various properties of Green functions that will be used later.

\subsection{Riemannian product manifold}\label{ss2.1}
Given two Riemannian manifolds \((M_1, { g}_1, d_1, V_1)\) and \((M_2, { g}_2, d_2, V_2)\), we consider their \emph{Riemannian product}
\[
M := M_1 \times M_2,
\]
endowed with the natural \emph{product metric} \({ g} = { g}_1 \oplus { g}_2\). This metric is defined by
\[
g(x_1,x_2)\big( (X_1, X_2), (Y_1, Y_2) \big) = g_1(x_1)(X_1, Y_1) + g_2(x_2)(X_2, Y_2),
\]
for all vectors \(X_1, Y_1 \in T_{x_1} M_1\) and \(X_2, Y_2 \in T_{x_2} M_2\), where \(T_{x_i} M_i\) denotes the tangent space at the point \(x_i \in M_i\).
The Riemannian distance on $M$ that is induced by \({ g}\), denoted by \(d(\cdot, \cdot)\), is given by
\begin{align}\label{eq-dg}
d\big((x_1, x_2), (y_1, y_2)\big) = \sqrt{ d_1(x_1, y_1)^2 + d_2(x_2, y_2)^2 }.
\end{align}
The Riemannian volume \(V\) is the product measure
\begin{align}\label{eq-Vg}
dV= dV_{1} \otimes dV_{2}.
\end{align}
If both \((M_1, { g}_1)\) and \((M_2, { g}_2)\) are noncompact and geodesically complete, then \((M, { g})\) is also noncompact and geodesically complete.
Moreover, closed balls in \((M, { g})\) are compact. In conclusion, for the Riemannian product \( M := M_1 \times M_2 \) with \(d\) and \(V\) as defined in \eqref{eq-dg} and \eqref{eq-Vg}, the triple \( (M, d, V)\) is a noncompact, complete metric measure space.
For a more detailed discussion of the Riemannian product manifold, the reader is referred to the monograph \cite[Section~3.8]{Grigoryan2009book}.

On the Riemannian product manifold \( M= M_1 \times M_2 \), we alternatively define for all \(x=(x_1, x_2)\in M\)
and \(y=(y_1, y_2)\in M\) that
\begin{align}\label{eq-dinfty}
d_\infty(x,y):=\max\left\{ d_{_1}(x_1, y_1),\, d_{_2}(x_2, y_2)\right\}.
\end{align}
This is usually known as the \emph{\(\infty\)-metric}.
Observe that
\begin{align}\label{eq-dg=dinfty}
d_\infty(x,y)\le d(x,y)\le {\sqrt 2}\, d_\infty(x,y).
\end{align}
Instead of using the Riemannian product metric \(d\) from \eqref{eq-dg}, it is sometimes more convenient
to use the \(\infty\)-metric, especially when a volume doubling condition is assumed for each \((M_i, d_i, V_i)\).

\begin{definition}\label{def-vd}
A Riemannian manifold \((M,d,V)\) is said to satisfy the \emph{volume doubling} condition, denoted by $(\hyperref[e:VD]{\mathrm{VD}})$, if there exists a constant \(C_D \ge 1\) such that, for every \(x \in M\) and \(r \in (0,\infty)\),
\begin{equation}\label{e:VD}
V(x, 2r) \le C_D \, V(x, r),
\end{equation}
\end{definition}

\begin{remark}
 Note that \eqref{e:VD} holds if and only if there exist constants $C_{D}'\in (1,\infty )$ and $\alpha_+ > 0$ such that, for all $x,\,y\in M$ and $0<r\leq R<\infty$,
\begin{align*}
  \frac{V(x,R)}{V(y,r)}\leq C_D' \left( \frac{d(x,y)+R}{r}\right) ^{\alpha_+ }.
\end{align*}
Another useful consequence of $\vd$ is that  for all $x,\,y\in M$ and $r\in(0,\infty)$,
\begin{equation}\label{eq:Vxy}
  V(x,r)+V(x,y)\simeq V(x, r+d(x,y)),
\end{equation}
where
\begin{align}\label{eq:defV}
 V(x,y):= V(B(x,d(x,y)))+ V(B(y,d(x,y))).
\end{align}
\end{remark}

\begin{proposition}\label{prop-vd}
For $i\in\{1,2\}$, assume that each manifold $(M_i, g_i, d_i, V_i)$ satisfies the volume doubling condition $\vd$.
Let \( M:= M_1 \times M_2 \) be the  Riemannian product manifold,  endowed with the product metric
\(g = g_1 \oplus g_2\), which induces \(d\) and \(V\) as in \eqref{eq-dg} and \eqref{eq-Vg}.
Then, the following  hold:

\begin{enumerate}[\rm (i)]
    \item For any $x=(x_1,x_2)\in M$ and $r\in(0,\infty)$, it holds that
  \begin{align}\label{eq1-Vprod}
V(x,r)\simeq \prod_{i=1}^2 V_{i}(x_i, r).
\end{align}
As a consequence, $(M, g, d, V)$ itself  satisfies the volume doubling condition $\vd$.

\item  For any $x=(x_1,x_2)\in M$, $y=(y_1,y_2)\in M$ and $r\in(0,\infty)$,
\begin{align}\label{eq2-Vprod}
V(x,r)+ V(x,y) \simeq V(x, r+d(x,y))
 \simeq V(x, r+d_\infty(x,y))
\end{align}
and
\begin{align}\label{eq3-Vprod}
V(x,y) \simeq \prod_{i=1}^2 V_{i}(x_i,\, d_\infty(x,y)),
\end{align}
where \(V(x, y)\) is defined as in \eqref{eq:defV}, and \(d_\infty\) is the \(\infty\)-metric defined in \eqref{eq-dinfty}.
\end{enumerate}
\end{proposition}

\begin{proof}
 For any $x=(x_1,x_2)\in M$ and $r\in(0,\infty)$, by the equivalence \(d \simeq d_\infty\) in \eqref{eq-dg=dinfty}, we have
  \begin{align}\label{eq-inclu}
 B(x, r)\subset  B_{1}(x_1, r) \times B_{2}(x_2, r)\subset  B(x, \sqrt 2r).
\end{align}
It is obvious that the first inclusion in \eqref{eq-inclu} implies
  \begin{align*}
V(x,r)=V(B(x, r))\le V \Big(B_{1}(x_1, r) \times B_{2}(x_2, r)\Big)
= \prod_{i=1}^2 V_{i}(x_i, r).
\end{align*}
Using the second inclusion in \eqref{eq-inclu} and the fact that each $(M_i, g_i, d_i, V_i)$ satisfies $\vd$,
we obtain
 \begin{align*}
 \prod_{i=1}^2 V_{i}(x_i, r)
 &\le C_D^2 \prod_{i=1}^2 V_{i}\left(x_i, \frac r 2\right)\\
 & = C_D^2 V\left(B_{1}\left(x_1, \frac r 2\right) \times B_{2}\left(x_2, \frac r 2\right)\right)\\
 &
 \le C_D^2 V \left(B\left(x, \frac {\sqrt 2 r} 2\right)\right)\\
 &
 \le C_D^2 V\left(x, r\right).
\end{align*}
Combining the last two estimates yields \eqref{eq1-Vprod}.

Once we have obtained \eqref{eq1-Vprod}, the fact that each \((M_i, g_i, d_i, V_i)\) satisfies  $\vd$ immediately implies that the Riemannian product manifold \((M, g,d, V)\) itself satisfies $\vd$.

Having established that the Riemannian product manifold $(M, g, d, V)$ satisfies $\vd$, it follows from the general volume doubling consequence \eqref{eq:Vxy} together with \eqref{eq-dg=dinfty} that \eqref{eq2-Vprod} holds. Furthermore, \eqref{eq3-Vprod} follows directly from \eqref{eq1-Vprod} and \eqref{eq-dg=dinfty}.
\end{proof}

\subsection{Green functions on a general Riemannian manifold}\label{ss2.2}

We begin with the following basic facts of Green functions  that will be used in this paper (see  \cite[Theorem 13.17]{Grigoryan2009book} and \cite[Corollary 13.13]{Grigoryan2009book}).

\begin{lemma}[\cite{Grigoryan2009book}]\label{lem-G-property}
Suppose that $(M,\ g)$ is a noncompact geodesically complete Riemannian manifold.
Fix $x\in M$. The Green function $G^M(x;\,\cdot)$ is the infimum of all positive fundamental solutions of $\varDelta_M$ at $x$. In particular, \(G^M(x;\, \cdot)\)
is harmonic in \(M \setminus \{x\}\). Moreover, if  \(G^M(x;\, y_0) < \infty\) for some \(y_0 \in M\), then
  \(G^M(x;\, y) < \infty\) for all \(y\in M\) and $G^M(x;\, \cdot)\in {\mathcal C}^\infty(M\setminus\{x\})$.
\end{lemma}

Let \((M,  g)\) be a general noncompact, geodesically complete Riemannian manifold (not necessarily a product manifold).
For any open subset \(\Omega \subset M\), let \(W^{1,2}_0(\Omega)\) denote the closure of \(\mathcal{C}_c^\infty(\Omega)\) in the Sobolev space \(W^{1,2}(M)\).
Define the Dirichlet Laplace operator
\[
\mathcal{L}^\Omega = -\varDelta_M \big|_{W^{1,2}_0(\Omega)}.
\]
With $\mathcal{L}^\Omega$, we associate the heat kernel \(\{p_t^\Omega\}_{t\in(0,\infty)}\) and the Green function \(G^\Omega(x, y)\). Denote by $\lambda_{\rm min}(\Omega)$ the \emph{infimum spectrum} of $\mathcal L^\Omega$, that is,
$$
\lambda_{\rm min}(\Omega)=\inf_{0\nequiv f\in {\mathcal C}_c^\infty(\Omega)} \frac{\|\nabla f\|_{L^2(\Omega)}^2}{\|f\|_{L^2(\Omega)}^2}.
$$
If, in addition, $\Omega\Subset  M$ is a precompact open set, then by \cite[Theorem~10.22]{Grigoryan2009book}, it always holds that
$$\lambda_{\rm min}(\Omega)>0.$$
When $\lambda_{\rm min}(\Omega)>0$, it follows from \cite[Theorem~13.4]{Grigoryan2009book} that the Green operator \(G^\Omega\) is a bounded self-adjoint operator on \(L^2(\Omega)\) and satisfies
\[
G^\Omega = (\mathcal L^\Omega)^{-1} \quad \text{in } L^2(\Omega).
\]
In this case, the Green function \(G^\Omega(x;\, y)\) is  finite and, for any \(x \in \Omega\), the function \(G^\Omega(x;\, \cdot)\) satisfies
\begin{align}\label{eq-GOmega}
\mathcal L^\Omega G^\Omega(x; \,\cdot) = \delta_x(\cdot),
\end{align}
where \(\delta_x\) denotes the Dirac measure at \(x\).
In particular, \(G^\Omega(x;\, \cdot)\) is harmonic on \(\Omega \setminus \{x\}\).

In general, assume only that $\Omega \subset M$ is open (not necessarily precompact).
Then, for any \(\lambda \in (0, \infty)\), define the \emph{resolvent} \[G_\lambda^\Omega = (\lambda + \mathcal{L}^\Omega)^{-1}.\] This is a bounded, nonnegative definite, self-adjoint operator on \(L^2(\Omega)\) with operator norm at most \(\lambda^{-1}\) (see \cite[Theorem~4.5]{Grigoryan2009book}). Moreover, for all \(f \in L^2(\Omega)\) and \(x \in \Omega\), we have (see \cite[Lemma~5.10]{Grigoryan2009book})
\[
G_\lambda^\Omega f(x) = \int_0^\infty e^{-\lambda t} P_t^\Omega f(x) \, dt,
\]
where \(\{P_t^\Omega\}_{t\in(0,\infty)}\) with \(P_t^\Omega = e^{-t\mathcal{L}^\Omega}\) is the heat semigroup associated with \(\mathcal{L}^\Omega\). From this representation and \eqref{eq:Green-kernel}, it follows that as \(\lambda \downarrow 0\), \(G_\lambda^\Omega f(x)\) increases monotonically to \(G^\Omega f(x)\).

\begin{lemma}\label{lem-Gf-continue}
Let \(\Omega\subset M\) be an open set. For any \(\lambda > 0\) and any \(f\in L^\infty(M)\) with compact support contained in \(\Omega\), the function \(G_\lambda^\Omega f\) belongs to \(\mathcal{C}(\Omega)\).
\end{lemma}

\begin{proof}
Assume without loss of generality that \(K = \operatorname{supp} f\). Then \(K \Subset  \Omega\).
By \cite[Theorem 7.20]{Grigoryan2009book}, for every \(t > 0\), the heat kernel \(p_t^{\Omega}\) is jointly continuous in \((t, x, y) \in (0, \infty) \times \Omega \times \Omega\). Consequently, for any precompact open ball \(B \Subset  \Omega\) and any \(t > 0\), there exists a constant \(C = C_{t, B, K}\) such that
\[
\sup_{x \in B, \, y \in K} p_t^{\Omega}(x, y) \le C.
\]
Using this uniform bound together with the boundedness of \(f\), we may apply the dominated convergence theorem to the representation
\[
P_t^{\Omega} f(x) = \int_M p_t^{\Omega}(x, y) f(y) \, dV(y),
\]
and conclude that \(P_t^{\Omega} f\) is continuous on \(B\). Since \(B\) is an arbitrary precompact ball in \(\Omega\), it follows that \(P_t^{\Omega} f\) is continuous on all of \(\Omega\).

Moreover, we have the estimate \(0\le P_t^{\Omega} f(x) \le \|f\|_{L^\infty(M)}\) for all \((x, t) \in \Omega \times (0, \infty)\).
Combining this with the continuity of \(P_t^{\Omega} f\) on \(\Omega\) and the resolvent representation
\[
G_\lambda^{\Omega} f(x) = \int_0^\infty e^{-\lambda t} P_t^{\Omega} f(x) \, dt,
\]
the dominated convergence theorem once again implies that \(G_\lambda^{\Omega} f\) is continuous on \(\Omega\). This completes the proof.
\end{proof}

Note that in Lemma \ref{lem-Gf-continue}, the open set \(\Omega \subset M\) can be taken to be \(M\) itself. Usually, one can consider \(G^\Omega f\) for precompact open sets with smooth boundary and then transfer properties of \(G^{\Omega}\) to \(G^M\) via a standard exhaustion argument.

By an \emph{exhaustion sequence} \(\{\Omega_k\}_{k \in \mathbb N}\) of \(M\), we mean that each \(\Omega_k\) is a precompact open subset of \(M\) with smooth boundary, \(\Omega_k \Subset  \Omega_{k+1}\) for all \(k \in \mathbb N\), and \(M = \bigcup_{k \in \mathbb N} \Omega_k\).
For any \(k \in \mathbb N\) and \(x, y \in M\), the monotonicity \(p_t^{\Omega_k}(x, y) \le p_t^{\Omega_{k+1}}(x, y)\) for all \(t > 0\) implies
\[
G^{\Omega_k}(x;\, y) \le G^{\Omega_{k+1}}(x;\, y).
\]
Moreover, by \cite{LiTam1987AJM} (see also \cite[p.~131]{Grigoryan1999BAMS}), for all \(x, y \in M\) we have
\begin{equation}\label{eq-GkG}
G^M(x;\, y) = \lim_{k \to \infty} G^{\Omega_k}(x;\, y).
\end{equation}
A combination of the exhaustion argument and Lemma \ref{lem-Gf-continue} yields the following conclusion.

\begin{corollary}\label{cor-Gf-lowersemi}
Let \(\Omega\subset M\) be an open set. For any nonnegative function \(f\) and $t\in(0,\infty)$, the set \(\{x\in M:\ G^\Omega f(x)>t\}\) is open and contained in  $\Omega$.
\end{corollary}

\begin{proof}
If \(x \notin \Omega\), then \(G^\Omega f(x) = 0\), which implies that
$
\{x \in M :\ G^\Omega f(x) > t\} \subset \Omega.
$
Thus, it suffices to show that the set
$
 \{x \in \Omega :\ G^\Omega f(x) > t\}
$
is open.
To this end, take a point \(x_0 \in \Omega\) such that $G^\Omega f(x_0) > t$, we need to find some \(\delta > 0\) such that \(G^\Omega f(x) > t\) for all \(x \in B(x_0, \delta)\).

Let \(\{\Omega_k\}_{k \in \mathbb N}\) be an exhaustion sequence  of \(\Omega\).
 For each \(k \in \mathbb{N}\), define
\[
f_k := \min\{k, f\} \cdot \mathbf{1}_{\Omega_k},
\]
which is a bounded function with compact support in $\Omega$. Note that the sequence \(\{f_k\}_{k \in \mathbb{N}}\) increases pointwise to \(f\) on \(M\).
Let \(\{\lambda_k\}_{k \in \mathbb{N}}\) be a positive sequence decreasing to \(0\). Then
\[
G_{\lambda_k}^\Omega f_k \nearrow G^\Omega f \quad \text{pointwise on }\, \Omega.
\]
In particular, we have
\[
\lim_{k \to \infty} G_{\lambda_k}^\Omega f_k(x_0) = G^\Omega f(x_0) > t,
\]
which implies that there exists some \(k_0 \in \mathbb{N}\) such that
\[
G_{\lambda_{k_0}}^\Omega f_{k_0}(x_0) > t.
\]
By Lemma \ref{lem-Gf-continue}, the function \(G_{\lambda_{k_0}}^\Omega f_{k_0}\) is continuous on \(\Omega\) and, hence, it is continuous at \(x_0\).
Thus, there exists \(\delta > 0\) such that for all \(x \in B(x_0, \delta)\),
\[
G_{\lambda_{k_0}}^\Omega f_{k_0}(x) > t.
\]
Finally, since \(G^\Omega f \ge G_{\lambda_{k_0}}^\Omega f_{k_0}\) pointwise, we obtain that for all \(x \in B(x_0, \delta)\),
\[
G^\Omega f(x) \ge G_{\lambda_{k_0}}^\Omega f_{k_0}(x) > t.
\]
This shows that the set \(\{x\in M:\ G^\Omega f(x)>t\}\) is open, thereby completing the proof.
\end{proof}


At the end of this subsection, we present the following lemma concerning the continuity of \(G^\Omega f\) for \(f \in \mathcal{C}^\infty(\Omega)\), which is taken from
\cite[Lemma~3.4]{FarajiGrigoryan2019RMI}. 

\begin{lemma}[\cite{FarajiGrigoryan2019RMI}]\label{lem-FG3.4}
Let \(\Omega\subset M\) be an open set.
If \(0\le f \in C^\infty(\Omega)\) such that $G^\Omega f(x)<\infty$ for some $x\in \Omega$, then
 \(G^\Omega f \in {\mathcal C}^\infty(\Omega)\) and $\cl^\Omega (G^\Omega f ) = f$ pointwise.
\end{lemma}

\section{A capacity theory for mixed-norm Lebesgue spaces}\label{s3}

Sections \ref{ss3.1}--\ref{sec-cap}--\ref{sec-cap-equiv} focus on some basic theory of mixed Lebesgue spaces and mixed capacities (where the endpoint cases $p_1, p_2 = 1$ are included as usual), while Section \ref{sec-mix-cap-poten} treats the nonlinear mixed-potential $\mathcal{G}_{p_1, p_2}^\Omega(f)$ for $p_1, p_2 \in (1, \infty)$.

\subsection{Basic properties of mixed-norm Lebesgue spaces}\label{ss3.1}

We begin with the following duality lemma that is essentially due to Benedek and Panzone \cite{BenedekPanzone1961Duke}.

\begin{lemma}\label{lem-duality}
Let $M = M_1 \times M_2$ be the Riemannian product manifold, equipped with the  Riemannian product distance $d$ and the Riemannian product volume $V$, as given by  \eqref{eq-dg} and \eqref{eq-Vg}.

 \begin{enumerate}
   \item[\rm (i)] For any $p_1,\ p_2\in[1,\infty)$ and any $V$-measurable function $f$ on $M$,
   \begin{align}\label{eq-norm}
  \|f\|_{L^{p_2}(L^{p_1})(M)}
  &=\sup\left\{\left|\int_M fg\,dV \right|:\  \|g\|_{L^{p_2'}(L^{p_1'})(M)}\le 1\right\}.
 \end{align}

  \item[\rm (ii)]   For any $p_1,\ p_2\in(1,\fz)$, any $0\not\equiv f\in L^{p_2}(L^{p_1})(M)$ and any $0\not\equiv g\in L^{p_2'}(L^{p_1'})(M)$,
       \begin{align}\label{eq-Holder}
  \int_M fg\,dV \le \|f\|_{L^{p_2}(L^{p_1})(M)}\|g\|_{L^{p_2'}(L^{p_1'})(M)},
 \end{align}
 with equality holds if and only if for almost all $x_1\in M_1$ and $x_2\in M_2$,
 \begin{align}\label{eq-Holder=}
 g(x_1,x_2) =\lambda\,\mathrm{sign} f(x_1,x_2)\, |f(x_1,x_2)|^{p_1-1} \left(\int_{M_1} |f(x_1,x_2)|^{p_1} \,dV_{1}(x_1)\right)^{\frac {p_2-p_1}{p_1}},
 \end{align}
 where $\lambda\in(0,\infty)$.

  \item[\rm (iii)] For any $p_1,\ p_2\in[1,\fz)$, there holds the duality identity
 \begin{align}\label{eq-duality}
  \left(L^{p_2}(L^{p_1})(M)\right)'=L^{p_2'}(L^{p_1'})(M).
 \end{align}
 \end{enumerate}

\end{lemma}

\begin{proof}
 The  equality in \eqref{eq-norm} follows from \cite[Theorem~2]{BenedekPanzone1961Duke}.  The H\"older inequality \eqref{eq-Holder} is a consequence of \eqref{eq-norm}. It is a direct computation that if we take $g$ as in \eqref{eq-Holder=} then we have equality in  \eqref{eq-Holder}.
 Finally, the duality in \eqref{eq-duality} is from \cite[Theorem~1.a)]{BenedekPanzone1961Duke}.
\end{proof}

The rest of this subsection deals with density properties of mixed-norm Lebesgue spaces.

\begin{lemma}\label{lem-dense}
Let  $M = M_1 \times M_2$ and $p_1,\ p_2\in[1,\fz)$.  Then, ${\mathcal C}_c^\infty(M)$
 is dense in $L^{p_2}(L^{p_1})(M)$.
\end{lemma}

\begin{proof}

It is obvious that  ${\mathcal C}_c(M)\subset L^{p_2}(L^{p_1})(M)$.
Given any $f\in L^{p_2}(L^{p_1})(M)$, we have by  \cite[p. 313]{BenedekPanzone1961Duke} that, for any $\varepsilon>0$, there exists a simple function
	$$\varphi(x_1,x_2):=\sum_{i=1}^N c_{i}\,{\bf 1}_{E_i}(x_1)\,{\bf 1}_{F_i}(x_2),$$
where $c_i\in\cc$, $E_i\subset M_1$, $F_i\subset M_2$, $V_{1}(E_i)<\infty$ and  $V_{2}(F_i)<\infty$ for all $i\in\{1,2,\dots, N\}$,
	such that
	\begin{align*}
		\|f-\varphi\|_{L^{p_2}(L^{p_1})(M)}<
		\varepsilon.
	\end{align*}
Thus, to obtain the density of ${\mathcal C}_c^\infty(M)$
in $L^{p_2}(L^{p_1})(M)$, it suffices to consider functions $f$ of type
$$f(x_1,x_2)={\bf 1}_E(x_1){\bf 1}_F(x_2),$$ where $E\subset  M_1$, $ F\subset  M_2$ such that $0<V_{1}(E)<\infty$ and  $0<V_{2}(F)<\infty$.

It is a classical result that ${\mathcal C}_c^\infty(M_i)$ is dense in
$L^{p_i}(M_i)$; see, for example,
  \cite[p.\,101, Exercise~4.4]{Grigoryan2009book}.
Thus,  for any $\varepsilon>0$, there exists a function $\psi_1\in {\mathcal C}_c^\infty(M_1)$ such that
 $$\|{\bf 1}_E-\psi_1\|_{L^{p_1}(M_1)}<\varepsilon \|{\bf 1}_F\|_{L^{p_2}(M_2)}^{-1}.$$
 In a similar manner, ${\mathcal C}_c^\infty(M_2)$ is dense in  $L^{p_2}(M_2)$.
So, for the above $\varepsilon>0$, there exists a function  $\psi_2\in {\mathcal C}_c^\infty(M_2)$ such that
 $$\|{\bf 1}_F-\psi_2\|_{L^{p_2}(M_2)}<\varepsilon \|\psi_1\|_{L^{p_1}(M_1)}^{-1}.$$
For any $x_1\in M_1$ and $x_2\in M_2$, let $$h(x_1,x_2):=\psi_1(x_1)\psi_2(x_2),$$ which belongs to ${\mathcal C}_c^\infty(M)$. Moreover, we have
	\begin{align*}
		&\|{\bf 1}_{E\times F}-h\|_{L^{p_2}(L^{p_1})(M)}\\
		&\quad=\lf(\int_{M_2}\lf(\int_{M_1}|{\bf 1}_{E}(x_1){\bf 1}_F(x_2)-\psi_1(x_1)\psi_2(x_2)|^{p_1} \,dV_{1}(x_1)\r)
^{\frac {p_2} {p_1} }\,dV_{2}(x_2)\r)^{\frac 1 {p_2}}\\
		&\quad\leq \lf(\int_{M_2}\lf(\int_{M_1}|{\bf 1}_{E}(x_1){\bf 1}_F(x_2)-\psi_1(x_1){\bf 1}_F(x_2)|^{p_1} \,dV_{1}(x_1)\r)
^{\frac {p_2} {p_1} }\,dV_{2}(x_2)\r)^{\frac 1 {p_2}}\\
		&\qquad +  \lf(\int_{M_2}
\lf(\int_{M_1}|\psi_1(x_1){\bf 1}_F(x_2)-\psi_1(x_1)\psi_2(x_2)|^{p_1}\, dV_{1}(x_1)\r)^{\frac {p_2} {p_1} }\,dV_{2}(x_2)\r)^{\frac 1 {p_2}}\\
&\quad= \|{\bf 1}_E-\psi_1\|_{L^{p_1}(M_1)} \|{\bf 1}_F\|_{L^{p_2}(M_2)}
+ \|\psi_1\|_{L^{p_1}(M_1)} \|{\bf 1}_F-\psi_2\|_{L^{p_2}(M_2)}\\
		&\quad<2\varepsilon.
	\end{align*}
In conclusion, we obtain that ${\mathcal C}_c^\infty(M)$
is dense in $L^{p_2}(L^{p_1})(M)$.
\end{proof}

\begin{remark}\label{rem-dense}
For \(p \in [1, \infty)\) and a general Riemannian  manifold
\((M, g, d, V)\), any nonnegative function \(f \in L^p(M)\) can be approximated by nonnegative functions in \({\mathcal C}_c^\infty(M)\).
Indeed, the proof for the case \(M = \rn\) follows from the classical density result \cite[Theorem~2.3]{Grigoryan2009book}; this argument extends directly to a general Riemannian manifold by using the partition of unity method and local coordinate charts.
By this general fact and the proof of Lemma \ref{lem-dense}, we obtain that any   nonnegative $f\in L^{p_2}(L^{p_1})(M)$  can be approximated with respect to the  mixed norm $\|\cdot\|_{L^{p_2}(L^{p_1})(M)}$ by a sequence $\{f_j\}_{j\in\nn}$ of nonnegative ${\mathcal C}_c^\infty(M)$-functions.
\end{remark}

\begin{lemma}\label{lem-dense-Cc}
  Let \(M = M_1 \times M_2\) and $p_1,\ p_2\in[1,\fz)$.  Then, for any $f\in {\mathcal C}_c(M)$,
  there exists a sequence $\{f_j\}_{j\in\nn}\subset {\mathcal C}_c^\infty(M)$ such that
  $$\lim_{j\to\infty}\left(\|f_j-f\|_{L^\infty(M)}+\|f_j-f\|_{L^{p_2}(L^{p_1})(M)}\right)=0.$$
\end{lemma}

\begin{proof}
We may assume without loss of generality that \(f\) is supported in \(K_1 \times K_2\), where \(K_i \subset M_i\) is compact.

For \(i = 1,2\), take a nonempty precompact open set \(\Omega_i \Subset  M_i\) containing \(K_i\).
Then, for any \(\varepsilon > 0\), there exists a function \(\phi \in \mathcal{C}_c^\infty(M)\) such that
\[
\|\phi - f\|_{L^\infty(M)} < \varepsilon \, \min\left\{1,\ V_1(\Omega_1)^{-1/p_1}\, V_2(\Omega_2)^{-1/p_2}\right\};
\]
see, for instance, \cite[Exercise~4.5]{Grigoryan2009book}. Moreover, the function \(\phi\) can be chosen so that
$
\operatorname{supp} \phi \subset \Omega_1 \times \Omega_2.
$
Consequently,
\begin{align*}
&\|\phi - f\|_{L^{p_2}(L^{p_1})(M)}\\
&\quad= \left( \int_{\Omega_2} \left( \int_{\Omega_1} |\phi(x_1, x_2) - f(x_1, x_2)|^{p_1} \, dV_{1}(x_1) \right)^{\frac{p_2}{p_1}} dV_{2}(x_2) \right)^{\frac{1}{p_2}} \\
&\quad\le \|\phi - f\|_{L^\infty(M)} \, V_{1}(\Omega_1)^{1/p_1} \, V_{2}(\Omega_2)^{1/p_2} \\
&\quad< \varepsilon.
\end{align*}
For any $j\in\nn$, by taking \(\varepsilon = 2^{-j}\) and setting \(f_j = \phi\), we obtain the desired sequence \(\{f_j\}_{j \in \mathbb{N}}\).
\end{proof}

\subsection{Measure-theoretic properties of the mixed-capacity}\label{sec-cap}

\begin{definition}
We use the terminology ``quasi-everywhere'' (abbreviated as $\operatorname{cap}_{p_1,\,p_2}$-q.e.) to mean ``everywhere except on a set of $\operatorname{cap}_{p_1,\,p_2}$-capacity zero''.
\end{definition}

\begin{remark}
Let \(M = M_1 \times M_2\),  \(\Omega \subset M\) be an open set and \(p_1, p_2 \in [1,\infty)\).
For any  $f \in L^{p_2}(L^{p_1})(M)$ and $t \in (0, \infty)$, by taking $t^{-1}|f|$ as a test function, we obtain
        $$\operatorname{cap}_{p_1,\,p_2}\lf(\lf\{x \in \Omega:\,\, |G^\Omega f(x)| \geq t \r\};\  \Omega\r)
	\leq t^{-p_2}\|f\|_{L^{p_2}(L^{p_1})(M)}^{p_2},$$
which, upon letting \(t\to\infty\), yields
$$\operatorname{cap}_{p_1,\,p_2}\lf(\lf\{x \in \Omega:\,\, |G^\Omega f(x)| =\infty \r\};\ \Omega\r)
=0.$$
Consequently, for every \(f \in L^{p_2}(L^{p_1})(M)\), the function \(G^\Omega f\) is finite $\operatorname{cap}_{p_1,\,p_2}$-q.e. on $\Omega$.
\end{remark}

The following measure-theoretic properties of mixed-capacities, already proved in \cite{JinLiuWuXiao2026JGA} when \(M = \mathbb{R}^n \times \mathbb{R}^n\), show that \(\operatorname{cap}_{p_1,\,p_2}\) is a \emph{Choquet capacity} in the sense of Choquet \cite{Choquet1953AIF} for any \(p_1, p_2 \in (1, \infty)\).

\begin{theorem}\label{Thm-Choquet}
Let \(M = M_1 \times M_2\),  \(\Omega \subset M\) be an open set and \(p_1, p_2 \in (1,\infty)\).
 Then the following  hold:
\begin{enumerate}[\rm (i)]
	\item $\operatorname{cap}_{p_1,\,p_2}(\emptyset;\ \Omega)=0$, where $\emptyset$ denotes the empty set in $M$.
	
	\item (Monotonicity) For any subsets $E_{1}, E_{2} $ of $\Omega$ satisfying $E_{1} \subset  E_{2}$,
	$$\operatorname{cap}_{p_1,\,p_2}(E_{1};\ \Omega)
	\leq \operatorname{cap}_{p_1,\,p_2}(E_{2};\ \Omega).$$
	
	\item (Subadditivity) For any sequence of subsets $\{E_{j}\}_{j \in \mathbb{N}}\subset\Omega$,
    \begin{align*}
    \operatorname{cap}_{p_1,\,p_2}\lf(\bigcup_{j \in \mathbb{N}}E_{j};\ \Omega\r)
    	\leq \left(\sum_{j \in \mathbb{N}}\lf[\operatorname{cap}_{p_1,\,p_2}\lf(E_{j};\ \Omega\r)\r]^{\frac{p_1}{p_2}\wedge 1}\right)^{\frac{p_2}{p_1}\vee 1}.
    \end{align*}

\item
For any $\{f_k\}_{k \in \mathbb{N}}$ and $f$  in $L^{p_2}(L^{p_1})(M)$
 satisfying
\begin{align*}
\lim_{k\to\infty} \lf\|f_{k}-f\r\|_{L^{p_2}(L^{p_1})(M)}=0,
\end{align*}
there exists a subsequence $\{f_{k_{i}}\}_{i \in \mathbb{N}}$
such that
\begin{align*}
\lim_{i\to\infty} G^\Omega f_{k_i}=G^\Omega f \qquad \operatorname{cap}_{p_1,\,p_2}\text{-q.e. on}\ \, \Omega.
\end{align*}

    \item (Continuity on the right)
	For any decreasing sequence of compact sets  $\{K_j\}_{j \in \mathbb{N}}\subset\Omega$,
	$$\lim_{j \to \infty}\operatorname{cap}_{p_1,\,p_2}(K_j;\ \Omega)
	= \operatorname{cap}_{p_1,\,p_2}\lf( \bigcap_{j \in \mathbb{N}}K_j;\ \Omega\r).$$

	\item (Continuity on the left)
	For any increasing sequence $\{E_j\}_{j \in \mathbb{N}}\subset\Omega$,
	$$\lim_{j \to \infty}\operatorname{cap}_{p_1,\,p_2}(E_j;\ \Omega)
	= \operatorname{cap}_{p_1,\,p_2}\lf( \bigcup_{j \in \mathbb{N}}E_j;\ \Omega\r).$$

\item (Regularity property)  Each Borel set $E\subset \Omega$ enjoys the outer regularity
  $$
\operatorname{cap}_{p_1,\,p_2}(E;\ \Omega)=\inf_{\gfz{E \subset  O\subset \Omega} {O \,\text{open}}}\operatorname{cap}_{p_1,\,p_2}(O;\ \Omega),
$$
    and the  inner regularity
$$
\operatorname{cap}_{p_1,\,p_2}(E;\ \Omega)=\sup_{\gfz{ K\subset  E } {K \,\text{compact}}}\operatorname{cap}_{p_1,\,p_2}(K;\ \Omega).
$$
\end{enumerate}
In particular, items (i), (ii), (iii), (iv), and the outer regularity of (vii) hold true even for the endpoint cases $p_1=1$ or $p_2=1$.
\end{theorem}

\begin{proof}
Items (i), (ii), and (iii) follow along the same lines as the proof of \cite[Theorem 4.1]{JinLiuWuXiao2026JGA}. Moreover, item (iv) can be obtained by applying the same arguments used in \cite[Lemma~4.2]{JinLiuWuXiao2026JGA}.

Once the outer regularity in (vii) is established, we may combine (i) through (iv) and again follow the reasoning in \cite[Theorem 4.3, Corollary 4.4]{JinLiuWuXiao2026JGA} to derive items (v) and (vi), as well as the inner regularity stated in (vii).

To obtain the outer regularity, we will use Corollary \ref{cor-Gf-lowersemi}.
First, by (ii), we see that the proof of the outer regularity can be reduced to proving
\begin{align}\label{eq1-outer}
\operatorname{cap}_{p_1,\,p_2}(E;\ \Omega)\ge \inf_{\gfz{E\subset O\subset \Omega} {O \,\text{open}}}\operatorname{cap}_{p_1,\,p_2}(O;\ \Omega).
\end{align}
Assume without loss of generality that $\operatorname{cap}_{p_1,\,p_2}(E;\ \Omega)$ is finite.
By \eqref{eq-cap},  for any $\varepsilon\in(0,1)$, there exists $0 \le f \in L^{p_2}(L^{p_1})(M)$
such that $G^{\Omega} f \ge \mathbf{1}_E$  and
\begin{align}\label{eq2-outer}
	\|f\|_{L^{p_2}(L^{p_1})(M)}^{p_2}
	\leq \operatorname{cap}_{p_1,\,p_2}(E;\,\Omega) + \varepsilon.
\end{align}
Let $f_{\varepsilon}:=(1-\varepsilon)^{-1}f$.
Applying Corollary \ref{cor-Gf-lowersemi} yields that the set
$$O_{\varepsilon}:=\{x \in \Omega:\,\,G^\Omega f_{\varepsilon}(x) > 1 \}=\{x \in M:\,\,G^\Omega f(x) > 1-\varepsilon \}$$
is  open and contains $E$.
 This, along with \eqref{eq-cap} and \eqref{eq2-outer}, yields
\begin{align*}
\operatorname{cap}_{p_1,\,p_2}(O_\varepsilon;\ \Omega)
&\leq \|f_\varepsilon\|_{L^{p_2}(L^{p_1})(M)}^{p_2}\\
&
	= (1-\varepsilon)^{-p_2}\|f\|_{L^{p_2}(L^{p_1})(M)}^{p_2}\\
&
\leq (1-\varepsilon)^{-p_2}
\left(\operatorname{cap}_{p_1,\,p_2}(E;\,\Omega) + \varepsilon\right).
\end{align*}
Letting $\varepsilon \rightarrow 0$ yields \eqref{eq1-outer}. This ends the proof of the outer regularity.
\end{proof}

\subsection{Equivalent characterizations of the mixed-capacity}\label{sec-cap-equiv}

\begin{theorem}\label{thm-approx-cap}
Let $M=M_1\times M_2$, $\Omega\subset M$ be an open set, and
\(\{\Omega_k\}_{k \in \mathbb N}\) be an exhaustion sequence of \(\Omega\).
Then, for any $p_1,\ p_2\in [1,\fz)$ and compact set $K\subset \Omega$, 
\begin{align}\label{eq-cap-limit}
\lim_{k\to\infty}\operatorname{cap}_{p_1,\,p_2}(K;\,\Omega_k)
=\operatorname{cap}_{p_1,\,p_2}(K;\, \Omega).
\end{align}
\end{theorem}

\begin{proof}
Since \( K \subset \Omega \) and \( \{\Omega_k\}_{k\in\nn} \) exhausts \( \Omega \), there exists \( k_0 \) such that \( K \subset \Omega_k \) for all \( k \geq k_0 \). We only consider \( k \geq k_0 \).
For such $k$, since $G^{\Omega_k}f\le G^\Omega f$ for all  nonnegative function $f$, it follows
 that
  $$
  \operatorname{cap}_{p_1,\,p_2}(K;\,\Omega_k)
\ge \operatorname{cap}_{p_1,\,p_2}(K;\, \Omega)
  $$
  by their definitions,
  which leads to the $\ge$ inequality in \eqref{eq-cap-limit}.

To prove the inequality \(\le\) in \eqref{eq-cap-limit}, we may assume that \(\operatorname{cap}_{p_1,\,p_2}(K;\, \Omega)\) is finite; otherwise there is nothing to prove.
For any \(\varepsilon > 0\), there exists a function \(0\le f \in L^{p_2}(L^{p_1})(M)\) such that
\[
G^\Omega f  \ge \mathbf{1}_K
\]
and
\[
\| f \|_{L^{p_2}(L^{p_1})(M)}^{p_2} < \operatorname{cap}_{p_1,\,p_2}(K;\, \Omega) + \varepsilon.
\]
 We may as well assume that $f=0$ outside of $\Omega$.
By the density of \(\mathcal{C}_c^\infty(M)\) in \(L^{p_2}(L^{p_1})(M)\) (see Lemma \ref{lem-dense} and Remark \ref{rem-dense}), we can choose a nonnegative sequence \(\{f_i\}_{i \in \mathbb{N}} \subset \mathcal{C}_c^\infty(M)\) such that
\[
\lim_{i \to \infty} \| f_i - f \|_{L^{p_2}(L^{p_1})(M)} = 0.
\]
After passing to a subsequence if necessary, we may assume that \(f_i \to f\) a.e. on \(M\) and that for every \(i \in \mathbb{N}\),
\begin{align}\label{eq-norm-fj-f}
\| f_i - f \|_{L^{p_2}(L^{p_1})(M)} < 2^{-i} \varepsilon.
\end{align}
Note that $f_i$ might be nonzero outside of $\Omega$.

For the given exhaustion sequence \(\{\Omega_i\}_{i \in \mathbb N}\)  of \(\Omega\),
by passing to a subsequence if necessary, we may assume that for each \(\Omega_i\),
\begin{align*}
\|f {\mathbf 1}_{M\setminus \Omega_{i}}\|_{L^{p_2}(L^{p_1})(M)}
=\|f {\mathbf 1}_{\Omega\setminus \Omega_{i}}\|_{L^{p_2}(L^{p_1})(M)} < 2^{-i-3}\varepsilon.
\end{align*}
Choose a sequence of cutoff functions \(\{\phi_i\}_{i \in \mathbb{N}} \subset \mathcal{C}_c^\infty(M)\) such that for all \(i \ge 4\),
\begin{align*}
\begin{cases}
   \phi_i=1\ \ \text{on}\ \ \Omega_{i-3}; \\
   0\le \phi_i\le 1; \\
    \phi_i=0\ \ \text{outside of}\ \ \Omega_{i-2}.
  \end{cases}
  \end{align*}
For each \(k \ge 4\), define
\begin{align}\label{eq-def-psik}
\psi_k := (f_4\phi_4) \vee \cdots \vee (f_k\phi_k).
\end{align}
Then \(\{\psi_k\}_{k =4}^\infty\) is an increasing sequence. Moreover, using \eqref{eq-norm-fj-f}, we obtain
 \begin{align}\label{eq-norm-psik-f}
&\|\psi_k-f\|_{L^{p_2}(L^{p_1})(M)}\\
&\quad\le \left\|\sum_{i=4}^k |f_i\phi_i -f|\right\|_{L^{p_2}(L^{p_1})(M)}\notag\\
&\quad\le \sum_{i=4}^k \left(\left\|(f_i-f)\phi_i \right\|_{L^{p_2}(L^{p_1})(M)}
+ \left\|f(1-\phi_i) \right\|_{L^{p_2}(L^{p_1})(M)}\right)\notag\\
&\quad\le \sum_{i=4}^k \left(\left\|f_i-f\right\|_{L^{p_2}(L^{p_1})(M)}
+ \left\|f {\mathbf 1}_{M\setminus \Omega_{i-3}} \right\|_{L^{p_2}(L^{p_1})(M)}\right)\notag\\
&\quad\le \sum_{i=4}^k(2^{-i}\varepsilon+2^{-i}\varepsilon)\notag\\
&\quad<2\varepsilon. \notag
	\end{align}

For any \(k \ge 4\), note that \(f_k \phi_k \in \mathcal{C}_c^\infty(\Omega_{k-1})\) and, hence, by \eqref{eq-def-psik},
\begin{align}\label{eq-psi-Cc}
\psi_k \in \mathcal{C}_c(\Omega_{k-1})\subset \mathcal{C}_c(\Omega_k) \subset \mathcal{C}_c(M) \subset L^{p_2}(L^{p_1})(M).
\end{align}
Then, for any \(\lambda > 0\), applying Lemma \ref{lem-Gf-continue} yields that
\(G_\lambda^{\Omega_k}(\psi_k) \in \mathcal{C}(\Omega_k)\). In particular, the set
\[
\left\{ x \in M :\ G_\lambda^{\Omega_k}(\psi_k)(x) > 1 - \varepsilon \right\}
\]
is open and contained in \(\Omega_k\).
Noting that \(\{f_k\phi_k\}_{k =4}^\infty\) converges to \(f\) almost everywhere on \(M\), it follows from the Fatou lemma that
\[
\mathbf{1}_K \le G^\Omega f \le \liminf_{k\to\infty} G^{\Omega_k}(f_k\phi_k)
\le \liminf_{k\to\infty} G^{\Omega_k}(\psi_k)
= \liminf_{k\to\infty} \lim_{\lambda\to0_+} G_\lambda^{\Omega_k}(\psi_k).
\]
Consequently, there exists a sequence \(\{\lambda_j\}_{j\in\nn}\) decreasing to \(0\) such that
\[
K \subset \bigcup_{k =4}^\infty\bigcup_{j\in\nn}\left\{x\in M:\  G_{\lambda_j}^{\Omega_k}(\psi_k)(x) > 1-\varepsilon\right\}.
\]
Since the sequence \(\{\psi_k\}_{k=4}^\infty\) is increasing, the mapping \(k \mapsto G_\lambda^{\Omega_k}(\psi_k)\) is increasing as well.
Moreover, for each fixed \(k\), the mapping \(\lambda \mapsto G_\lambda^{\Omega_k}(\psi_k)\) is increasing as \(\lambda\) decreases.
With these observations and the compactness of \(K\), there exist \(k_\varepsilon\) and \(j_\varepsilon\) such that
\[
K \subset \left\{x \in M:\ G_{\lambda_{j_\varepsilon}}^{\Omega_{k_\varepsilon}}(\psi_{k_\varepsilon})(x) > 1 - \varepsilon\right\}.
\]
Further, invoking the fact that \(G_\lambda^\Omega f \le G^\Omega f\) for all \(\lambda > 0\) and all open sets $\Omega$, we obtain
\begin{align}\label{eq-Gpsi>=1}
K \subset \left\{x \in M:\ G^{\Omega_{k_\varepsilon}}(\psi_{k_\varepsilon})(x) > 1 - \varepsilon\right\}.
\end{align}

Using \eqref{eq-psi-Cc}, \eqref{eq-Gpsi>=1}, \eqref{eq-cap}, and \eqref{eq-norm-psik-f}, we have for all \(k \ge k_\varepsilon\) that
    \begin{align*}
		\operatorname{cap}_{p_1,\,p_2}(K;\,\Omega_k)
&\le \operatorname{cap}_{p_1,\,p_2}(K;\,\Omega_{k_\varepsilon})\\
&\le \|(1-\varepsilon)^{-1}\psi_{k_{\varepsilon}}\|_{L^{p_2}(L^{p_1})(M)}^{p_2}\\
		&\le(1-\varepsilon)^{-p_2}\left(\|\psi_{k_{\varepsilon}}-f\|_{L^{p_2}(L^{p_1})(M)}
		+\|f\|_{L^{p_2}(L^{p_1})(M)}\right)^{p_2}\\ &<(1-\varepsilon)^{-p_2}\left(2\varepsilon+\left(\operatorname{cap}_{p_1,\,p_2}(K;\,\Omega)+\varepsilon\right)^{\frac 1{p_2}}\right)^{p_2}.
	\end{align*}
    Due to the arbitrariness of $\varepsilon$, we get
    $$\limsup_{k\to\infty}\,\operatorname{cap}_{p_1,\,p_2}(K;\,\Omega_k)
    \le \operatorname{cap}_{p_1,\,p_2}(K;\,\Omega),$$
        as desired.
\end{proof}

The first equivalent characterization of \(\operatorname{cap}_{p_1,\,p_2}(K;\  \Omega)\) for compact sets \(K \subset \Omega\) is as follows.

\begin{theorem}\label{thm-Ccinfty-cap}
  Let \(M = M_1 \times M_2\) and \(p_1, p_2 \in [1,\infty)\). For any open set \(\Omega \subset M\) and  compact set \(K \subset \Omega\),
\begin{align}\label{eq-cap-Ccinfty}
\operatorname{cap}_{p_1,\,p_2}(K;\ \Omega)=\inf\lf\{\|\phi\|_{L^{p_2}(L^{p_1})(M)}^{p_2}:\
	0\le \phi\in {\mathcal C}_c^\infty(\Omega)\ \ {\rm and}\ \ G^\Omega\phi\ge {\mathbf 1}_K\r\}.
    \end{align}

\end{theorem}

\begin{proof}
The inequality \(\le\) in \eqref{eq-cap-Ccinfty} follows immediately from \eqref{eq-cap}. Thus, we only need to prove the opposite inequality \(\ge\) in \eqref{eq-cap-Ccinfty}.
Indeed, the main step has already been carried out in Theorem \ref{thm-approx-cap}.

Assume without loss of generality that \(\operatorname{cap}_{p_1,\,p_2}(K;\, \Omega)\) is finite.
Following the proof of Theorem \ref{thm-approx-cap}, we now let \(\{\Omega_i\}_{i \in \mathbb N}\) be an exhaustion sequence of \(\Omega\).
Given any \(\varepsilon > 0\), there exists a nonnegative function \(f \in L^{p_2}(L^{p_1})(M)\) such that
\[
G^\Omega f \ge \mathbf{1}_K
\]
and
\[
\| f \|_{L^{p_2}(L^{p_1})(M)}^{p_2} < \operatorname{cap}_{p_1,\,p_2}(K; \Omega) + \varepsilon.
\]
From \eqref{eq-psi-Cc}, \eqref{eq-Gpsi>=1}, and \eqref{eq-norm-psik-f},
it follows that there exists a function \(\psi_{k_\varepsilon} \in \mathcal{C}_c(\Omega_{k_\varepsilon-1})\) such that
\[
\|\psi_{k_\varepsilon} - f\|_{L^{p_2}(L^{p_1})(M)} < 2\varepsilon
\]
and
\[
K \subset \left\{ x \in \Omega_{k_\varepsilon} :\ G^{\Omega_{k_\varepsilon}}(\psi_{k_\varepsilon})(x) > 1 - \varepsilon \right\}.
\]
For any precompact open subset \(\Omega \subset M\), it is known that \(G^\Omega\) is bounded on \(L^p(\Omega)\) for all \(p \in [1,\infty]\); see \cite[Exercise~13.26]{Grigoryan2009book}.
In particular, it makes sense to set
$$a:=1+\|G^{\Omega_{k_\varepsilon}}\|_{L^\infty(\Omega_{k_\varepsilon})\to L^\infty(\Omega_{k_\varepsilon})}.$$
For \(\psi_{k_\varepsilon} \in \mathcal{C}_c(\Omega_{k_\varepsilon-1}) \subset \mathcal{C}_c(M)\), applying Lemma \ref{lem-dense-Cc} yields a function \(\phi_\varepsilon \in \mathcal{C}_c^\infty(M)\) such that
 \[
\|\phi_\varepsilon  - \psi_{k_\varepsilon}\|_{L^\infty(M)} < a^{-1}\varepsilon
\]
and
\[
\|\phi_\varepsilon  - \psi_{k_\varepsilon}\|_{L^{p_2}(L^{p_1})(M)}< \varepsilon.
\]
 Since $\Omega_{k_\varepsilon-1}\Subset \Omega_{k_\varepsilon}\Subset \Omega$, we may even assume that $\supp \phi_\varepsilon \subset\Omega_{k_\varepsilon}$.
From the first inequality, we obtain
\begin{align*}
\|G^{\Omega_{k_\varepsilon}}(\phi_\varepsilon )- G^{\Omega_{k_\varepsilon}}(\psi_{k_\varepsilon})\|_{L^\infty(M)}
&= \|G^{\Omega_{k_\varepsilon}}(\phi_\varepsilon - \psi_{k_\varepsilon})\|_{L^\infty(\Omega_{k_\varepsilon})}\\
&
\le\|G^{\Omega_{k_\varepsilon}}\|_{L^\infty(\Omega_{k_\varepsilon})\to L^\infty(\Omega_{k_\varepsilon})} \|\phi_\varepsilon  - \psi_{k_\varepsilon}\|_{L^\infty(\Omega_{k_\varepsilon})} < \varepsilon,
\end{align*}
and therefore
\[
K \subset \left\{ x \in \Omega_{k_\varepsilon} :\ G^{\Omega_{k_\varepsilon}}(\phi_\varepsilon )(x) > 1 - 2\varepsilon \right\}\subset \left\{ x \in \Omega :\ G^{\Omega}(\phi_\varepsilon )(x) > 1 - 2\varepsilon \right\}.
\]
Meanwhile, we have
\[
\|\phi_\varepsilon - f\|_{L^{p_2}(L^{p_1})(M)}
\le \|\phi_\varepsilon - \psi_{k_\varepsilon}\|_{L^{p_2}(L^{p_1})(M)}
+ \|\psi_{k_\varepsilon} - f\|_{L^{p_2}(L^{p_1})(M)} < 3\varepsilon.
\]
Consequently,
\begin{align*}
&\inf \left\{ \|\phi\|_{L^{p_2}(L^{p_1})(M)}^{p_2} :\
0 \le \phi \in \mathcal{C}_c^\infty(\Omega) \text{ and } G^\Omega \phi \ge \mathbf{1}_K \right\} \\
&\quad \le \big\|(1-2\varepsilon)^{-1} \phi_\varepsilon \big\|_{L^{p_2}(L^{p_1})(M)}^{p_2} \\
&\quad \le (1-2\varepsilon)^{-p_2}
      \left( \|\phi_\varepsilon - f\|_{L^{p_2}(L^{p_1})(M)}
            + \|f\|_{L^{p_2}(L^{p_1})(M)} \right)^{p_2} \\
&\quad < (1-2\varepsilon)^{-p_2}
      \left( 3\varepsilon + \big( \operatorname{cap}_{p_1,\,p_2}(K;\Omega) + \varepsilon \big)^{\frac 1{p_2}} \right)^{p_2}.
\end{align*}
Finally, since \(\varepsilon > 0\) is arbitrary, we conclude that
   \begin{align*}
\inf\lf\{\|\phi\|_{L^{p_2}(L^{p_1})(M)}^{p_2}:\
	0\le \phi\in {\mathcal C}_c^\infty(\Omega)\ \ {\rm and}\ \ G^\Omega\phi\ge {\mathbf 1}_K\r\}\le
\operatorname{cap}_{p_1,\,p_2}(K;\Omega),
	\end{align*}
as desired.
\end{proof}

\begin{remark}\label{rem-capK-finite}
For any compact set $K \subset \Omega$, the formula \eqref{eq-cap-Ccinfty} implies
\[
\operatorname{cap}_{p_1,\, p_2}(K;\ \Omega) < \infty.
\]
To see this, choose an arbitrary nonzero nonnegative function $\phi \in \mathcal{C}_c^\infty(\Omega)$.
If $G^\Omega \phi \equiv \infty$ on $\Omega$, then $\phi$ itself can serve as a test function in \eqref{eq-cap-Ccinfty}, thereby yielding
\[
\operatorname{cap}_{p_1,\, p_2}(K;\ \Omega) \le \|\phi\|_{L^{p_2}(L^{p_1})(M)}^{p_2} < \infty.
\]
Otherwise, if $G^\Omega \phi(x) < \infty$ for some $x \in \Omega$, then by Lemma \ref{lem-FG3.4} we have $G^\Omega \phi \in \mathcal{C}^\infty(\Omega)$.
In particular, since $K \subset \Omega$ is compact, $G^\Omega \phi$ attains a minimum value $m\in( 0,\infty)$ on $K$.
Consequently, $m^{-1}\phi$ serves as a valid test function in \eqref{eq-cap-Ccinfty}, giving
\[
\operatorname{cap}_{p_1,\, p_2}(K;\ \Omega) \le m^{-p_2} \|\phi\|_{L^{p_2}(L^{p_1})(M)}^{p_2} < \infty.
\]
Thus, in either case, we have $\operatorname{cap}_{p_1,\, p_2}(K;\ \Omega)<\infty$.
\end{remark}

As a consequence of Theorem \ref{thm-Ccinfty-cap}, we have the following duality characterization of \(\operatorname{cap}_{p_1,\,p_2}(K;\Omega)\) for compact sets \(K \subset \Omega\).

\begin{theorem}
\label{thm-cap=mu}
 Let \(M = M_1 \times M_2\) and  \(p_1, p_2 \in [1,\infty)\). For any open set \(\Omega \subset M\) and  compact set \(K \subset \Omega\),
 \begin{align}\label{eq-cap-mu}
\operatorname{cap}_{p_1,\,p_2}(K;\ \Omega)=
\sup\lf\{\mu(K)^{p_2}:\ \mu\in\cm^+(K),\ \|G^\Omega\mu\|_{L^{p_2'}(L^{p_1'})(M)}\le1\r\},
\end{align}
where $\cm^+(K)$ is the set of all positive Radon measures
supported on $K$.
 \end{theorem}

\begin{proof} Note that the duality relation \eqref{eq-norm} holds for the mixed-norm Lebesgue space \(L^{p_2}(L^{p_1})(M)\).
By Lemma \ref{lem-dense}, the space \({\mathcal C}_c^\infty(M)\) is dense in \(L^{p_2}(L^{p_1})(M)\).
Moreover, if \(0\le \phi\in {\mathcal C}_c^\infty(M)\), then Lemma \ref{lem-FG3.4} implies that
\(G^\Omega \phi\in {\mathcal C}^\infty(\Omega)\).

Combining these three  facts, together with \eqref{eq-cap-Ccinfty} and the Minimax theorem in  \cite[Theorem~2.4.1]{AdamsHedberg1996Book},
we follow exactly \cite[Proposition~4.6]{JinLiuWuXiao2026JGA}
and obtain  \eqref{eq-cap-mu}.
\end{proof}

\subsection{Nonlinear mixed-potential ${\mathcal G}_{p_1,p_2}(f)$}\label{sec-mix-cap-poten}

For the nonlinear mixed-potential defined in Definition~\ref{def-mix-poten}, we first examine it from the perspective of PDE theory.

\begin{proposition}\label{prop-PDE}
Let $M = M_1 \times M_2$ and $\Omega=\Omega_1\times\Omega_2$, where every $\Omega_i\Subset  M_i$ is a precompact open set.
For any \(p_1, p_2 \in (1,\infty)\) and $0\le f\in {\mathcal C}_c^\infty(\Omega)$,  the nonlinear mixed-potential
$$u={\mathcal G}_{p_1,\, p_2}^{\Omega}(f)= \int_{M_2}  \int_{M_1}\|G^\Omega (f)(\cdot,\, y_2)\|_{L^{p_1'}(M_1)}^{p_2'-p_1'} G^\Omega(x;\, y)
  \left[G^\Omega (f)(y)\right]^{p_1'-1}\,dV_{1}(y_1)\,dV_{2}(y_2)$$
solves
\begin{align*}
\varDelta_M\left( |\varDelta_M u|^{p_1-2}  \varDelta_M u \left\|\varDelta_M u(\cdot, \, y_2)\right\|_{L^{p_1}(M_1)}^{p_2-p_1}\right)=f,
\end{align*}
where $\varDelta_M$ denotes the Laplace--Beltrami operator on $M$ as defined in \eqref{eq-DeltaM}.
\end{proposition}

\begin{proof}
For simplicity, let $\varDelta:=\varDelta_M$.
By \eqref{eq-GOmega} and $0\le f\in {\mathcal C}_c^\infty(\Omega)$, we deduce from the definition of $u$ that for any $y=(y_1,y_2)\in M$,
\begin{align}\label{eq1-Deltau}
  -\varDelta  u (y) =\|G^\Omega (f)(\cdot,\, y_2)\|_{L^{p_1'}(M_1)}^{p_2'-p_1'}
  \left[G^\Omega (f)(y)\right]^{p_1'-1}.
\end{align}
Since $f\ge0$, it follows directly that
\begin{align*}
  -\varDelta  u (y)\ge 0.
  \end{align*}
  Using $p_1(p_1'-1)=p_1'$ and $(p_2'-p_1')p_1+p_1'=p_1(p_2'-1)$, we obtain
  \begin{align*}
    \int_{M_1} \left[-\varDelta  u (y)\right]^{p_1}\,dV_1(y_1)
    &= \|G^\Omega (f)(\cdot,\, y_2)\|_{L^{p_1'}(M_1)}^{(p_2'-p_1')p_1}
  \int_{M_1} \left[G^\Omega (f)(y_1, y_2)\right]^{p_1(p_1'-1)}\,dV_1(y_1)\\
  &= \|G^\Omega (f)(\cdot,\, y_2)\|_{L^{p_1'}(M_1)}^{(p_2'-p_1')p_1}
  \int_{M_1} \left[G^\Omega (f)(y_1, y_2)\right]^{p_1'}\,dV_1(y_1)\\
  &= \|G^\Omega (f)(\cdot,\, y_2)\|_{L^{p_1'}(M_1)}^{(p_2'-p_1')p_1+p_1'}\\
  &= \|G^\Omega (f)(\cdot,\, y_2)\|_{L^{p_1'}(M_1)}^{p_1(p_2'-1)},
  \end{align*}
  that is,
  $$\left\|\varDelta u(\cdot, \, y_2)\right\|_{L^{p_1}(M_1)}= \|G^\Omega (f)(\cdot,\, y_2)\|_{L^{p_1'}(M_1)}^{p_2'-1}.$$
 Inserting this last equality into \eqref{eq1-Deltau} yields
  \begin{align*}
  -\varDelta  u (y) =\left(\left\|\varDelta u(\cdot, y_2)\right\|_{L^{p_1}(M_1)}\right)^{\frac{p_2'-p_1'}{p_2'-1}}
  \left[G^\Omega (f)(y)\right]^{p_1'-1}
\end{align*}
  and, hence,
 \begin{align*}
 \left[ -\varDelta  u (y)\right]^{p_1-1}
 &=\left(\left\|\varDelta u(\cdot,  \, y_2)\right\|_{L^{p_1}(M_1)}\right)^{\frac{p_2'-p_1'}{p_2'-1}\cdot(p_1-1)}
  \left[G^\Omega (f)(y)\right]^{(p_1'-1)(p_1-1)}\\
 &= \left\|\varDelta u(\cdot, \, y_2)\right\|_{L^{p_1}(M_1)}^{p_1-p_2}
 G^\Omega (f)(y),
\end{align*}
since
 \begin{align*}
\frac{p_2'-p_1'}{p_2'-1}\cdot(p_1-1)
&= \frac{(p_2'-1)-(p_1'-1)}{p_2'-1}\cdot(p_1-1)\\
&
= (p_1-1)-(p_2-1)
=p_1-p_2.
\end{align*}
Consequently, we obtain
 \begin{align*}
 \left[ -\varDelta  u (y)\right]^{p_1-1}
\left\|\varDelta u(\cdot, \, y_2)\right\|_{L^{p_1}(M_1)}^{p_2-p_1}
=
 G^\Omega (f)(y),
\end{align*}
that is,
$$-\varDelta \left( \left[ -\varDelta  u (y)\right]^{p_1-1} \left\|\varDelta  u(\cdot,  \, y_2)\right\|_{L^{p_1}(M_1)}^{p_2-p_1}\right)=f.$$
This gives the desired conclusion since $-\varDelta  u\ge 0$.
\end{proof}

Next, we  prove an extension of \cite[Theorem 2.2.7]{AdamsHedberg1996Book} to the mixed-capacity setting.

\begin{theorem}\label{thm-equiv-mu}
Let \(M = M_1 \times M_2\),  \(\Omega \subset M\) be an open set, and \(p_1, p_2 \in (1,\infty)\).
Then, for any compact set $K\subset  \Omega$,
there exist a nonnegative function $f^K$  and a Radon measure $\mu^K\in {\mathcal M}^+(K)$ such that the following hold:
\begin{enumerate}[\rm (i)]
  \item The function $f^K$ and the measure $\mu^K$ satisfy
  \begin{align}\label{eq-x1}
f^K(x)= f^K(x_1, x_2)=  \|G^\Omega (\mu^K)(\cdot,\, x_2)\|_{L^{p_1'}(M_1)}^{p_2'-p_1'}
  \left[G^\Omega (\mu^K)(x_1, x_2)\right]^{p_1'-1}
\end{align}
and
\begin{align}\label{eq-x2}
\mu^K(K)=\|G^\Omega (\mu^K)\|_{L^{p_2'}(L^{p_1'})(M)}^{p_2'}
&= \int_M G^\Omega(f^K)\,d\mu^K\\
&=\|f^K\|_{L^{p_2}(L^{p_1})(M)}^{p_2}= \operatorname{cap}_{p_1,\,p_2}(K;\ \Omega).\notag
\end{align}

  \item For $\operatorname{cap}_{p_1,\,p_2}$-q.e. $x\in K$,
\begin{align}\label{eq-x4}
G^\Omega(f^K)(x)={\mathcal G}_{p_1,\,p_2}^\Omega(\mu^K)(x)\ge  1.
\end{align}

  \item
For any  $x\in\supp \mu^K$,
\begin{align}\label{eq-x3}
G^\Omega(f^K)(x)={\mathcal G}_{p_1,\,p_2}^\Omega(\mu^K)(x)\le  1.
\end{align}

\end{enumerate}

\end{theorem}

\begin{proof}
If \(\operatorname{cap}_{p_1,\,p_2}(K;\ \Omega)=0\), then it suffices to take \(\mu^K \equiv 0\) and \(f^K \equiv 0\). Below, we assume without loss of generality that \(\operatorname{cap}_{p_1,\,p_2}(K;\ \Omega)\neq 0\).

Given \eqref{eq-x1}, a direct calculation yields
$
G^\Omega(f^K) = \mathcal{G}_{p_1, p_2}(\mu^K)
$
and the equalities
\[
\|G^\Omega(\mu^K)\|_{L^{p_2'}(L^{p_1'})(M)}^{p_2'}
= \int_M G^\Omega(f^K) \, d\mu^K
= \|f^K\|_{L^{p_2}(L^{p_1})(M)}^{p_2}
\]
in \eqref{eq-x2}. The remainder of the proof is divided into five steps.

\medskip

{\it Step 1:\, existence of a measure $\mu_0$ attaining the supremum in \eqref{eq-cap-mu}.}
By \eqref{eq-cap-mu}, there exists a sequence of Radon measures $\{\mu_n\}_{n \in \mathbb{N}} \subset \mathcal{M}^+(K)$ such that for every $n \in \mathbb{N}$,
\[
\|G^\Omega \mu_n\|_{L^{p_2'}(L^{p_1'})(M)} \le 1
\]
and
\[
\lim_{n \to \infty} \mu_n(K) = \operatorname{cap}_{p_1,\, p_2}(K;\ \Omega)^{\frac 1 {p_2}}.
\]
Since $\mathcal{C}(K)$ is separable and $\mathcal{M}^+(K) = (\mathcal{C}(K))^*$, the Banach--Alaoglu theorem implies the existence of a measure $\mu_0 \in \mathcal{M}^+(K)$ such that
\[
\mu_n \to \mu_0 \quad \text{in the weak-* topology}.
\]
Consequently,
\begin{align}\label{eq-mu0}
\mu_0(K) = \lim_{n \to \infty} \mu_n(K) = \operatorname{cap}_{p_1,\, p_2}(K;\ \Omega)^{\frac 1 {p_2}}.
\end{align}
Next, we will show that
\begin{align}\label{eq-Gmu0}
\|G^\Omega \mu_0\|_{L^{p_2'}(L^{p_1'})(M)} = 1.
\end{align}

To prove the inequality $\le$ in \eqref{eq-Gmu0}, observing that $G^\Omega \mu_0$ vanishes outside $\Omega$ and $\mathcal{C}_c^\infty(M)$ is dense in $L^{p_2}(L^{p_1})(M)$, we deduce from the duality norm equality in \eqref{eq-norm} that
\begin{align}\label{eq-norm-Gmu0}
\|G^\Omega \mu_0\|_{L^{p_2'}(L^{p_1'})(M)}
= \sup\left\{ \left| \int_M \phi \, G^\Omega \mu_0 \, dV \right|:\
\phi \in \mathcal{C}_c^\infty(\Omega), \ \|\phi\|_{L^{p_2}(L^{p_1})(M)} \le 1 \right\}.
\end{align}
Let $\phi \in \mathcal{C}_c^\infty(M)$ with $\|\phi\|_{L^{p_2}(L^{p_1})(M)} \le 1$.
By  the H\"older inequality  \eqref{eq-Holder}, we obtain
\[
\int_M G^\Omega \phi \, d\mu_n
= \int_M \phi \, G^\Omega \mu_n \, dV
\le \|\phi\|_{L^{p_2}(L^{p_1})(M)} \, \|G^\Omega \mu_n\|_{L^{p_2'}(L^{p_1'})(M)}
\le 1.
\]
In particular, this shows that $G^\Omega \phi(x) < \infty$ for some $x \in \Omega$, so  Lemma \ref{lem-FG3.4} can be applied to deduce that $G^\Omega \phi \in \mathcal{C}(M)$.  Consequently,
we have
\[
\int_M \phi \, G^\Omega \mu_0 \, dV
= \int_M G^\Omega \phi \, d\mu_0
= \lim_{n \to \infty} \int_M G^\Omega \phi \, d\mu_n
= \lim_{n \to \infty} \int_M \phi \, G^\Omega \mu_n \, dV.
\]
From this, together with \eqref{eq-Holder} and $\|G^\Omega \mu_n\|_{L^{p_2'}(L^{p_1'})(M)} \le 1$, it follows that
 \begin{align*}
\int_M \phi \, G^\Omega \mu_0 \, dV
&= \lim_{n \to \infty} \int_M \phi \, G^\Omega \mu_n \, dV\\
&
\le \lim_{n \to \infty} \|\phi\|_{L^{p_2}(L^{p_1})(M)} \, \|G^\Omega \mu_n\|_{L^{p_2'}(L^{p_1'})(M)}\\
&
\le 1.
\end{align*}
Substituting this last estimate into \eqref{eq-norm-Gmu0} yields
\[
\|G^\Omega \mu_0\|_{L^{p_2'}(L^{p_1'})(M)} \le 1,
\]
which establishes the inequality $\le$ in \eqref{eq-Gmu0}.

To obtain the inequality $\ge$ in \eqref{eq-Gmu0}, we consider the normalized measure
\[
\bar \mu_0 := \frac{\mu_0}{\|G^\Omega \mu_0\|_{L^{p_2'}(L^{p_1'})(M)}}.
\]
Since
\(
\|G^\Omega \bar \mu_0\|_{L^{p_2'}(L^{p_1'})(M)} = 1,
\)
it follows from \eqref{eq-cap-mu} and \eqref{eq-mu0} that
\[
\bar \mu_0(K) \le \operatorname{cap}_{p_1,\, p_2}(K;\ \Omega)^{\frac 1 {p_2}}= \mu_0(K).
\]
This implies the desired inequality
\[
\|G^\Omega \mu_0\|_{L^{p_2'}(L^{p_1'})(M)} \ge 1.
\]

\medskip

{\it Step 2:\,  existence of $\mu^K \in \mathcal{M}^+(K)$ satisfying the first equality in \eqref{eq-x2}.}
Define
\[
\mu^K := \operatorname{cap}_{p_1,\, p_2}(K;\ \Omega)^{\frac{1}{p_2'}} \, \mu_0,
\]
where $\mu_0$ is as in {\it Step 1}.
By \eqref{eq-mu0} and \eqref{eq-Gmu0}, we obtain
\[
\mu^K(K) = \operatorname{cap}_{p_1,\, p_2}(K;\ \Omega)^{\frac{1}{p_2'}} \, \mu_0(K) = \operatorname{cap}_{p_1,\, p_2}(K;\ \Omega)
\]
and
\[
\|G^\Omega \mu^K\|_{L^{p_2'}(L^{p_1'})(M)}^{p_2'}
= \operatorname{cap}_{p_1,\, p_2}(K;\ \Omega) \, \|G^\Omega \mu_0\|_{L^{p_2'}(L^{p_1'})(M)}^{p_2'}
= \operatorname{cap}_{p_1,\, p_2}(K;\ \Omega).
\]
Thus,
\begin{align}\label{eq-muK-cap}
\mu^K(K) = \|G^\Omega \mu^K\|_{L^{p_2'}(L^{p_1'})(M)}^{p_2'} = \operatorname{cap}_{p_1,\, p_2}(K;\ \Omega).
\end{align}

\medskip

{\it Step 3:\, existence of $f^K \in L^{p_2}(L^{p_1})(M)$ satisfying \eqref{eq-x4} and the last equality in \eqref{eq-x2}.}
By \eqref{eq-cap-Ccinfty},  there exists a  nonnegative sequence $\{\phi_n\}_{n\in\nn}\subset {\mathcal C}_c^\infty(M)$ such that for any $n\in\nn$,
$$G^\Omega\phi_n\ge {\mathbf 1}_K$$
and
$$
\operatorname{cap}_{p_1,\,p_2}(K;\ \Omega)\le \|\phi_n\|_{L^{p_2}(L^{p_1})(M)}^{p_2}
<\operatorname{cap}_{p_1,\,p_2}(K;\ \Omega)+2^{-n}.
$$
Based on Remark \ref{rem-capK-finite}, the sequence $\{\phi_n\}_{n \in \mathbb{N}}$ is bounded in $L^{p_2}(L^{p_1})(M)$. By \eqref{eq-duality} and the Banach--Alaoglu theorem, there exists a subsequence $\{\phi_{n_i}\}_{i \in \mathbb{N}}$ converging to some element, denoted by $f^K$, in the weak-$*$ topology of $L^{p_2}(L^{p_1})(M)$. Consequently, for every $h \in L^{p_2'}(L^{p_1'})(M)$,
\[
\lim_{i \to \infty} \langle \phi_{n_i},\, h \rangle = \langle f^K, \,h \rangle.
\]
In particular, this shows that $\{\phi_{n_i}\}_{i \in \mathbb{N}}$ converges to $f^K$ in the weak topology of $L^{p_2}(L^{p_1})(M)$.
Furthermore, for any $k \in \mathbb{N}$, applying the Mazur theorem (see, e.g., \cite[Theorem~3.13]{Rudin1991FAbook}) to the sequence $\{\phi_{n_i}\}_{i=k}^\infty$, we obtain a function $f_k$ which is a finite convex combination of $\{\phi_{n_i}\}_{i=k}^\infty$, i.\,e.,
\[
f_k = \sum_{i=k}^{N_k} \lambda_i^k \phi_{n_i},
\]
where $N_k \in \mathbb{N}$, each $\lambda_i^k \in [0,1]$, and $\sum_{i=k}^{N_k} \lambda_i^k = 1$, such that
\begin{align*}
\|f_k - f^K\|_{L^{p_2}(L^{p_1})(M)} < 2^{-k}.
\end{align*}
Thus,
\begin{align}\label{eq-fkfK}
\lim_{k\to \infty} \|f_k-f^K\|_{L^{p_2}(L^{p_1})(M)}
 =0.
\end{align}

Since \eqref{eq-fkfK} holds, an application of Theorem \ref{Thm-Choquet}(iv) yields a subsequence $\{f_{k_i}\}_{i \in \mathbb{N}}$ such that
\begin{align}\label{eq-Gfk-limit}
\lim_{i \to \infty} G^\Omega f_{k_i} = G^\Omega f^K \qquad \operatorname{cap}_{p_1,\, p_2}\text{-q.e. on } M.
\end{align}
Moreover, for any $k \in \mathbb{N}$, it follows from the definition of $f_k$ that
\begin{align}\label{eq-Gfk}
G^\Omega f_k = \sum_{i=k}^{N_k} \lambda_i^k \, G^\Omega \phi_{n_i} \ge \mathbf{1}_K.
\end{align}
Combining \eqref{eq-Gfk-limit} and \eqref{eq-Gfk} gives
\[
G^\Omega f^K = \lim_{i \to \infty} G^\Omega f_{k_i} \ge \mathbf{1}_K \qquad \operatorname{cap}_{p_1,\, p_2}\text{-q.e. on } M,
\]
which establishes \eqref{eq-x4}.

Next, we show the last equality in \eqref{eq-x2}, that is,
\begin{align}\label{eq-cap-fKnorm}
\operatorname{cap}_{p_1,\, p_2}(K;\ \Omega)^{\frac 1{p_2}} = \|f^K\|_{L^{p_2}(L^{p_1})(M)}.
\end{align}
Indeed, on the one hand, from \eqref{eq-fkfK} and \eqref{eq-Gfk}, we deduce
\begin{align*}
\operatorname{cap}_{p_1,\, p_2}(K;\ \Omega)^{\frac 1{p_2}}
\le \lim_{k \to \infty} \|f_k\|_{L^{p_2}(L^{p_1})(M)}
= \|f^K\|_{L^{p_2}(L^{p_1})(M)}.
\end{align*}
On the other hand, by  \eqref{eq-fkfK} and the choice of $\{\phi_n\}_{n\in\nn}$,
we have
\begin{align*}
\|f^K\|_{L^{p_2}(L^{p_1})(M)}
&= \lim_{k \to \infty} \|f_k\|_{L^{p_2}(L^{p_1})(M)} \\
&= \lim_{k \to \infty} \Bigg\| \sum_{i=k}^{N_k} \lambda_i^k \phi_{n_i} \Bigg\|_{L^{p_2}(L^{p_1})(M)} \\
&\le \lim_{k \to \infty} \sum_{i=k}^{N_k} \lambda_i^k \|\phi_{n_i}\|_{L^{p_2}(L^{p_1})(M)} \\
&= \operatorname{cap}_{p_1,\, p_2}(K;\ \Omega)^{\frac 1{p_2}}.
\end{align*}
Combining the last two estimates yields \eqref{eq-cap-fKnorm}.

\medskip

{\it Step 4:\, proving \eqref{eq-x1} and \eqref{eq-x2}.}
Let $S := \{ x \in K :\ G^\Omega f^K(x) < 1 \}.$
By \eqref{eq-x4}, we have $\operatorname{cap}_{p_1,\, p_2}(S;\, \Omega) = 0$,
which together with \eqref{eq-cap-mu} implies $\mu^K(S) = 0$.
Using this, along with \eqref{eq-muK-cap}, \eqref{eq-cap-fKnorm}, the Fubini  theorem and the H\"older inequality \eqref{eq-Holder},  we obtain
\begin{align}\label{eq-cap-GmufK}
\operatorname{cap}_{p_1,\,p_2}(K;\ \Omega)
=\mu^K(K)
&\le\int_KG^\Omega f^K\,d\mu^K\\
&\le\int_M G^\Omega f^K\,d\mu^K\notag\\
&=\int_Mf^K G^\Omega (\mu^K)\,dV\notag\\
&
\le \|G^\Omega \mu^K\|_{L^{p_2'}(L^{p_1'})(M)}
\|f^K\|_{L^{p_2}(L^{p_1})(M)}\notag\\
&=\operatorname{cap}_{p_1,\,p_2}(K;\ \Omega).\noz
\end{align}
Thus, all inequalities in \eqref{eq-cap-GmufK} must be equalities. In particular,
the equality in the H\"older inequality \eqref{eq-Holder} yields the relationship between $f^K$ and $\mu^K$ given in \eqref{eq-x1}.
Combining \eqref{eq-x1}, \eqref{eq-cap-GmufK}, \eqref{eq-muK-cap}, and \eqref{eq-cap-fKnorm}, we obtain all equalities in \eqref{eq-x2}.

\medskip

{\it Step 5:\, proving \eqref{eq-x3}.}
Suppose that $G^\Omega f^K(x_0) > 1$ for some $x_0 \in \Omega$.
Choose  $\delta \in (0,1)$ such that $G^\Omega f^K(x_0) > 1 + \delta$.
By Corollary \ref{cor-Gf-lowersemi}, the set
$\{ x \in \Omega :\ G^\Omega f^K(x) > 1 + \delta \}$ is open, so there is a neighborhood $O$ of $x_0$ such that for all $x \in O$,
\[
G^\Omega f^K(x) \ge 1 + \delta.
\]
It was shown in Step 4 that $G^\Omega f^K(x) \ge 1$ for $\mu^K$-a. e. $x \in K$.
Combining this with \eqref{eq-x2} yields
\begin{align*}
\mu^K(K) = \int_M G^\Omega f^K \, d\mu^K\ge (1 + \delta) \mu^K(O) + \mu^K(K \setminus O)\ge \delta \mu^K(O) + \mu^K(K).
\end{align*}
This implies $\mu^K(O) = 0$ and, hence, $x_0 \notin \operatorname{supp} \mu^K$.
Consequently, we obtain that $G^\Omega f^K \le 1$ on $\operatorname{supp} \mu^K$, which proves \eqref{eq-x3}.

Summarizing all, we complete the proof of Theorem~\ref{thm-equiv-mu}.
\end{proof}

\section{Mixed-parabolicity and mixed-Liouville property via Green function integrability}\label{s4}

This section is devoted to proving the following theorem, which establishes the equivalence (i) $\Leftrightarrow$ (ii) $\Leftrightarrow$ (iii) in Theorem \ref{thm-main}.

\begin{theorem}\label{thm-para-Green}
  Let \(M = M_1 \times M_2\) and \(p_1, p_2 \in (1,\infty)\). Then, the following assertions are equivalent:
\begin{enumerate}[\rm (i)]
  \item\label{mix-para} $M$ is $L^{p_2}(L^{p_1})$-parabolic.

  \item\label{G=fz-some}   For some $o\in M$ and $r_0\in (0,\infty)$,
  $$
  \left\|G^M(o;\ \cdot) {\mathbf 1}_{M\setminus B(o,\, r_0)}(\cdot)\right\|_{L^{p_2'}(L^{p_1'})(M)}=\infty.
  $$

    \item\label{G=fz-all}  For all $x\in M$ and all $r\in (0,\infty)$,
  $$
  \left\|G^M(x;\ \cdot) {\mathbf 1}_{M\setminus B(x,\, r)}(\cdot)\right\|_{L^{p_2'}(L^{p_1'})(M)}=\infty.
  $$

    \item\label{mix-Liouv} $M$ admits the \emph{\(L^{p_2'}(L^{p_1'})\)-Liouville property}.

\end{enumerate}
\end{theorem}

\subsection{Preliminaries} \label{ss4.1}

A fundamental tool that we will repeatedly use throughout the proof of Theorem \ref{thm-para-Green} is the following comparison property for Green functions (see also \cite[Lemma~3.4]{GrigoryanPessoaSetti2025}).

\begin{lemma}\label{lem-EHIG}
Suppose that $(M, d, V)$ is a noncompact geodesically complete Riemannian manifold. Assume that $G^M\not\equiv\infty$. Then, for any ball $B(x,r)\subset M$ with $x\in M$ and $r\in(0,\infty)$, there exists a positive constant $C=C(x,r)$ such that for all $z,w\in  B(x,r)$ and $y\notin B(x,2r)$,
\begin{align}\label{eq-compa-G}
 C^{-1}G^M(z;\,y)\le G^M(w;\,y) \le C G^M(z;\,y)
\end{align}
\end{lemma}

\begin{proof}
If $G^M \not\equiv \infty$, then $G^M(x_0;\, y_0) \neq \infty$ for some $x_0, y_0 \in M$. By Lemma~\ref{lem-G-property}, this implies that $G^M(z;\, y) \neq \infty$ for arbitrary distinct $z, y \in M$.
Moreover,   $G^M(\cdot;\,y)$
is harmonic in \(M \setminus \{y\}\).

Let us recall the \emph{Harnack inequality} established in \cite[Theorem~13.10]{Grigoryan2009book}: for any compact set $K \subset \Omega \setminus \{x\}$, there exists a positive constant $C(K)$, depending only on \(K\) and the geometry of \(\Omega\), such that for any nonnegative harmonic function \(f\) on \(\Omega\),
\begin{equation}\label{eq-EHI-G}
\sup_{z \in K} f(z) \le C \inf_{z \in K} f(z).
\end{equation}
Then, inequality \eqref{eq-compa-G} follows from \eqref{eq-EHI-G} by taking \(f(z) = G^M(z;\, y)\) and letting \(K\) be the closure of the geodesic ball \(B(x,r)\).
\end{proof}

The following minimum/maximum principle is taken from \cite[Lemma~4.1]{GrigoryanHu2014CJM}, which was proved in the abstract framework of metric measure spaces. The Euclidean case $M = \mathbb{R}^n$ appears in \cite[Theorem~8.1]{GilbargTrudinger2001book}. For Riemannian manifolds with $u \in C(\overline{\Omega}) \cap C^2(\Omega)$, we refer to \cite[Corollaries~8.14 and~8.16]{Grigoryan2009book}.

\begin{lemma}\label{EMinP}
(Minimum/Maximum principle) Let $(M,d, V)$ be a noncompact geodesically complete Riemannian manifold. Suppose that $\Omega, \Omega_1,\Omega_2$ are precompact open sets  such that $$\Omega_1\Subset \Omega\Subset \Omega_2\Subset M.$$
If $u\ge 0$ on $M$ and $u\in W^{1,2}(\Omega)$ is  superharmonic (resp. subharmonic) in $\Omega$, then
  $$
\einf_{z \in \Omega} u(z) \ge  \einf_{z \in \Omega_2\setminus \Omega_1} u(z)
  \qquad
  \left(\text{resp.} \quad \esup_{z \in \Omega} u(z) \le  \esup_{z \in \Omega_2\setminus \Omega_1} u(z)\right).
  $$
  Moreover, if $u$ is continuous in a neighborhood of $\partial\Omega:=\overline\Omega\setminus\Omega$, then the above inequalities can
be replaced by
 $$
\einf_{z \in \overline\Omega} u(z) =  \inf_{z \in \partial\Omega} u(z)
  \qquad
  \left(\text{resp.} \quad\esup_{z \in \overline\Omega} u(z) =  \sup_{z \in \partial\Omega} u(z)\right).
  $$
\end{lemma}

\begin{lemma}
(The exterior maximum principle) Suppose that $(M,d, V)$ is a noncompact geodesically complete Riemannian manifold. Then, for any $r\in(0,\infty)$ and $x\in M$,
\begin{equation}\label{eq-EMP-Gmax}
\sup_{y \in M\setminus B(x,\,r)} G^M(x;\, y) =  \sup_{y \in \partial B(x,\,r)} G^M(x;\, y).
\end{equation}

\end{lemma}
\begin{proof}

By \eqref{eq-GkG}, the Green function $G^M(x;\,y)$ can be obtained by the limit of Dirichlet kernel $G^{\Omega_k}(x;\,y)$ with $\Omega_k$ being an exhaustion of $M$, and by Lemma ~\ref{EMinP}, we have
\begin{align*}
\sup_{y \in \Omega_k\setminus B(x,\,r)} G^{\Omega_k}(x;\, y) =  \sup_{y \in \partial B(x,\,r)} G^{\Omega_k}(x;\, y)
\end{align*}
Letting $k\to\infty$ yields 
\begin{align*}
\sup_{y \in M\setminus B(x,\,r)} G^M(x;\, y) =\lim_{k\to\infty} \sup_{y \in \Omega_k\setminus B(x,\,r)} G^{\Omega_k}(x;\, y)=  \sup_{y \in \partial B(x,\,r)} G^M(x;\, y),
\end{align*}
which completes the proof of \eqref{eq-EMP-Gmax}.
\end{proof}

\begin{lemma}(Comparison principle) \label{lem-compari}
Let $(M,d, V)$ be a noncompact geodesically complete Riemannian manifold and $G^M\not\equiv\infty$.  Let $ f$ be a nonzero nonnegative superharmonic function in $M$.
  For any $x\in M$, $r\in (0,\infty)$ and precompact open set $\Omega$ satisfying $B(x,r)\Subset  \Omega\Subset  M$, there
  exists a constant $C=C(f,x,r)$ such that
\begin{align}\label{eq-G-comp-local}
G^\Omega(x;\,y)\le C f(y)\quad\text{for a.e.}\ \, y\in \Omega\setminus{\overline{B(x,r)}}.
\end{align}
Consequently,  for any $x\in M$ and  $r\in (0,\infty)$, there
  exists a constant $C=C(x,r)$ such that
\begin{align}\label{eq-G-comp}
G^M(x;\,y)\le  C f(y)\quad\text{for a.e.}\ \, y\in M\setminus{\overline {B(x,r)}}.
\end{align}
\end{lemma}

\begin{proof}
If $x$ and $y$ lie in different connected components of $M$, then $G^M(x;\,y)=0$ and, hence, both \eqref{eq-G-comp-local} and \eqref{eq-G-comp} hold. Thus, we assume without loss of generality that $M$ itself is connected.

According to \cite[Lemma~6.1]{GrigoryanHu2014CJM} or \cite[Exercises~7.19 and 7.30]{Grigoryan2009book},
any nonnegative function $f$ is superharmonic in $M$ if
and only if $P_t^Mf\le f$ in $M$ for all $t\in(0,\infty)$.

Let $ f$ be a nonzero nonnegative superharmonic function in $M$. For any $k\in\nn$, the function $f_k:=\min\{f,k\}$ is also superharmonic in \(M\). This follows from the above fact of \cite[Lemma~6.1]{GrigoryanHu2014CJM} since
$$P_t^Mf_k \le P_t^M f\le f\quad \text{and}\quad P_t^Mf_k \le k P_t^M 1\le k.$$
Further, for any $s,t\in(0,\infty)$, by the semigroup property and
the superharmonicity of $f_k$,  we have
$$
P_s^M (P_t^M f_k) = P_{t+s}^M f_k = P_t^M (P_s^M f_k) \le P_t^M f_k.
$$
Thus, for any $t\in(0,\infty)$, the function $P_t^Mf_k$ remains  superharmonic in $M$.

On a connected manifold, the heat kernel $p_t^M(x,y)$ is strictly positive  for all $t\in(0,\infty)$ and $x,y\in M$ (see \cite[Corollary~8.12]{Grigoryan2009book}). Consequently, $P_t^M f_k(x) > 0$ for every $x\in M$; otherwise this would contradict the assumption that $f$ is nonzero on $M$.

If $f$ is  bounded   on $M$, then  $P_t^M f \in {\mathcal C}^\infty(M)$ (see \cite[Exercise~7.33]{Grigoryan2009book}).
Therefore, by replacing $f$ with $P_t^M f_k$, we may assume without loss of generality that $f$ is bounded, continuous, strictly positive, and superharmonic on $M$.

Under the above reductions, we now proceed to prove \eqref{eq-G-comp-local}.  On the outer boundary \(\partial\Omega\), we have for every \(z \in \partial\Omega\),
\[
G^{\Omega}(x;\,z) = 0 < f(z).
\]
On the inner boundary \(\partial B(x,\,r)\), we use the strict positivity of $f$, together with the continuity of the mappings \(z \mapsto G^{\Omega}(x;\,z)\) and \(z \mapsto f(z)\). Note that $G^M\not\equiv\infty$, which implies
\[
C=C(f,x,r) := \max_{z \in \partial B(x,\,r)} \frac{G^{M}(x;\,z)}{f(z)}\in (0,\infty).
\]
Thus,  for all \(z \in \partial B(x,\,r)\), we have
\[
G^{\Omega}(x;\,z) \le G^{M}(x;\,z) \le C \, f(z).
\]
Since the function \(G^{\Omega}(x;\,\cdot)\) is harmonic and \(f\) is superharmonic on \(\Omega \setminus \overline{B(x,r)}\), it follows that
\[
\cl^{\Omega}\left(G^{\Omega}(x;\,\cdot)\right)=0\le \cl^{\Omega}\left(Cf\right) \quad \text{in} \ \, \Omega \setminus \overline{B(x,r)}.
\]
In other words, the function $Cf(\cdot)-G^{\Omega}(x;\,\cdot)$ is superharmonic in $\Omega \setminus \overline{B(x,r)}$.
From this and Lemma ~\ref{EMinP}, it follows that
$$
\einf_{z \in \Omega \setminus \overline{B(x,\,r)}} \left(Cf(z)-G^{\Omega}(x;\,z)\right) =  \inf_{z \in \partial\Omega \cup \partial B(x,\,r)} \left(Cf(z)-G^{\Omega}(x;\,z)\right)\ge 0,
$$
which gives \eqref{eq-G-comp-local}.

If we have obtained \eqref{eq-G-comp-local}, then
\eqref{eq-G-comp} follows by applying Lemma~\ref{lem-G-property} together with the exhaustion limit \eqref{eq-GkG}.
\end{proof}

\subsection{Proof of (\ref{mix-para}) $\Leftrightarrow$ (\ref{G=fz-some}) in Theorem \ref{thm-para-Green}\,} \label{ss4.2}

\begin{proof}[Proof of \eqref{mix-para} $\Rightarrow$ \eqref{G=fz-some} in Theorem \ref{thm-para-Green}]

Suppose that \eqref{G=fz-some} fails, that is, for every $x \in M$ and every $r \in (0, \infty)$,
\begin{align}\label{eq-G-int}
G^M(x; \,\cdot\,) \mathbf{1}_{M \setminus B} \in L^{p_2'}(L^{p_1'})(M),
\end{align}
with $B:=B(x,r)$.
We will then show that \eqref{mix-para} fails, meaning that $M$ is not $L^{p_2}(L^{p_1})$-parabolic.
To achieve this, it suffices to show that for any compact set $K \subset M$ with $V(K) > 0$, we always have
\[
\operatorname{cap}_{p_1,\, p_2}(K) > 0.
\]
This is equivalent to proving the following implication:
\begin{align}\label{eq-cap=0=>mu=0}
\operatorname{cap}_{p_1,\, p_2}(K) = 0 \quad \Rightarrow \quad V(K) = 0.
\end{align}

Suppose that $K \subset M$ is a non-empty compact set with
$\operatorname{cap}_{p_1,\, p_2}(K) = 0$. By \eqref{eq-cap-Ccinfty}, for any $\varepsilon \in (0,1)$,
there exists some nonnegative $ f \in \mathcal{C}_c^\infty(M)$ such that
$
G^M f \ge \mathbf{1}_K $ and
$\|f\|_{L^{p_2}(L^{p_1})(M)} < \varepsilon$.
Then, applying the Fubini theorem and \eqref{eq-Holder} yields
\begin{align}\label{eq-muK}
V(K) \le \int_K G^Mf \,dV
&= \int_{M} (G^M\mathbf{1}_K)(y) \, f(y) \, dV(y) \\
&\le \bigl\| G^M\mathbf{1}_K \bigr\|_{L^{p_2'}(L^{p_1'})(M)} \;
     \| f \|_{L^{p_2}(L^{p_1})(M)} < \varepsilon \; \bigl\| G^M\mathbf{1}_K \bigr\|_{L^{p_2'}(L^{p_1'})(M)}. \notag
\end{align}
Once we have established
\begin{align}\label{eq-G1K-fin}
\bigl\| G^M\mathbf{1}_K \bigr\|_{L^{p_2'}(L^{p_1'})(M)} < \infty,
\end{align}
then letting $\varepsilon \to 0$ in \eqref{eq-muK} yields $V(K) = 0$.
This will complete the proof of \eqref{eq-cap=0=>mu=0}.

It remains to prove \eqref{eq-G1K-fin}.
To this end, we take a large geodesic ball $B$ in $M$ such that $K \subset B$. Write
\begin{align}\label{eq-G1K}
\bigl\|G^M\mathbf{1}_K\bigr\|_{L^{p_2'}(L^{p_1'})(M)}
&\le \bigl\|(G^M\mathbf{1}_K) \mathbf{1}_{M\setminus 2B}\bigr\|_{L^{p_2'}(L^{p_1'})(M)}
+ \bigl\|(G^M\mathbf{1}_K) \mathbf{1}_{2B}\bigr\|_{L^{p_2'}(L^{p_1'})(M)}
\end{align}
and we will show that each term is finite.

Denote by $o$ the center of the ball $B$.
By \eqref{eq-compa-G}, there exists a positive constant $C$, depending only on $B$, such that for any $z \in K \subset B$ and $x = (x_1, x_2) \in M \setminus 2B$, we have
\begin{align*}
C^{-1} G^M(x;\ o) \le G^M(x;\ z) \le C G^M(x;\ o).
\end{align*}
Using this and \eqref{eq-G-int}, we obtain
\begin{align}\label{eq-G1K-far}
&\lf\|(G^M{\mathbf 1}_K) {\mathbf 1}_{M\setminus 2B}\r\|_{L^{p_2'}(L^{p_1'})(M)}\\
&\quad =\lf[\int_{M_2}\lf(\int_{M_1}{\mathbf 1}_{M\setminus 2B}(x_1, x_2)\lf[\int_K G^M(x_1, x_2;\,z)\,dV(z)\r]^{p_1'}
\,dV_{1}(x_1)\r)^{\frac{p_2'}{p_1'}}\,dV_{2}(x_2)\r]^{\frac1{p_2'}}\notag\\
&\quad \le C V(K) \lf[\int_{M_2}\lf(\int_{M_1}{\mathbf 1}_{M\setminus 2B}(x_1, x_2)\lf[ G^M(x_1, x_2;\,o)\r]^{p_1'}
\,dV_{1}(x_1)\r)^{\frac{p_2'}{p_1'}}\,dV_{2}(x_2)\r]^{\frac1{p_2'}}\notag\\
&\quad = C V(K) \lf\|G^M(\cdot;\ o) \mathbf{1}_{M \setminus 2B}\r\|_{L^{p_2'}(L^{p_1'})(M)}\notag\\
&\quad<\infty. \notag
\end{align}
Next, by \eqref{eq-dg=dinfty}, we may even assume that $2B \subset B_1 \times B_2$, where each $B_i$ is a geodesic ball in $M_i$.
Choose a nonnegative function $\phi \in \mathcal{C}_c^\infty(M)$ such that $\phi \equiv 1$ on $B$.
By \eqref{eq-G-int}, there exists at least one $y \in M$ for which $G^M(x;\, y) < \infty$.
From this and Lemma~\ref{lem-FG3.4}, it follows that $G^M\phi \in \mathcal{C}^\infty(M)$. Consequently,
\[
C := \sup_{x \in \overline{B_1} \times \overline{B_2}} G^M\phi(x) < \infty.
\]
From these facts, we deduce
\begin{align}\label{eq-G1K-near}
\lf\|(G^M{\mathbf 1}_K) {\mathbf 1}_{2B}\r\|_{L^{p_2'}(L^{p_1'})(M)}
&\le C\lf\| {\mathbf 1}_{B_1 \times B_2}\r\|_{L^{p_2'}(L^{p_1'})(M)}= C V_{1}(B_1)^\frac1{p_1'} V_{2}(B_2)^\frac1{p_2'} <\infty.
\end{align}
Substituting \eqref{eq-G1K-far} and \eqref{eq-G1K-near} into \eqref{eq-G1K} yields
\eqref{eq-G1K-fin}.  Thus, we obtain \eqref{eq-cap=0=>mu=0}, which completes the proof of the implication \eqref{mix-para} $\Rightarrow$ \eqref{G=fz-some}.
\end{proof}

\begin{proof}[Proof of \eqref{G=fz-some} $\Rightarrow$ \eqref{mix-para} in Theorem \ref{thm-para-Green}]

By \eqref{G=fz-some}, for some $o \in M= M_1 \times M_2$ and $r_0 \in (0,\infty)$, we have
\begin{align}\label{eq-G-intBcomp}
\bigl\| G^M(o; \,\cdot\,) \, \mathbf{1}_{M \setminus B(o, \, r_0)}(\cdot) \bigr\|_{L^{p_2'}(L^{p_1'})(M)} = \infty.
\end{align}
To show  \eqref{mix-para},  it suffices to prove that for every $R \in (2r_0, \infty)$, the closure of the geodesic ball
\(
B_R := B(o, R)
\)
has zero mixed-capacity, namely,
\begin{align}\label{eq-cap0}
\operatorname{cap}_{p_1,\,p_2}\left(\overline B_R\right) = 0.
\end{align}

If $G^M(o;\, \cdot) \equiv \infty$ on $M$, then \eqref{eq-cap0} follows directly from \eqref{eq-cap-mu}. Hence, in what follows we assume that $G^M(o;\, \cdot) \not\equiv \infty$. In this case, Lemma~\ref{lem-G-property} implies that $G^M(o;\, \cdot)$ is continuous on $M \setminus \{o\}$. Consequently, the function $G^M(o;\, \cdot)$ is bounded on the annulus $B(o, 2R) \setminus B(o, \, r_0)$, which yields
\begin{equation}\label{eq-G-int-annu}
\bigl\| G^M(o; \,\cdot\,) \, \mathbf{1}_{B(o,\, 2R) \setminus B(o, \, r_0)}(\cdot) \bigr\|_{L^{p_2'}(L^{p_1'})(M)} < \infty.
\end{equation}
Combining this with \eqref{eq-G-intBcomp} gives
\begin{equation}\label{eq-G-int-largeBcomp}
\bigl\| G^M(o; \,\cdot\,) \, \mathbf{1}_{M \setminus B(o,\, 2R)}(\cdot) \bigr\|_{L^{p_2'}(L^{p_1'})(M)} = \infty.
\end{equation}

Next, we aim to apply \eqref{eq-G-int-largeBcomp} together with \eqref{eq-cap-mu} to establish \eqref{eq-cap0}.
To this end, by \eqref{eq-compa-G}, there exists a positive constant $C=C(o,R)$, such that for all $y \in \overline B_R=\overline {B(o, R)}$ and $x \in M \setminus B(o, 2R)$,
\[
C^{-1} G^M(x;\, o) \le G^M(x;\, y) \le C G^M(x;\, o).
\]
Consequently, for any $\mu \in \mathcal{M}^+(\overline B_R)$ and any $x \in M \setminus B(o, 2R)$,
\begin{align*}
  G^M\mu(x) & =\int_{\overline B_R} G^M(x;\, y)\,d\mu(y) \simeq \int_{\overline B_R} G^M(x;\ o)\,d\mu(y)\simeq G^M(x;\ o) \mu(\overline B_R).
\end{align*}
If $\mu(\overline B_R) \neq 0$, then this together with \eqref{eq-G-int-largeBcomp} implies that
\[
\bigl\|(G^M\mu)(\cdot) \, \mathbf{1}_{M \setminus B(o,\, 2R)}(\cdot) \bigr\|_{L^{p_2'}(L^{p_1'})(M)}
\simeq \bigl\| G^M(o; \,\cdot\,) \, \mathbf{1}_{M \setminus B(o,\, 2R)}(\cdot) \bigr\|_{L^{p_2'}(L^{p_1'})(M)} = \infty.
\]
In other words, any $\mu \in \mathcal{M}^+(\overline B_R)$ satisfying
$\|G^M\mu\|_{L^{p_2'}(L^{p_1'})(M)} \le 1$ must satisfy $\mu(\overline B_R) = 0$.
From this and \eqref{eq-cap-mu}, it follows that \eqref{eq-cap0} holds for every $R \in (2r_0, \infty)$,
which proves that $M$ is $L^{p_2}(L^{p_1})$-parabolic. Thus, we obtain the implication
\eqref{G=fz-some} $\Rightarrow$ \eqref{mix-para}.
\end{proof}

\subsection{Proof of (\ref{G=fz-some}) $\Leftrightarrow$ (\ref{G=fz-all})  in Theorem \ref{thm-para-Green} } \label{ss4.3}

\begin{proof}[Proof of \eqref{G=fz-some} $\Leftrightarrow$ \eqref{G=fz-all} in Theorem \ref{thm-para-Green}]

It suffices to prove the implication \eqref{G=fz-some} $\Rightarrow$ \eqref{G=fz-all}; the converse implication is trivial.

Assume that
\eqref{G=fz-some} holds, that is, for some $o \in M= M_1 \times M_2$ and $r_0 \in (0,\infty)$, we have
\begin{align}\label{eq-G=fz-ar0}
\bigl\| G^M(o; \,\cdot\,) \, \mathbf{1}_{M \setminus B(o, \, r_0)}(\cdot) \bigr\|_{L^{p_2'}(L^{p_1'})(M)} = \infty.
\end{align}
Then, we claim that for every $R \in (0,\infty)$,
\begin{equation}\label{eq-G=fz-aR}
\bigl\| G^M(o; \,\cdot\,) \, \mathbf{1}_{M \setminus B(o,\, R)}(\cdot) \bigr\|_{L^{p_2'}(L^{p_1'})(M)} = \infty.
\end{equation}
Indeed, when $R \in(0, r_0]$, this follows directly from \eqref{eq-G=fz-ar0}.
For $R \in (r_0, \infty)$,  if $G^M(o;\, \cdot) \equiv \infty$ on $M$, then \eqref{eq-G=fz-aR} is  immediate as well.
In the remaining case where $G^M(o; \,\cdot) \not\equiv \infty$ and $R \in (r_0, \infty)$, the same reasoning that led to \eqref{eq-G-int-annu} gives
\begin{equation*}
\bigl\| G^M(o; \,\cdot\,) \, \mathbf{1}_{B(o,\, R) \setminus B(o, \, r_0)}(\cdot) \bigr\|_{L^{p_2'}(L^{p_1'})(M)} < \infty,
\end{equation*}
which again implies the validity of \eqref{eq-G=fz-aR}. Thus, \eqref{eq-G=fz-aR} holds for every $R \in (0,\infty)$.

Next, we will show that for any other point $x \neq o$ and the special radius $r_x := d(x, o)$,
\begin{equation}\label{eq-G=fz-xRx}
\bigl\| G^M(x; \,\cdot\,) \, \mathbf{1}_{M \setminus B(x,\, r_x)}(\cdot) \bigr\|_{L^{p_2'}(L^{p_1'})(M)} = \infty.
\end{equation}
To this end, set $R_x := 2d(x, o)$.
Since $x \in B(o,\, R_x)$, it follows from \eqref{eq-compa-G} that for any $y \notin B(o, 2R_x)$,
\[
C^{-1} G^M(o;\, y) \le G^M(x;\, y) \le C G^M(o;\, y),
\]
where $C$ is a positive constant independent of $y$ (but depending on $o$ and $R_x$). As a consequence,
\[
G^M(x; \,\cdot\,) \, \mathbf{1}_{M \setminus B(o,\, 2R_x)}(\cdot)
\simeq G^M(o; \,\cdot\,) \, \mathbf{1}_{M \setminus B(o,\, 2R_x)}(\cdot).
\]
Combining this with \eqref{eq-G=fz-aR} gives
\[
\bigl\| G^M(x; \,\cdot\,) \, \mathbf{1}_{M \setminus B(o,\, 2R_x)}(\cdot) \bigr\|_{L^{p_2'}(L^{p_1'})(M)} = \infty.
\]
Observing that $B(x, r_x) \subset B(o, 2R_x)$, we then obtain
\[
\bigl\| G^M(x; \,\cdot\,) \, \mathbf{1}_{M \setminus B(x,\, r_x)}(\cdot) \bigr\|_{L^{p_2'}(L^{p_1'})(M)}
\ge \bigl\| G^M(x; \,\cdot\,) \, \mathbf{1}_{M \setminus B(o,\, 2R_x)}(\cdot) \bigr\|_{L^{p_2'}(L^{p_1'})(M)} = \infty.
\]
Hence, \eqref{eq-G=fz-xRx} holds.

Once \eqref{eq-G=fz-xRx} is established, it follows from \eqref{eq-G=fz-aR} that the special radius $r_x$ in \eqref{eq-G=fz-xRx} can be replaced by an arbitrary radius $r \in (0,\infty)$ while preserving the infiniteness of the mixed-norm. This will finally yield \eqref{G=fz-all}.
\end{proof}

\subsection{Proof of (\ref{G=fz-all}) $\Leftrightarrow$ (\ref{mix-Liouv})  in Theorem \ref{thm-para-Green} } \label{ss4.4}

\begin{proof}[Proof of (\ref{G=fz-all}) $\Rightarrow$ (\ref{mix-Liouv})  in Theorem \ref{thm-para-Green}]

Let us prove it by contradiction. Suppose that $M$ does not possess the $L^{p_2'}(L^{p_1'})$-Liouville property. Then there exists a  nonconstant nonnegative superharmonic function $f \in L^{p_2'}(L^{p_1'})(M)$.
If $G^M\equiv\infty$, then by Theorem \ref{thm-G-para}
the manifold $M$ is classically parabolic and, hence,
every positive bounded superharmonic function on $M$ is constant.  Thus, we must have $G^M\not\equiv\infty$.
By Lemma \ref{lem-compari}, for any given geodesic ball $B(o,r)$ with center $o\in M$ and
radius $r\in (0,\infty)$, there exists a positive constant $C$, depending on $f, o$ and $r$, such that
\begin{align*}\label{eq-G-comp}
G^M(o;\,x)\le  C f(x)\quad\text{for a.e.}\  x\in M\setminus{\overline {B(o,r)}}.
\end{align*}
Consequently,
$$
  \left\|G^M(o;\ \cdot) {\mathbf 1}_{M\setminus B(o,\, r)}(\cdot)\right\|_{L^{p_2'}(L^{p_1'})(M)}\le C  \left\|f {\mathbf 1}_{M\setminus B(o,\, r)}\right\|_{L^{p_2'}(L^{p_1'})(M)}<\infty,
  $$
which contradicts \eqref{G=fz-all}.
Thus, we obtain \eqref{G=fz-all} $\Rightarrow$ \eqref{mix-Liouv}.
\end{proof}

\begin{proof}[Proof of (\ref{mix-Liouv}) $\Rightarrow$ (\ref{G=fz-all})  in Theorem \ref{thm-para-Green}]
Suppose to the contrary that \eqref{G=fz-all} does not hold, that is, there exist a point $o \in M$ and a number $r \in (0,\infty)$ such that
\[
G^M(o; \,\cdot\,) \mathbf{1}_{M \setminus B(o,\, r)} \in L^{p_2'}(L^{p_1'})(M).
\]
Then, we take $y\notin B(o, r)$ such that $G^M(o; \,y)<\infty$. Set $a:=2G^M(o; \,y)$.
For any $x \in M$, define
\[
f(x) := \min\bigl\{a,\; G^M(o;\, x)\bigr\},
\]
which is a nonconstant function in
$L^{p_2'}(L^{p_1'})(M)$. Moreover, $f$ is a positive superharmonic function, which implies that $M$ does not possess the $L^{p_2'}(L^{p_1'})$-Liouville property. Consequently, \eqref{mix-Liouv} fails. This establishes the implication \eqref{mix-Liouv} $\Rightarrow$ \eqref{G=fz-all}.
\end{proof}

\section{Green function integrability and the  nonlinear mixed-potential}\label{s5}

The main aim of this section is to prove that, under the weak radial Harnack-type inequality \eqref{eq-weak-radialharnack}, item (iv) of Theorem \ref{thm-main} is equivalent to items (i), (ii), and (iii) of the same theorem. This follows as a consequence of Theorem \ref{thm-G-poten} below.

\begin{theorem}\label{thm-G-poten}
Let \(M = M_1 \times M_2\) and \(p_1, p_2 \in (1,\infty)\). Suppose that \(M\) satisfies the weak radial Harnack-type inequality \eqref{eq-weak-radialharnack}.
  Then, the following assertions are equivalent:
\begin{enumerate}[\rm (i)]
  \item\label{G=fz-all0}
 For  all $x\in M$ and all $r\in (0,\infty)$,
  $$
  \left\|G^M(x;\ \cdot) {\mathbf 1}_{M\setminus B(x,\, r)}(\cdot)\right\|_{L^{p_2'}(L^{p_1'})(M)}=\infty.
  $$

  \item\label{Gf=fz-all}   For all nonzero $0\le f\in {\mathcal C}_c^\infty(M)$ and all  $x\in M$,
  \begin{align*}
  {\mathcal G}_{p_1,\,p_2}(f)(x)\equiv \infty.
\end{align*}

  \item\label{Gf=fz-some}    For some $x_0\in M$ and some nonzero $0\le f\in {\mathcal C}_c^\infty(M)$,
  \begin{align*}
  {\mathcal G}_{p_1,\,p_2}(f)(x_0)\equiv\infty.
  \end{align*}
\end{enumerate}
\end{theorem}

We remark that \eqref{eq-weak-radialharnack} is only used in the proof of
\eqref{Gf=fz-some} $\Rightarrow$ \eqref{G=fz-all0} in Theorem \ref{thm-G-poten}.

\subsection{Sufficient conditions for the weak radial Harnack-type inequality}\label{ss5.1}

To begin with, we introduce the following \emph{radial Harnack-type inequality}, which is evidently stronger than the weak radial Harnack-type inequality in Definition~\ref{def:weak-radial-Harnack}.

\begin{definition}\label{def:radial-Harnack}
We say that $M$ satisfies a \emph{radial Harnack-type inequality} if for any $x \in M$, there exist positive constants $\delta = \delta(x)$ and $C = C(x)$ such that for all $y, z \in M$ with
\[
\delta < \frac{d(x,z)}{2} \le d(x,y) \le 2 d(x,z),
\]
we have
\begin{equation}\label{eq:Gcomp-Harnack}
C^{-1} \, G^M(x;\,z) \le G^M(x;\,y) \le C \, G^M(x;\,z).
\end{equation}
\end{definition}

\begin{lemma}\label{lem-Gxy}
Suppose that the Riemannian manifold $(M, d, V)$ satisfies $\SG$. Then the following  hold:
\begin{enumerate}
  \item[\rm (i)] For all distinct $x,y\in M$,
\begin{align}\label{eq-Gxy}
G^M(x;\,y)\simeq \int_{d(x,\,y)}^\infty \frac{r\,dr}{V(x,r)}.
\end{align}
\item[\rm (ii)]  For all  $x, x', y\in M$ with $d(x,x')<d(x,y)/2$,
\begin{align}\label{eq-Gcomp-equiv}
   G^M(x;\,y)\simeq G^M(x',y).
\end{align}
\end{enumerate}
In particular, if $(M, d, V)$ is a complete Riemannian  manifold with  nonnegative Ricci curvature, then both \eqref{eq-Gxy} and \eqref{eq-Gcomp-equiv} hold. It is obvious that \eqref{eq-Gcomp-equiv}
implies  \eqref{eq:Gcomp-Harnack} and, hence, \eqref{eq-weak-radialharnack}.
\end{lemma}

\begin{proof}
Under $\SG$, it is known that \eqref{eq-Gxy} holds (see, for example,  \cite[Lemma~2.4]{CaoGrigoryanLiu2021JFA}).
Next, we will show \eqref{eq-Gcomp-equiv}.
By the condition $d(x,x') < d(x,y)/2$ and the triangle inequality, we obtain
\[
d(x,x') < d(x',y) \quad \text{and} \quad
\frac{d(x,y)}{2} < d(x',y) < 2d(x,y).
\]
A change of variable $r = t/2$ yields
\begin{align*}
\int_{d(x,\,y)}^\infty \frac{r\,dr}{V(x,r)}
&\le \int_{d(x',\, y)/2}^\infty \frac{r\,dr}{V(x,r)}
= \frac14 \int_{d(x',\, y)}^\infty \frac{t\,dt}{V(x,t/2)}.
\end{align*}
Since $\SG$ holds, it follows that the volume doubling property $\vd$ is satisfied (see \cite{Grigoryan1991MSb, Saloff-Coste1992IMRN}).
Thus, for any $t \ge d(x',y)$, we have
\[
B(x',t) \subset B(x,\, d(x,x')+t)
\subset B(x,\, d(x',y)+t)
\subset B(x,\, 2t)
\]
and, hence,
\[
V(x',t) \le V(x,2t) \le C_D^2 \, V(x,t/2),
\]
where $C_D$ denotes the volume doubling constant.
Consequently,
\begin{align*}
\int_{d(x,\,y)}^\infty \frac{r\,dr}{V(x,r)}
&\le \frac{C_D^2}{4} \int_{d(x',\, y)}^\infty \frac{t\,dt}{V(x',t)}.
\end{align*}
Similarly, the same distance comparability and volume doubling estimates, applied to the lower integral limit $d(x',\, y)$ and the integrand $\frac{1}{V(x',r)}$, yield the reverse inequality
\begin{align*}
\int_{d(x',\, y)}^\infty \frac{r\,dr}{V(x',r)}
&\le \frac{C_D^2}{4} \int_{d(x,\,y)}^{\infty} \frac{r\,dr}{V(x,r)}.
\end{align*}
Combining these last two estimates with \eqref{eq-Gxy}, we conclude that \eqref{eq-Gcomp-equiv} holds.
\end{proof}

Applying \eqref{eq-ptptMi}, we see that $\SG$ and, hence, the conclusions of Lemma \ref{lem-Gxy}, are all preserved on product manifolds.

\begin{lemma}\label{lem-Green-Product}
Let \( M = M_1 \times M_2 \). Assume that  the heat kernels \(\{p_t^{M_1}\}_{t\in(0,\infty)}\) and \(\{p_t^{M_2}\}_{t\in(0,\infty)}\) satisfy $\SG$. Then, the product heat kernel \(\{p_t^{M}\}_{t\in(0,\infty)}\) also satisfies $\SG$. Consequently, the Green function \(G^M(x;\,y)\) satisfies both \eqref{eq-Gxy} and \eqref{eq-Gcomp-equiv}, and therefore also \eqref{eq:Gcomp-Harnack} and  \eqref{eq-weak-radialharnack}.

\end{lemma}

\begin{proof} Write $x = (x_1, x_2)$ and $y = (y_1, y_2)$, where $x_i, y_i \in M_i$ for $i = 1, 2$.
Since each $\{p_t^{M_i}\}_{t\in(0,\infty)}$ satisfies $\SG$, it follows that each $(M_i, d_i, V_i)$ satisfies $\vd$.
Then, using \eqref{eq-ptptMi} and \eqref{eq1-Vprod}, we obtain
\begin{align*}
  p_t^{M}(x;\,y) & =p_t^{M_1}(x_1,y_1)p_t^{M_2}(x_2,y_2) \notag\\
  &\simeq \frac{1}{\sqrt{\prod_{i=1}^2 V_{i}( x_i,t^{1/2 }) V_{i}( y_i,t^{1/2 })}  }
\exp\left( -c \frac{d_{1}(x_1,y_1)^2}{t}\right)\exp\left( -c \frac{d_2(x_2,y_2)^2}{t}\right)\notag\\
&\simeq  \frac{1}{\sqrt{V( x,t^{1/2 }) V( y,t^{1/2 })}  }
\exp\left( -c \frac{d(x,y)^2}{t}\right).
\end{align*}
Thus, the product heat kernel $\{p_t^{M}\}_{t\in(0,\infty)}$ satisfies $\SG$.
Consequently, applying Lemma \ref{lem-Gxy} yields that \(G^M(x;\,y)\) satisfies  \eqref{eq-Gxy} and \eqref{eq-Gcomp-equiv}.
\end{proof}

\begin{example}\label{eq-M-ends}\rm
Suppose that $(M_1, d_1, V_1)$ is a Riemannian manifold such that the associated heat kernel \(\{p_t^{M_1}\}_{t\in(0,\infty)}\) satisfies $\SG$, that is, for all $x_1,y_1\in M_1$ and $t\in(0,\infty)$,
 \begin{align}\label{eq00-ptM}
  p_t^{M_1}(x_1,\,y_1)\asymp \frac1{V_1(x_1,\,\sqrt t)} \exp\left(-c\frac{d_1^2(x_1,\,y_1)}{t}\right).
\end{align}
Let $N \ge 4$ be an integer. Following Grigor'yan and Saloff-Coste \cite{GrigoryanSaloffCoste2009AIF}, we consider the connected sum
\[
M_2 := \mathcal{R}^1 \mathbin{\#} \mathcal{R}^3,
\]
where $\mathcal{R}^1 := \mathbb{R}_+ \times \mathbb{S}^{N-1}$ and $\mathcal{R}^3 := \mathbb{R}^3 \times \mathbb{S}^{N-3}$.
For the Riemannian product manifold
\( M = M_1 \times M_2 \), we will show that the corresponding Green kernel $G^M$ satisfies the weak radial Harnack-type inequality \eqref{eq-weak-radialharnack}.

Let us recall the estimates for the heat kernel $\{p_t^{M_2}\}_{t\in(0,\infty)}$ on $M_2$ established in \cite[Section~6.4]{GrigoryanSaloffCoste2009AIF}.
Denote by $d_2$ and $V_2$ the geodesic distance and Riemannian volume on $M_2$.
Suppose that $K$ is the central part of $M_2$ and $E_1, E_2$ are the ends of $M_2$ so that $E_i$ is isometric to the complement of a compact set in $\mathcal{R}^i$. Let $E_0\subset M_2$ be a precompact open set  containing $K$. For any point $x_2 \in M_2$, set
\[
\|x_2\| := \sup_{z \in K} d_2(x_2,\, z).
\]
Then we have $\|x_2\| \simeq 1 + d_2(x_2, K)$ for all $x_2 \in M_2$.
By \cite[Section~6.4]{GrigoryanSaloffCoste2009AIF}, we have the following estimates:
\begin{itemize}
  \item[(0)] For any $x_2,y_2\in M_2$ and $t\in(0,1]$,
\begin{align}\label{eq0-ptM1}
  p_t^{M_2}(x_2,\,y_2)\asymp \frac1{V_2(x_2,\,\sqrt t)} \exp\left(-c\frac{d_2^2(x_2,\,y_2)}{t}\right).
\end{align}
  \item[(1)] For any $x_2\in E_0\cup E_1,\,y_2\in E_0\cup E_2$ and $t\in(1,\infty)$,
\begin{align}\label{eq1-ptM1}
  p_t^{M_2}(x_2,\,y_2)\asymp \frac1{t^{3/2}} \left(1+\frac{\|x_2\|}{\|y_2\|}\right) \exp\left(-c\frac{d_1^2(x_2,\,y_2)}{t}\right).
\end{align}
  \item[(2)]For any $x_2,y_2\in E_0\cup E_1$ and $t\in(1,\infty)$,
\begin{align}\label{eq2-ptM1}
  p_t^{M_2}(x_2,\,y_2)\asymp \frac{\|x_2\|\|y_2\|}{\sqrt{t(t+\|x_2\|^2)(t+\|y_2\|^2)\,}} \exp\left(-c\frac{d_1^2(x_2,\,y_2)}{t}\right).
\end{align}
   \item[(3)] For any $x_2,y_2\in E_0\cup E_2$ and $t\in(1,\infty)$,
\begin{align}\label{eq3-ptM1}
  p_t^{M_2}(x_2,\,y_2)\asymp \frac1{t^{3/2}} \exp\left(-c\frac{d_1^2(x_2,\,y_2)}{t}\right).
\end{align}
\end{itemize}

Let $d$ denote the geodesic distance on $M$, defined as in \eqref{eq-dg}.
Fix a point $x= (x_1, x_2)\in M$ with $x_i \in M_i$ for $i\in\{1,2\}$, and fix a large number $r \in (10^{10},\, \infty)$. For $i=1,2$, define the ball $$B_i:=\{z\in M_i:\, d_i(z, x_i)<r\}.$$
For any $y = (y_1, y_2)\in M$ such that $y_1 \notin 2B_1$ and $y_2 \in B_2$, by
\eqref{eq00-ptM} and  \eqref{eq0-ptM1}-\eqref{eq1-ptM1}-\eqref{eq2-ptM1}-\eqref{eq3-ptM1},
we have
\begin{align}\label{eq-GM-radialH}
 G^M(x;\,y)
 &=\int_0^\infty  p_t^{M_1}(x_1,\,y_1)p_t^{M_2}(x_2,\,y_2)\,dt\\
 &\asymp \int_0^\infty D(t,x_2,y_2)  \frac1{V_1(x_1,\,\sqrt t)} \exp\left(-c\frac{d^2(x,\,y)}{t}\right)\,dt, \notag
\end{align}
where
$$
D(t,x_2,y_2)
:=\begin{cases}
   \frac1{V_2(x_2,\,\sqrt t)}\quad  & \mbox{as }\ x_2,\, y_2\in M_2\ \text{ and }\  t\in(0,1]; \\
   \frac1{t^{3/2}} \left(1+\frac{\|x_2\|}{\|y_2\|}\right) \quad  & \mbox{as }\  x_2\in E_0\cup E_1,\, y_2\in E_0\cup E_2\ \text{ and }\ t\in(1,\infty); \\
   \frac1{t^{3/2}} \left(1+\frac{\|y_2\|}{\|x_2\|}\right) \quad  & \mbox{as }\  x_2\in E_0\cup E_2,\, y_2\in E_0\cup E_1\ \text{ and }\ t\in(1,\infty); \\
   \frac{\|x_2\|\,\|y_2\|}{\sqrt{t(t+\|x_2\|^2)(t+\|y_2\|^2)\,}}\quad  & \mbox{as }\  x_2,\, y_2\in E_0\cup E_1\ \text{ and }\ t\in(1,\infty); \\
   \frac1{t^{3/2}}\quad  & \mbox{as } x_2,\, y_2\in E_0\cup E_2\ \text{ and }\ t\in(1,\infty).
  \end{cases}
$$
 For $y_1 \notin 2B_1$ and $y_2 \in B_2$, we have for the point $y_x:=(y_1, x_2)$ we have
$$d(x,\, y_x)=d_1(x_1,\,y_1) \simeq d(x,\,y)$$
and, hence,
$$
\exp\left(-c\frac{d^2(x,\,y)}{t}\right) \asymp \exp\left(-c\frac{d^2(x,\,y_x)}{t}\right).
$$
Note that $x_1, x_2$ and $r$ are fixed. Since $d_2(y_2,x_2)<r$, it follows that
$$\|x_2\| \simeq 1 + d_2(x_2, K) \simeq 1 + d_2(y_2, K)\simeq \|y_2\|.$$
In other words, for the second integral in \eqref{eq-GM-radialH},
by changing $y_2$ to $x_2$, it remains to be true that
\begin{align*}
G^M(x;\,y)&\asymp
\int_0^\infty D(t,x_2,x_2)  \frac1{V_1(x_1,\,\sqrt t)} \exp\left(-c\frac{d^2(x,\,y_x)}{t}\right)\,dt\notag\\
&\asymp G^M(x;\,y_x). \notag
\end{align*}
This proves that  $G^M$ satisfies the weak radial Harnack-type inequality \eqref{eq-weak-radialharnack}.

It can be seen from \eqref{eq-GM-radialH} that $G^M$ does not satisfy the radial Harnack-type inequality \eqref{eq:Gcomp-Harnack}. In fact, there are a number of examples of manifolds constructed by connected sums that satisfy  \eqref{eq-weak-radialharnack}.
\end{example}

\subsection{Proof of (\ref{G=fz-all0}) $\Rightarrow$ (\ref{Gf=fz-all}) $\Rightarrow$  (\ref{Gf=fz-some})  in Theorem \ref{thm-G-poten}} \label{ss5.2}

The implication \eqref{Gf=fz-all} $\Rightarrow$  \eqref{Gf=fz-some} is obvious. So, we only need to show
\eqref{G=fz-all0} $\Rightarrow$ \eqref{Gf=fz-all}.

\begin{proof}[Proof of \eqref{G=fz-all0} $\Rightarrow$ \eqref{Gf=fz-all} in Theorem \ref{thm-G-poten}]
Choose an arbitrary nonzero, nonnegative function $f \in \mathcal{C}_c^\infty(M)$.
Our aim is to show that if \eqref{G=fz-all0} holds, then $\mathcal{G}_{p_1,\, p_2}(f)(x) = \infty$ for all $x \in M$.

Since $f \in \mathcal{C}_c^\infty(M)$, there exists a small geodesic ball $B(o, r)$ with center $o \in M$ and radius $r \in (0,\infty)$ such that $f$ is strictly positive on the closure of $B(o, r)$. Consequently, we have
\[
c_0 := \min_{d(z,\, o)\le r} f(z) \in( 0,\infty).
\]
By \eqref{eq-compa-G},  there exists a positive constant $C=C(o,r)$ such that for all $z\in  B(o, r)$ and $y\notin B(o, 2r)$,
\begin{align*}
 C^{-1}G^M(y;\, o)\le G^M(y;\,z) \le C G^M(y;\, o)
\end{align*}
Thus, for any $y\notin B(o, 2r)$, we have
\begin{align}\label{eq-Gf>}
  G^Mf(y) & \ge \int_{B(o,\, r)} G^M(y;\, z) f(z)\, dV(z)
   \ge c_0  C^{-1} V(B(o, r))G^M(y;\, o).
\end{align}
Similarly, if we take a large $R\in(r,\infty)$  such that $\supp f\subset B(o, R)$, then for any $y\notin B(o, 2R)$,
we have
\begin{align}\label{eq-Gf<}
  G^Mf(y) & \le \int_{B(o,\, R)} G^M(y;\, z) f(z)\, dV(z)\\
  &
   \le C G^M(y;\, o) \int_{B(o,\, R)}f(z)\, dV(z)
   = C G^M(y;\, o)\|f\|_{L^1(M)}.\notag
\end{align}

Write $o = (o_1, o_2)$ and $x = (x_1, x_2)$, where $o_1, x_1 \in M_1$ and $o_2, x_2 \in M_2$.
Consider the geodesic balls
\[
B := B\!\left(o, \, R + d(o, x)\right) \subset M
\]
and, for $i = 1, 2$,
\[
B_i := B_{i}\!\left(o_i, \, 4R + 4d(o, x)\right) \subset M_i.
\]
Let $y = (y_1, y_2) \in M_1 \times M_2$ such that either $y_1 \notin B_1$ or $y_2 \notin B_2$.
Then, it follows from \eqref{eq-dg=dinfty} that $y \notin 4B$.
For such $y$, applying \eqref{eq-Gf>} when $p_2'\ge p_1'$, or \eqref{eq-Gf<} when $p_2'<p_1'$, we  always obtain
\begin{align}\label{eq-Gf-low}
\|G^M (f)(\cdot,\, y_2)\|_{L^{p_1'}(M_1)}^{p_2'-p_1'}
\gs
\|G^M(o;\, \cdot,\,y_2)\|_{L^{p_1'}(M_1)}^{p_2'-p_1'}.
\end{align}
Meanwhile, since $x \in B$ and $y \notin 4B$, we  apply \eqref{eq-compa-G} again and obtain that
\begin{equation}\label{eq2-GxGa}
C^{-1} G^M(o;\, y) \le G^M(x;\, y) \le C G^M(o;\, y)
\end{equation}
holds
for some positive constant $C = C(o, x, R)$.
By \eqref{eq-Gf-low},  \eqref{eq2-GxGa}, and the definition of ${\mathcal G}_{p_1,\, p_2}(f)$ in \eqref{eq-GGf}, we have
\begin{align*}
  &{\mathcal G}_{p_1,\, p_2}(f)(x)\\
  &\quad \ge \int_{M_2\setminus B_2}  \int_{M_1}\|G^M (f)(\cdot,\, y_2)\|_{L^{p_1'}(M_1)}^{p_2'-p_1'} G^M(x;\, y)
  \left[G^M (f)(y)\right]^{p_1'-1}\,dV_{1}(y_1)\,dV_{2}(y_2)\notag\\
  &\qquad + \int_{B_2}  \int_{M_1\setminus B_1}\|G^M (f)(\cdot,\, y_2) {\mathbf 1}_{M_1\setminus B_1}(\cdot)\|_{L^{p_1'}(M_1)}^{p_2'-p_1'} G^M(x;\, y)
  \left[G^M (f)(y)\right]^{p_1'-1}\,dV_{1}(y_1)\,dV_{2}(y_2)\notag\\
  &\quad \gs \int_{M_2\setminus B_2}  \int_{M_1}\|G^M(o;\, \cdot,\,y_2)\|_{L^{p_1'}(M_1)}^{p_2'-p_1'}
  \left[G^M(o;\, y_1,y_2)\right]^{p_1'}\,dV_{1}(y_1)\,dV_{2}(y_2)\notag\\
  &\qquad + \int_{B_2}  \int_{M_1\setminus B_1}\|G^M(o;\, \cdot,\, y_2) {\mathbf 1}_{M_1\setminus B_1}(\cdot)\|_{L^{p_1'}(M_1)}^{p_2'-p_1'} \left[G^M(o;\, y_1,y_2)\right]^{p_1'}\,dV_{1}(y_1)\,dV_{2}(y_2)\notag\\
  &\quad \simeq \int_{M_2\setminus B_2}  \|G^M(o;\, \cdot,\,y_2)\|_{L^{p_1'}(M_1)}^{p_2'}\,dV_{2}(y_2) + \int_{B_2}  \|G^M(o;\, \cdot,\, y_2) {\mathbf 1}_{M_1\setminus B_1}(\cdot)\|_{L^{p_1'}(M_1)}^{p_2'} \,dV_{2}(y_2).\notag
\end{align*}
In view of \eqref{eq-dg=dinfty}, we have \(B_1 \times B_2 \subset \tau B\) for some positive constant \(\tau\) depending only on \(x\), \(o\) and \(R\). Consequently, since
\[
M \setminus (\tau B) \subset M \setminus (B_1 \times B_2) \subset \bigl( M_1 \times (M_2 \setminus B_2) \bigr) \cup \bigl( (M_1 \setminus B_1) \times B_2 \bigr),
\]
 it follows that
\begin{align*}
&\bigl\| G^M(o;\, \cdot) \, \mathbf{1}_{M \setminus (\tau B)}(\cdot) \bigr\|_{L^{p_2'}(L^{p_1'})(M)}^{p_2'}\\
&\quad= \left(\int_{M_2 \setminus B_2} +\int_{B_2}\right)\bigl\| G^M(o;\, \cdot, y_2)\, \mathbf{1}_{M \setminus (\tau B)}(\cdot, y_2)  \bigr\|_{L^{p_1'}(M_1)}^{p_2'} \, dV_{2}(y_2)\\
&\quad\le
\int_{M_2 \setminus B_2} \bigl\| G^M(o;\, \cdot, y_2) \bigr\|_{L^{p_1'}(M_1)}^{p_2'} \, dV_{2}(y_2) \\
&\qquad + \int_{B_2} \bigl\| G^M(o;\, \cdot, y_2) \, \mathbf{1}_{M_1 \setminus B_1}(\cdot) \bigr\|_{L^{p_1'}(M_1)}^{p_2'} \, dV_{2}(y_2).
\end{align*}
Thus, under \eqref{G=fz-all0}, we conclude that
\begin{align*}
\mathcal{G}_{p_1,\, p_2}(f)(x)
\gtrsim \bigl\| G^M(o;\, \cdot) \, \mathbf{1}_{M \setminus (\tau B)}(\cdot) \bigr\|_{L^{p_2'}(L^{p_1'})(M)}^{p_2'} = \infty,
\end{align*}
as desired.
\end{proof}

\subsection{Proof of (\ref{Gf=fz-some}) $\Rightarrow$ (\ref{G=fz-all0}) in Theorem \ref{thm-G-poten}} \label{ss5.3}

\begin{proof}[Proof of (\ref{Gf=fz-some}) $\Rightarrow$ (\ref{G=fz-all0}) in Theorem \ref{thm-G-poten}]
By \eqref{Gf=fz-some}, there exist a point $o \in M$ and a nonnegative function $f \in \mathcal{C}_c^\infty(M)$ such that
\[
\mathcal{G}_{p_1,\, p_2}(f)(o) = \infty.
\]
Let $r_0$ be the large number determined in Definition \ref{def:weak-radial-Harnack}.
Choose $r \in (r_0,\infty)$ sufficiently large so that $\operatorname{supp} f \subset B(o, r)$. Moreover, we may assume without loss of generality that $G^M(o;\, \cdot) \not\equiv \infty$; otherwise \eqref{G=fz-all0} holds trivially and there is nothing to prove.

Write $o = (o_1, o_2)$, where $o_1 \in M_1$ and $o_2 \in M_2$.
Set $B := B(o, 2r)$ and $B_i := B_{i}(o_i, 2r)$ for $i = 1, 2$.
From \eqref{eq-dg=dinfty}, it follows that
\[
B \subset B_1 \times B_2 \subset 2B.
\]
By splitting the double integral in \eqref{eq-GGf} into three parts, we have
\begin{align*}
\infty &= \mathcal{G}_{p_1,\, p_2}(f)(o) \\
&= \sum_{j=1}^3 \iint_{W_j} \bigl\| G^M f(\cdot,\, x_2) \bigr\|_{L^{p_1'}(M_1)}^{p_2'-p_1'} \,
   G^M(o;\, x) \, \bigl[ G^M f(x) \bigr]^{p_1'-1} \, dV_{1}(x_1) \, dV_{2}(x_2) \\
&=: \sum_{j=1}^3 \mathrm{Z}_j,
\end{align*}
where
\[
\begin{cases}
W_1 := M_1 \times (M_2 \setminus B_2); \\[4pt]
W_2 := (M_1 \setminus B_1) \times B_2; \\[4pt]
W_3 := B_1 \times B_2.
\end{cases}
\]
Thus, at least one of the terms \(\mathrm{Z}_j\) with \(j \in \{1,2,3\}\) must satisfy \(\mathrm{Z}_j = \infty\).
Indeed, we will show that for each \(j = 1,2,3\), the following estimate holds:
\begin{equation}\label{eq-aim-Zj}
\mathrm{Z}_j \lesssim 1 +
\bigl\| G^M(o;\, \cdot) \, \mathbf{1}_{M \setminus (B_1 \times B_2)}(\cdot) \bigr\|_{L^{p_2'}(L^{p_1'})(M)}^{p_2'}.
\end{equation}
Since \(B \subset B_1 \times B_2\), it follows from \eqref{eq-aim-Zj} that
\[
\bigl\| G^M(o;\, \cdot) \, \mathbf{1}_{M \setminus B}(\cdot) \bigr\|_{L^{p_2'}(L^{p_1'})(M)} = \infty,
\]
which, together with  \eqref{G=fz-some} $\Rightarrow$ \eqref{G=fz-all} in  Theorem \ref{thm-para-Green}, completes the proof of \eqref{G=fz-all0}.

It remains to prove \eqref{eq-aim-Zj}. For this purpose, we give two useful  estimates for \(G^M f\).
For \(x = (x_1, x_2) \in B_1 \times B_2\), since \(f \in \mathcal{C}_c^\infty(M)\) and \(G^M(o;\, \cdot) \not\equiv \infty\), Lemma~\ref{lem-FG3.4} implies that \(G^M f \in \mathcal{C}^\infty(M)\) and, hence,
\begin{equation}\label{eq-Gf-bddB}
G^M f \in L^\infty(B_1 \times B_2).
\end{equation}
For any \(x \notin B\) and \(y \in \operatorname{supp} f \subset B(o,r)= \frac12 B\), it follows from \eqref{eq-compa-G} that
\[
C^{-1} G^M(x;\, o) \le G^M(x;\, y) \le C G^M(x;\, o)
\]
for some positive constant \(C = C(o, r)\). Therefore,  for any \(x = (x_1, x_2)\) with either \(x_1 \notin B_1\) or \(x_2 \notin B_2\), we obtain
\begin{equation}\label{eq-Gf-upp}
G^M f(x) = \int_{\frac12 B} G^M(x;\, y) \, f(y) \, dV(y)
\simeq \, G^M(x;\, o) \, \|f\|_{L^1(M)},
\end{equation}
where the implicit equivalence constants are independent of $x$.
With the help of \eqref{eq-Gf-bddB} and \eqref{eq-Gf-upp}, we now show that each \(\mathrm{Z}_j\) satisfies \eqref{eq-aim-Zj}.

\medskip

{\it Step 1:\, proving that \(\mathrm{Z}_1\) satisfies \eqref{eq-aim-Zj}.}
For \(x = (x_1, x_2) \in W_1\), we have \(x_2 \notin B_2\), which allows us to apply \eqref{eq-Gf-upp} to  both \(G^M f(\cdot,\, x_2)\) in the norm \(\| G^M f(\cdot,\, x_2)\|_{L^{p_1'}(M_1)}\) and for \(G^M f(x)\). This yields
\begin{align*}
\mathrm{Z}_1
&\simeq  \int_{M_2 \setminus B_2} \int_{M_1}
   \bigl\| G^M(o;\, \cdot,\, x_2) \bigr\|_{L^{p_1'}(M_1)}^{p_2'-p_1'} \,
   \bigl[ G^M(o;\, x) \bigr]^{p_1'}
   \, dV_{1}(x_1) \, dV_{2}(x_2) \\[4pt]
&\simeq \int_{M_2 \setminus B_2}
   \bigl\| G^M(o;\, \cdot,\, x_2) \bigr\|_{L^{p_1'}(M_1)}^{p_2'} \, dV_{2}(x_2). \notag
\end{align*}
Regarding this last integral, note that for any \(x_1 \in M_1\) and \(x_2 \in M_2 \setminus B_2\), we have \((x_1, x_2) \in M \setminus (B_1 \times B_2)\) and, hence,
\[
\bigl\| G^M(o;\, \cdot,\, x_2) \bigr\|_{L^{p_1'}(M_1)}^{p_2'}
= \bigl\| G^M(o;\, \cdot,\, x_2) \, \mathbf{1}_{M \setminus (B_1 \times B_2)}(\cdot,\, x_2) \bigr\|_{L^{p_1'}(M_1)}^{p_2'}.
\]
Consequently, we obtain
\begin{align*}
&\int_{M_2 \setminus B_2}
   \bigl\| G^M(o;\, \cdot,\, x_2) \bigr\|_{L^{p_1'}(M_1)}^{p_2'} \, dV_{2}(x_2)\\
&\quad= \int_{M_2 \setminus B_2}
   \bigl\| G^M(o;\, \cdot,\, x_2) \, \mathbf{1}_{M \setminus (B_1 \times B_2)}(\cdot,\, x_2) \bigr\|_{L^{p_1'}(M_1)}^{p_2'}
   \, dV_{2}(x_2) \\[4pt]
&\quad\le \bigl\| G^M(o;\, \cdot) \, \mathbf{1}_{M \setminus (B_1 \times B_2)}(\cdot) \bigr\|_{L^{p_2'}(L^{p_1'})(M)}^{p_2'}.
\end{align*}
Thus, we have proved that \(\mathrm{Z}_1\) satisfies the desired estimate \eqref{eq-aim-Zj}.

\medskip

{\it Step 2:\, proving that \(\mathrm{Z}_2\) satisfies \eqref{eq-aim-Zj}.}
Again, applying   \eqref{eq-Gf-upp} yields
\begin{align}\label{eq-Z20}
  {\rm Z}_2
  &\simeq \int_{B_2}\int_{M_1\setminus B_1} \bigl\| G^M f(\cdot,\, x_2) \bigr\|_{L^{p_1'}(M_1)}^{p_2'-p_1'} \,
\, \bigl[G^M(o;\, x) \bigr]^{p_1'} \, dV_{1}(x_1) \, dV_{2}(x_2)\\
&\simeq \int_{ B_2}\bigl\| G^M f(\cdot,\, x_2) \bigr\|_{L^{p_1'}(M_1)}^{p_2'-p_1'}
   \bigl\| G^M(o;\, \cdot,\, x_2) {\mathbf 1}_{M_1 \setminus B_1}(\cdot) \bigr\|_{L^{p_1'}(M_1)}^{p_1'} \, dV_{2}(x_2). \notag
\end{align}
Let us examine the norm \(\| G^M f(\cdot,\, x_2) \|_{L^{p_1'}(M_1)}\). For any \(x_2 \in B_2\), we write
\begin{align}\label{eq1-inner}
&\| G^M f(\cdot,\, x_2) \|_{L^{p_1'}(M_1)}^{p_1'}\\
&\quad= \left(\int_{B_1}
   + \int_{M_1 \setminus B_1}\right)\, \bigl[ G^M f(x_1, x_2) \bigr]^{p_1'}\, dV_{1}(x_1)\notag \\
&\quad\simeq \int_{B_1} \bigl[ G^M f(x_1, x_2) \bigr]^{p_1'} \, dV_{1}(x_1)
   + \int_{M_1 \setminus B_1} \bigl[ G^M(o;\, x_1, x_2) \bigr]^{p_1'} \, dV_{1}(x_1), \notag
\end{align}
where for the integral over \(M_1 \setminus B_1\) we again used \eqref{eq-Gf-upp}.
Moreover, since now \(x_2 \in B_2\), by \eqref{eq-Gf-bddB} we have
\begin{equation}\label{eq2-inner}
\int_{B_1} \bigl[ G^M f(x_1, x_2) \bigr]^{p_1'} \, dV_{1}(x_1) \lesssim 1.
\end{equation}
With \eqref{eq1-inner} and \eqref{eq2-inner} at hand, we claim that
\begin{equation}\label{eq-Z2}
\mathrm{Z}_2 \lesssim 1+\int_{B_2}
   \bigl\| G^M(o;\, \cdot,\, x_2) \, \mathbf{1}_{M_1 \setminus B_1}(\cdot) \bigr\|_{L^{p_1'}(M_1)}^{p_2'} \, dV_{2}(x_2).
\end{equation}
To prove \eqref{eq-Z2}, we treat separately the cases \(p_2 \ge p_1\) and \(p_2 < p_1\).

The proof of \eqref{eq-Z2} for the case \(p_2 \ge p_1\) is straightforward. Indeed, when \(p_2 \ge p_1\) we have \(p_2' \le p_1'\). By discarding the first term on the right-hand side of \eqref{eq1-inner}, we obtain
\begin{align*}
\| G^M f(\cdot,\, x_2) \|_{L^{p_1'}(M_1)}^{p_2'-p_1'}
&\lesssim \left( \int_{M_1 \setminus B_1} \bigl[ G^M(o;\, x_1, x_2) \bigr]^{p_1'} \, dV_{1}(x_1) \right)^{\frac{p_2'-p_1'}{p_1'}} \\
&\simeq \bigl\| G^M(o;\, \cdot,\, x_2) \, \mathbf{1}_{M_1 \setminus B_1}(\cdot) \bigr\|_{L^{p_1'}(M_1)}^{p_2'-p_1'},
\end{align*}
which, together with \eqref{eq-Z20}, implies that \eqref{eq-Z2} holds.

Next, we verify that \eqref{eq-Z2} holds when \(p_2 < p_1\). In this case, we have \(p_2' > p_1'\). Then, it follows from \eqref{eq1-inner} and \eqref{eq2-inner} that
\begin{align*}
\| G^M f(\cdot,\, x_2) \|_{L^{p_1'}(M_1)}^{p_2'-p_1'}
&\lesssim 1 + \left( \int_{M_1 \setminus B_1} \bigl[ G^M(o;\, x_1, x_2) \bigr]^{p_1'} \, dV_{1}(x_1) \right)^{\frac{p_2'-p_1'}{p_1'}} \\
&\lesssim 1 + \bigl\| G^M(o;\, \cdot,\, x_2) \, \mathbf{1}_{M_1 \setminus B_1} (\cdot) \bigr\|_{L^{p_1'}(M_1)}^{p_2'-p_1'}.
\end{align*}
Substituting this estimate into \eqref{eq-Z20} yields
\begin{align*}
\mathrm{Z}_2 &\lesssim \int_{B_2}
   \bigl\| G^M(o;\, \cdot,\, x_2) \, \mathbf{1}_{M_1 \setminus B_1}(\cdot) \bigr\|_{L^{p_1'}(M_1)}^{p_1'} \, dV_{2}(x_2) \\
   &\qquad + \int_{B_2}
   \bigl\| G^M(o;\, \cdot,\, x_2) \, \mathbf{1}_{M_1 \setminus B_1} (\cdot)\bigr\|_{L^{p_1'}(M_1)}^{p_2'} \, dV_{2}(x_2).
\end{align*}
Regarding the first integral on the right-hand side, since \(p_2' > p_1'\),   by comparing the norm
\[\| G^M(o;\, \cdot,\, x_2) \, \mathbf{1}_{M_1 \setminus B_1}(\cdot) \|_{L^{p_1'}(M_1)}\] with the value \(1\),
we obtain the elementary estimate
\[
\bigl\| G^M(o;\, \cdot,\, x_2) \, \mathbf{1}_{M_1 \setminus B_1}(\cdot) \bigr\|_{L^{p_1'}(M_1)}^{p_1'}
\le 1 + \bigl\| G^M(o;\, \cdot,\, x_2) \, \mathbf{1}_{M_1 \setminus B_1}(\cdot) \bigr\|_{L^{p_1'}(M_1)}^{p_2'}.
\]
This immediately yields
\begin{equation*}
\mathrm{Z}_2 \lesssim 1 + \int_{B_2}
   \bigl\| G^M(o;\, \cdot,\, x_2) \, \mathbf{1}_{M_1 \setminus B_1}(\cdot) \bigr\|_{L^{p_1'}(M_1)}^{p_2'} \, dV_{2}(x_2).
\end{equation*}
Consequently, \eqref{eq-Z2} remains valid in the case \(p_2 < p_1\).

Having established \eqref{eq-Z2}, we now observe that for any \(x_2 \in B_2\),
\begin{equation}\label{eq-obser-xx}
   \bigl\| G^M(o;\, \cdot,\, x_2) \, \mathbf{1}_{M_1 \setminus B_1}(\cdot) \bigr\|_{L^{p_1'}(M_1)}
   = \bigl\| G^M(o;\, \cdot,\, x_2) \, \mathbf{1}_{(M_1 \setminus B_1) \times B_2}(\cdot,\, x_2) \bigr\|_{L^{p_1'}(M_1)}.
\end{equation}
Combining this observation with \eqref{eq-Z2}, we conclude that \(\mathrm{Z}_2\) satisfies \eqref{eq-aim-Zj}.

\medskip

{\it Step 3:\, proving that \(\mathrm{Z}_3\) satisfies \eqref{eq-aim-Zj}.}
By \eqref{eq-Gf-bddB}, we have
\begin{align}\label{eq-Z3}
{\rm Z}_3
&\ls
\int_{B_2}\int_{B_1} \bigl\| G^M f(\cdot,\, x_2) \bigr\|_{L^{p_1'}(M_1)}^{p_2'-p_1'} \,
   G^M(o;\, x) \, dV_{1}(x_1) \, dV_{2}(x_2).
\end{align}
For any \(x_2 \in B_2\), the norm \(\| G^M f(\cdot,\, x_2) \|_{L^{p_1'}(M_1)}\) above still satisfies \eqref{eq1-inner} and \eqref{eq2-inner}. We now proceed by considering the cases \(p_2 \ge p_1\) and \(p_2 < p_1\).

Let \(p_2 \ge p_1\).
By Lemma~\ref{lem-G-property}, we have \(G^M(o;\, \cdot) \in \mathcal{C}^\infty(M \setminus \{o\})\) and it is positive. Thus, on the precompact open set \((2B_1 \setminus B_1) \times B_2\), which does not contain \(o\), there exists a constant \(c \in (0,1)\) such that for all
\(x_1 \in 2B_1\setminus B_1\) and \(x_2 \in B_2\),
\[
G^M(o;\, x_1, x_2) \ge c.
\]
This, combined with \eqref{eq1-inner}, gives
\begin{align*}
\| G^M f(\cdot,\, x_2) \|_{L^{p_1'}(M_1)}^{p_1'}
\gtrsim \int_{2B_1\setminus B_1} \bigl[ G^M(o;\, x_1, x_2) \bigr]^{p_1'} \, dV_{1}(x_1) \gtrsim 1.
\end{align*}
Since \(p_2' - p_1' \le 0\), we return to \eqref{eq-Z3} and get
\begin{align*}
\mathrm{Z}_3
&\lesssim \int_{B_2} \int_{B_1} G^M(o;\, x) \, dV_{1}(x_1) \, dV_{2}(x_2).
\end{align*}
Choose a nonnegative function \(\phi \in \mathcal{C}_c^\infty(M)\) such that \(\phi \equiv 1\) on \(B_1 \times B_2\).
By Lemma~\ref{lem-FG3.4}, we have \(G^M(\phi) \in \mathcal{C}^\infty(M)\), which further implies
\begin{align*}
\mathrm{Z}_3
&\lesssim \int_{B_2} \int_{B_1} G^M(o;\, x) \, \phi(x) \, dV_{1}(x_1) \, dV_{2}(x_2)
\lesssim G^M(\phi)(o)\lesssim 1.
\end{align*}
This proves that \(\mathrm{Z}_3\) satisfies \eqref{eq-aim-Zj} when $p_2\ge p_1$.

Let now \(p_2 < p_1\), that is, \(p_2' > p_1'\).
In this case, substituting  \eqref{eq1-inner} and \eqref{eq2-inner} (with $2B_1$ now replaced by $B_1$ therein) into \eqref{eq-Z3}, we  obtain
\begin{align*}
\mathrm{Z}_3
&\lesssim \int_{B_2}\int_{B_1} \left(1 + \bigl\| G^M(o;\, \cdot,\, x_2) \, \mathbf{1}_{M_1 \setminus (2B_1)}(\cdot) \bigr\|_{L^{p_1'}(M_1)}\right)^{p_2'-p_1'}
   G^M(o;\, x) \, dV_{1}(x_1) \, dV_{2}(x_2).
\end{align*}
By the weak radial Harnack-type inequality \eqref{eq-weak-radialharnack},
the Green function satisfies
\begin{align*}
G^M(o;\, \cdot,\, x_2) \, \mathbf{1}_{M_1 \setminus (2B_1)}(\cdot)\simeq G^M(o;\, \cdot,\, o_2) \, \mathbf{1}_{M_1 \setminus (2B_1)}(\cdot)\quad \text{for all}\ \, x_2\in B_2.
\end{align*}
This, along with the fact that $(1+t)^{p_2'-p_1'}\le (1+t)^{p_2'}$ for any $t\in(0,\infty)$,
and the local integrability of $G^M(o;\,\cdot)$ on $(2B_1)\times B_2$, further gives
\begin{align*}
\mathrm{Z}_3
&\lesssim \left(1 + \bigl\| G^M(o;\, \cdot,\, o_2) \, \mathbf{1}_{M_1 \setminus (2B_1)}(\cdot) \bigr\|_{L^{p_1'}(M_1)}\right)^{p_2'-p_1'} \int_{B_2}\int_{B_1}
   G^M(o;\, x) \, dV_{1}(x_1) \, dV_{2}(x_2)\\
&\ls \left(1 + \bigl\| G^M(o;\, \cdot,\, o_2) \, \mathbf{1}_{M_1 \setminus (2B_1)}(\cdot) \bigr\|_{L^{p_1'}(M_1)}\right)^{p_2'}\\
&\simeq
\frac1{V_{2}(B_2)}\int_{B_2} \left(1 + \bigl\| G^M(o;\, \cdot,\, x_2) \, \mathbf{1}_{M_1 \setminus (2B_1)}(\cdot) \bigr\|_{L^{p_1'}(M_1)}\right)^{p_2'} \, dV_{2}(x_2)\\
&\ls 1+ \bigl\| G^M(o;\, \cdot) \, \mathbf{1}_{M \setminus (B_1 \times B_2)}(\cdot) \bigr\|_{L^{p_2'}(L^{p_1'})(M)}^{p_2'}.
\end{align*}
Finally, invoking the observation \eqref{eq-obser-xx}, we obtain  that \(\mathrm{Z}_3\) satisfies \eqref{eq-aim-Zj}.

Summarizing all, we complete the proof of (\ref{Gf=fz-some}) $\Rightarrow$ (\ref{G=fz-all0}).
\end{proof}

\section{Applications of Theorem~\ref{thm-main}} \label{s6}

\subsection{Geometric criteria via volume growth}\label{ss6.1}

\begin{proposition}\label{prop-para-vol}
For $i \in\{ 1, 2\}$, let $p_i \in (1, \infty)$ and $(M_i, d_i, V_i)$ be a noncompact geodesically complete Riemannian manifold with $\operatorname{Ric}_{M_i} \ge 0$.
Then  $M := M_1 \times M_2$ is $L^{p_2}(L^{p_1})$-parabolic if and only if
\footnote{Note that $V_1'(r)$ and $V_2'(r)$ exist for almost all $r\in(0,\infty)$.}
\begin{equation}\label{eq1-para-vol}
\int^\infty \left[ \int^\infty \left( \int_{r \vee s}^\infty \frac{t \, dt}{V_1(t) V_2(t)} \right)^{p_1'} V_1'(r) \, dr \right]^{\frac{p_2'}{p_1'}} V_2'(s) \, ds = \infty,
\end{equation}
where \(V_i(t) := V_i(o_i, t)\) for \(i \in \{1, 2\}\) and \(t \in (0, \infty)\), and \(o = (o_1, o_2) \in M\) is some fixed point.
\end{proposition}

\begin{proof}
According to Theorem \ref{thm-main} or Theorem~\ref{thm-para-Green}, $M$ is $L^{p_2}(L^{p_1})$-parabolic  if and only if for some, equivalently all, $R \in (0, \infty)$,
\begin{equation}\label{eq1-G=fz}
  \left\|G^M(o;\ \cdot) {\mathbf 1}_{M\setminus B(o,\,R)}(\cdot)\right\|_{L^{p_2'}(L^{p_1'})(M)}=\infty.
\end{equation}
For $i=1,2$ and $R\in (0,\infty)$, set
$$
   B_R^i := B_{i}(o_i,\, R) \subset M_i.
$$
Using
\eqref{eq-dg=dinfty}, it is easy to verify that
$$\frac 12 B_R^1 \times \frac 12 B_R^2\subset B(o, R) \subset B_R^1\times B_R^2.$$
This, along with \eqref{eq1-G=fz} and the local boundedness of $G^M(o; \,\cdot\,)$ in $M \setminus \{o\}$, implies that $M$ is $L^{p_2}(L^{p_1})$-parabolic if and only if for some/all $R \in (0, \infty)$,
\begin{equation}\label{eq2-G=fz}
\bigl\| G^M(o; \,\cdot\,) \, \mathbf{1}_{M \setminus (B_R^1 \times B_R^2)}(\cdot) \bigr\|_{L^{p_2'}(L^{p_1'})(M)} = \infty.
\end{equation}

For $i = 1, 2$, since $(M_i, d_i, V_i)$ is geodesically complete and satisfies $\operatorname{Ric}_{M_i} \ge 0$,
it follows that the heat kernel $\{p_t^{M_i}\}_{t\in(0,\infty)}$ satisfies $\SG$ and $(M_i, d_i, V_i)$ satisfies $\vd$.
Applying Lemma~\ref{lem-Green-Product} yields that  $G^M(x;\,y)$ satisfies both \eqref{eq-Gxy} and \eqref{eq-Gcomp-equiv}. Consequently, \eqref{eq2-G=fz} holds if and only if
\begin{equation}\label{eq3-G=fz}
  \int_{M_2} \left[ \int_{M_1} \left( \int_{d(x,\, o)}^\infty \frac{t \, dt}{V_1(t) V_2(t)} \right)^{p_1'}\, \mathbf{1}_{M \setminus (B_R^1 \times B_R^2)}(x)\, dV_{1}(x_1) \right]^{\frac{p_2'}{p_1'}} \, dV_{2}(x_2) =\infty.
\end{equation}
Using \eqref{eq-dg=dinfty} and the fact that $(M_i, d_i, V_i)$  satisfies \(\vd\), we may replace the lower limit \(d(x, o)\) of the inner integration by \(\max\{d_1(x_1, o_1),\, d_2(x_2, o_2)\}\), where $x=(x_1,x_2)\in M$ with $x_i\in M_i$.
Further, observing  that
$$
M\setminus (B_R^1 \times B_R^2) = \left((M_1\setminus B_R^1)\times (M_2\setminus B_R^2)\right)
\cup \left(B_R^1 \times (M_2\setminus B_R^2) \right) \cup \left( (M_1\setminus B_R^1)\times B_R^2 \right),
$$
we then obtain that \eqref{eq3-G=fz} holds if and only if either one of the three parts is $\infty$:
$$
\begin{cases}
  {\rm I}(R):= \displaystyle\int_{M_2\setminus B_R^2} \left[ \int_{M_1\setminus B_R^1} \left( \int_{d_{1}(x_1,\, o_1)\vee d_{2}(x_2,\, o_2)}^\infty \frac{t \, dt}{V_1(t) V_2(t)} \right)^{p_1'} dV_{1}(x_1) \right]^{\frac{p_2'}{p_1'}} \, dV_{2}(x_2) =\infty;\\
  {\rm II}(R):= \displaystyle\int_{M_2\setminus B_R^2} \left[ \int_{B_R^1} \left( \int_{d_{2}(x_2,\, o_2)}^\infty \frac{t \, dt}{V_1(t) V_2(t)} \right)^{p_1'} dV_{1}(x_1) \right]^{\frac{p_2'}{p_1'}} \, dV_{2}(x_2) =\infty;\\
  {\rm III}(R):= \displaystyle\int_{B_R^2} \left[ \int_{M_1\setminus B_R^1} \left( \int_{d_{1}(x_1,\, o_1)}^\infty \frac{t \, dt}{V_1(t) V_2(t)} \right)^{p_1'} dV_{1}(x_1) \right]^{\frac{p_2'}{p_1'}} \, dV_{2}(x_2) =\infty.
\end{cases}
$$

Further, by using the coarea formula to get that
$$
\int_{M_i} F(d_i(o_i,\,z))\, dV_i(z)=\int_0^\infty F(r)\, dV_i(o_i, r)
$$
and noting that $dV_i(o_i, r)=V_i'(r)\,dr$ for a. e. $r\in(0,\infty)$,
we have
\begin{align*}
  {\rm I}(R)\simeq \displaystyle\int_R^\infty \left[ \int_R^\infty \left( \int_{r\vee s}^\infty \frac{t \, dt}{V_1(t) V_2(t)} \right)^{p_1'}\, V_1'(r)\,dr \right]^{\frac{p_2'}{p_1'}} \, V_2'(s)\,ds.
\end{align*}
which implies that ${\rm I}(R)=\infty$ if and only if \eqref{eq1-para-vol} holds. Consequently, if \eqref{eq1-para-vol} holds, then ${\rm I}(R) = \infty$ for some $R$, and hence \eqref{eq3-G=fz} holds; we thereby obtain that $M$ is $L^{p_2}(L^{p_1})$-parabolic.

Next, we show that if $M$ is $L^{p_2}(L^{p_1})$-parabolic, i.e., \eqref{eq3-G=fz} holds, then ${\rm I}(R)=\infty$ for all $R\in(0,\infty)$. To this end, assume to the contrary that ${\rm I}(R)<\infty$ for some $R\in(0,\infty)$.
Then, on the one hand, we have
\begin{align*}
{\rm I}(R)&\ge \int_{3R}^\infty \left[ \int_R^{2R} \left( \int_{r\vee s}^\infty \frac{t \, dt}{V_1(t) V_2(t)} \right)^{p_1'}\, V_1'(r)\,dr \right]^{\frac{p_2'}{p_1'}} \, V_2'(s)\,ds\\
&\simeq  \int_{3R}^\infty \left[ \int_R^{2R} \left( \int_{s}^\infty \frac{t \, dt}{V_1(t) V_2(t)} \right)^{p_1'}\, V_1'(r)\,dr \right]^{\frac{p_2'}{p_1'}} \, V_2'(s)\,ds\\
&\simeq  \left[ \int_R^{2R}  V_1'(r)\,dr \right]^{\frac{p_2'}{p_1'}} \int_{3R}^\infty  \left( \int_{s}^\infty \frac{t \, dt}{V_1(t) V_2(t)} \right)^{p_2'} \, V_2'(s)\,ds\\
& \simeq \int_{3R}^\infty  \left( \int_{s}^\infty \frac{t \, dt}{V_1(t) V_2(t)} \right)^{p_2'} \, V_2'(s)\,ds.
\end{align*}
On the other hand,
\begin{align*}
  {\rm II}(R)&\simeq \int_R^\infty \left[ \int_0^R \left( \int_{s}^\infty \frac{t \, dt}{V_1(t) V_2(t)} \right)^{p_1'} \, V_1'(r)\,dr \right]^{\frac{p_2'}{p_1'}} \, V_2'(s)\,ds\\
  &\simeq \left[V_1(R)\right]^{\frac{p_2'}{p_1'}} \int_R^\infty  \left( \int_{s}^\infty \frac{t \, dt}{V_1(t) V_2(t)} \right)^{p_2'} \, V_2'(s)\,ds\\
 & \simeq \int_R^\infty  \left( \int_{s}^\infty \frac{t \, dt}{V_1(t) V_2(t)} \right)^{p_2'} \, V_2'(s)\,ds.
\end{align*}
Thus, ${\rm I}(R) < \infty$ forces ${\rm II}(R) < \infty$. In a similar manner, one can show that ${\rm I}(R) < \infty$ also forces ${\rm III}(R) < \infty$.
Consequently, if \eqref{eq3-G=fz} holds, then we must have ${\rm I}(R) = \infty$ for every $R \in (0, \infty)$.
\end{proof}

\begin{corollary}\label{cor-para-mix}
For $i = 1, 2$, let $(M_i, d_i, V_i)$ be a noncompact geodesically complete Riemannian manifold with $\operatorname{Ric}_{M_i} \ge 0$.
Then, for any $p_1, p_2 \in (1, \infty)$, if either
\[
\int^\infty  \left( \int_{s}^\infty \frac{t \, dt}{V_1(t) V_2(t)} \right)^{p_2'} \, V_2'(s)\,ds = \infty
\]
or
\[
\int^\infty  \left( \int_{r}^\infty \frac{t \, dt}{V_1(t) V_2(t)} \right)^{p_1'} \, V_1'(r)\,dr = \infty,
\]
then the Riemannian product manifold $M := M_1 \times M_2$ is $L^{p_2}(L^{p_1})$-parabolic.
\end{corollary}

\begin{proof}
This can be seen from the proof of Proposition~\ref{prop-para-vol}, since either condition implies \eqref{eq1-para-vol}.
\end{proof}

\begin{proposition}\label{prop-para-mix-Rn}
	Let $p_1,p_2\in(1,\infty)$. For $i\in\{1,2\}$, let $(M_i,d_i,V_i)$ be a connected, geodesically complete Riemannian manifold with $\operatorname{Ric}_{M_i}\ge 0$.
 Assume that  $M=M_1\times M_2$ is noncompact.
Suppose that each factor $M_i$ is either compact, in which case we set $n_i:=0$, or has polynomial volume growth of order $n_i\in(0,\infty)$ in the following radial sense: for a fixed point $o_i\in M_i$,
	\begin{equation}\label{eq-polynomial-growth-Ni}
	V_i(o_i,t)\simeq t^{n_i}
		\quad	\text{and}\quad		V_i'(o_i,t)\simeq t^{n_i-1}
		\quad\text{for a.e.}\ \,t\in (1,\infty),
	\end{equation}
where the implicit constants are independent of $t$.
	Then,
	\begin{equation}\label{eq-generalized-product-threshold}
	M=M_1\times M_2
\ \  \text{is}\ \ L^{p_2}(L^{p_1})\text{-parabolic}
	\quad\Leftrightarrow\quad	\frac{n_1}{p_1}+\frac{n_2}{p_2}\le 2 .
	\end{equation}
	In particular, for $k_1,k_2\in\nn$, if
	$
	M_1=\mathbb R^{n_1}\times {\mathbb S}^{k_1}$ and $M_2=\mathbb R^{n_2}\times {\mathbb S}^{k_2}$,
	then
	\[
	M_1\times M_2
\ \ \text{is}\ \ L^{p_2}(L^{p_1})\text{-parabolic}
	\quad\Leftrightarrow\quad
	\frac{n_1}{p_1}+\frac{n_2}{p_2}\le2 .
	\]
\end{proposition}

\begin{proof}
For $i\in\{1,2\}$, let $q_i:=p_i'$. We first consider the case in which both $M_1$ and $M_2$ are noncompact. By Theorem~\ref{thm-para-Green} and Proposition~\ref{prop-para-vol}, we obtain that $M=M_1\times M_2$ is $L^{p_2}(L^{p_1})$-parabolic if and only if
	\begin{equation}\label{eq-J-N1D2}
		\cj:=\int^\infty \left( \int^\infty \left[ \int_{r\vee s}^\infty
		t^{1-n_1-n_2}\,dt\right]^{q_1} r^{n_1-1}\,dr\right)^{\frac{q_2}{q_1}}s^{n_2-1}\,ds=\infty,
	\end{equation}
	where we have used the polynomial estimates \eqref{eq-polynomial-growth-Ni}.

Set $N:=n_1+n_2$. If $N\le 2$, then $\cj=\infty$, since the inner integral satisfies
$$\int_{r\vee s}^\infty t^{1-n_1-n_2}\,dt=\int_{r\vee s}^\infty t^{1-N}\,dt=\infty.$$
Now assume $N>2$. Then, using
	\[
	\int_{r\vee s}^\infty t^{1-n_1-n_2}\,dt=\int_{r\vee s}^\infty t^{1-N}\,dt\simeq (r\vee s)^{2-N},
	\]
we obtain
	\begin{align*}
		\cj
&\simeq \int^\infty \left( \int^\infty \left( r\vee s\right)^{q_1(2-N)} r^{n_1-1} \,dr\right)^{\frac{q_2}{q_1}}s^{n_2-1}\,ds  \\
		&\simeq \int^\infty \left( \int_s^\infty r^{q_1(2-N)+n_1-1}\,dr
		+\int_1^s s^{q_1(2-N)}r^{n_1-1}\,dr\right)^{\frac{q_2}{q_1}}s^{n_2-1}\,ds  \\
		&\simeq \int^\infty \left(\int_s^\infty r^{q_1(2-N)+n_1-1}\,dr
		+ s^{q_1(2-N)+n_1}\right)^{\frac{q_2}{q_1}}s^{n_2-1}\,ds .
	\end{align*}
If $q_1(2-N)+n_1\ge 0$, then again $\cj=\infty$, since
$$\int_s^\infty r^{q_1(2-N)+n_1-1}\,dr=\infty.$$
	If $q_1(2-N)+n_1<0$, then
	\[
	\cj\simeq \int^\infty s^{q_2(2-N)+n_1\frac{q_2}{q_1}+n_2-1}\,ds,
	\]
	which is finite if and only if
	\[
	q_2(2-N)+\frac{n_1q_2}{q_1}+n_2<0.
	\]
Altogether, we obtain that $\cj<\infty$ if and only if
$$
\begin{cases}
  N=n_1+n_2>2; \\
  q_1(2-N)+n_1<0; \\
  q_2(2-N)+n_1\frac{q_2}{q_1}+n_2<0.
\end{cases}
\quad \Leftrightarrow\quad 2<\frac{n_1}{p_1}+\frac{n_2}{p_2}.
$$
Equivalently, we have
$$
\cj=\infty
\quad \Leftrightarrow\quad \frac{n_1}{p_1}+\frac{n_2}{p_2}\le 2.
$$
This proves the desired equivalence in \eqref{eq-generalized-product-threshold}.

It remains to explain why \eqref{eq-generalized-product-threshold} remains valid if either $M_1$ or $M_2$ is compact. Suppose, for example, that $M_2$ is compact. In this case, the mixed criterion \eqref{eq1-para-vol} from Proposition~\ref{prop-para-vol} reduces to
\[
\int^\infty \left(\int_r^\infty \frac{t\,dt}{V_1(t)}\right)^{q_1} V_1'(r)\,dr=\infty,
\]
that is,
\[
\int^\infty \left(\int_r^\infty t^{1-n_1}\,dt\right)^{q_1} r^{n_1-1}\,dr=\infty,
\]
which is equivalent to
\[
\frac{n_1}{p_1}\le 2.
\]
This is exactly \eqref{eq-generalized-product-threshold}, since $n_2=0$ when $M_2$ is compact. The case in which $M_1$ is compact is analogous, with the roles of $(n_1,p_1)$ and $(n_2,p_2)$ interchanged. Altogether, we complete the proof.
\end{proof}

\subsection{A unified perspective via the effective dimension}\label{ss6.2}

The following  anisotropic effective dimension criterion is an immediate consequence of Proposition~\ref{prop-para-mix-Rn}.

\begin{corollary}\label{cor-effDim}
 For $i=1,2$, let $p_i \in (1, \infty)$ and  $q_i=p_i'$, that is, $\frac1{p_i}+\frac 1{q_i}=1$.
Then, the following three statements are equivalent:
\begin{enumerate}[\rm (i)]

\item $\rr^{n_1}\times\rr^{n_2}$ is $L^{p_2}(L^{p_1})$-parabolic.

  \item  $\rr^{n_1}\times\rr^{n_2}$ admits the $L^{q_2}(L^{q_1})$-Liouville property.

  \item The $(p_1,p_2)$-effective dimension on $\rr^{n_1}\times\rr^{n_2}$, denoted by
  $$D_{\rm eff}:=\frac{n_1}{p_1}+\frac{n_2}{p_2}$$
  satisfies
  $D_{\rm eff}\le 2.$
\end{enumerate}
\end{corollary}

To the best of our knowledge, the anisotropic effective dimension criterion in Corollary~\ref{cor-effDim} has not previously appeared in the literature. It extends the classical isotropic Liouville theorem and identifies \(D_{\mathrm{eff}}\) as a sharp invariant.

As a consistency check, setting $n_1 = n_2 = N$ and $p_1 = p_2 = p$ in Corollary~\ref{cor-effDim} recovers the classical isotropic $L^q$-Liouville criterion:
\[
\mathbb R^{2N} \text{ admits the } L^q\text{-Liouville property} \quad \Leftrightarrow \quad D_{\mathrm{eff}} = \frac{2N}{p} \le 2,
\]
where $1/q + 1/p = 1$.
Note that the value $2$ in $D_{\mathrm{eff}} \le 2$ is invariably the order of the elliptic operator;
what changes is the effective dimension $D_{\mathrm{eff}}$ induced by the mixed norm, which encodes the anisotropic geometry of the underlying product structure.

This effective dimension is not merely a notational shorthand. It arises naturally from the scaling properties of the mixed norm. Indeed, consider on $\rr^{n_1}\times\rr^{n_2}$ the anisotropic dilation
\[
T_\lambda:\ (x,y) \mapsto (\lambda^{\alpha} x, \, \lambda^{\beta} y), \qquad \alpha,\, \beta \in (0, \infty).
\]
Under this dilation, the mixed norm scales as
\[
\|u \circ T_\lambda\|_{L^{p_2}(L^{p_1})(\rr^{n_1}\times\rr^{n_2})}
=
\lambda^{\frac{\alpha n_1}{p_1} + \frac{\beta n_2}{p_2}}
\|u\|_{L^{p_2}(L^{p_1})(\rr^{n_1}\times\rr^{n_2})}.
\]
With the natural choice $\alpha = \beta = 1$, the scaling exponent becomes
\[
\frac{n_1}{p_1} + \frac{n_2}{p_2} = D_{\mathrm{eff}}.
\]
Thus $D_{\mathrm{eff}}$ governs the homogeneous scaling of the mixed norm, just as the total dimension $N$ governs the scaling of the standard isotropic norm $\|\cdot\|_{L^p(\mathbb R^N)}$.

\addcontentsline{toc}{section}{References}

\providecommand{\bysame}{\leavevmode\hbox to3em{\hrulefill}\thinspace}

\medskip

\noindent Liguang Liu

\smallskip

\noindent School of Mathematics,
Renmin University of China,
Beijing 100872, People's Republic of China

\smallskip

\noindent{\it E-mail:} \texttt{liuliguang@ruc.edu.cn}

\medskip

\noindent Yuhua Sun

\smallskip

\noindent School of Mathematical Sciences and LPMC,
Nankai University,
Tianjin 300071, People's Republic of China

\smallskip

\noindent{\it E-mail:} \texttt{sunyuhua@nankai.edu.cn}

\medskip

\noindent Suqing Wu

\smallskip

\noindent School of Science,
		Dalian Maritime University,
		Dalian 116024, People's Republic of China

\smallskip

\noindent{\it E-mail:} \texttt{wusq@dlmu.edu.cn}

\end{document}